\documentclass[11pt]{article}

 \usepackage{amsmath,amssymb,amscd,amsthm,esint}
 
\usepackage{graphics,amsmath,amssymb,amsthm,mathrsfs}

\newtheorem{theorem}{Theorem}[section]
\newtheorem{lemma}[theorem]{Lemma}

\def\xxint#1#2#3{{\setbox0=\hbox{$#1{#2#3}{\int}$}
  \vcenter{\hbox{$#2#3$}}\kern-.5\wd0}}

\def\loc{\text{\rm loc}}
\def\unif{\text{\rm unif}}

\def\varep{\varepsilon}

\newcommand{\average}{-\!\!\!\!\!\!\int}

\numberwithin{equation}{section}
\newtheorem{thm}{Theorem}[section]
\newtheorem{prop}[thm]{Proposition}
\newtheorem{cor}[thm]{Corollary}
\newtheorem{lem}[thm]{Lemma}

\newtheorem{rmk}[thm]{Remark}

\newcommand{\Abs}[1]{\left\lvert#1\right\rvert}
\newcommand{\norm}[1]{\lVert#1\rVert}
\newcommand{\Norm}[1]{\left\lVert#1\right\rVert}
\newcommand{\ag}[1]{\langle#1\rangle}

\newcommand{\cL}{\mathcal{L}}
\newcommand{\R}{\mathbb{R}}
\newcommand{\cA}{\mathcal{A}}
\newcommand{\e}{\varepsilon}
\begin{document}

\title{Approximate Correctors and Convergence Rates\\ in Almost-Periodic Homogenization}

\author{Zhongwei Shen\thanks{Supported in part by NSF grant DMS-1161154.}
\and Jinping Zhuge\thanks{Supported in part by NSF grant DMS-1161154.}}

\date{}

\maketitle

\begin{abstract}

We carry out a comprehensive study of quantitative homogenization of second-order 
elliptic systems with bounded measurable coefficients that are almost-periodic 
in the sense of H. Weyl. We obtain uniform local $L^2$ estimates for the 
approximate correctors in terms of a function that quantifies the almost-periodicity of the coefficient matrix.
We give a condition that implies the existence of (true) correctors. 
These estimates as well as similar estimates for the dual approximate correctors 
yield optimal or near optimal convergence rates in $H^1$ and $L^2$.
The $L^2$-based H\"older and Lipschitz estimates at large scale are also established.

\bigskip



\end{abstract}

\tableofcontents

\section{Introduction}

In this paper we shall be interested in quantitative homogenization of
a family of second-order elliptic operators with rapidly oscillating almost-periodic coefficients,
\begin{equation}\label{cond_Le}
\cL_\e = - \text{div}(A(x/\e) \nabla ) = - \frac{\partial}{\partial x_i} \left\{ a^{\alpha\beta}_{ij} \bigg( \frac{x}{\e}\bigg) \frac{\partial}{\partial x_j}\right\},  \qquad \e > 0
\end{equation}
(the summation convention is used throughout). We assume that the coefficient matrix $A(y) = (a^{\alpha\beta}_{ij}(y))$ with $1\le i,j \le d$ and $1\le  \alpha, \beta \le m$ is real, bounded measurable, and satisfies the ellipticity condition,
\begin{equation}\label{cond_ellipticity}
\mu|\xi|^2 \le a^{\alpha\beta}_{ij}(y) \xi_i^\alpha \xi_j^\beta \le \mu^{-1}|\xi|^2 \quad \text{for a.e. } y\in\R^d \text{ and } \xi = (\xi_i^\alpha) \in \R^{ m \times d},
\end{equation}
where $\mu>0$. We further  assume that $A(y)$ is almost-periodic (a.p.) in sense of H. Weyl, which we denote by $A \in APW^2(\R^d)$. This means that each entry of $A$ may be approximated by a sequence of real trigonometric polynomials with respect to the semi-norm,
\begin{equation}\label{cond_W2}
\norm{F}_{W^2} : = \limsup_{R\to \infty} \sup_{x\in\R^d} \left( \fint_{B(x,R)} |F|^2 \right)^{1/2}.
\end{equation}

The qualitative homogenization theory for elliptic equations and systems with a.p.
coefficients has been known since late 1970's \cite{Kozlov-1979, Papanicolaou-1979}.
Let $u_\varep\in H^1(\Omega; \R^m)$ be the weak solution 
to the Dirichlet problem,
\begin{equation}\label{DP-1}
\mathcal{L}_\e (u_\e)= F \quad \text{ in } \Omega \quad \text{ and } \quad  u_\e =f \quad \text{ on } \partial\Omega, 
\end{equation}
where $F\in H^{-1}(\Omega; \R^m)$, $f\in H^{1/2}(\partial\Omega; \R^m)$, and
$\Omega$ is a bounded Lipschitz domain in $\R^d$.
Suppose that $A(y)$ satisfies the ellipticity condition (\ref{cond_ellipticity}) and
is a.p. in the sense of Besicovich (a larger class than $APW^2(\R^d))$.
Then $u_\e$ converges weakly in $H^1(\Omega; \R^m)$  and thus strongly in $L^2(\Omega; \R^m)$ to 
a function $u_0\in H^1(\Omega; \R^m)$.
Moreover, $u_0$ is the weak solution to the (homogenized) Dirichlet problem,
\begin{equation}\label{DP-H-1}
\mathcal{L}_0 (u_0)= F \quad \text{ in } \Omega \quad \text{ and } \quad  u_0 =f \quad \text{ on } \partial\Omega, 
\end{equation}
where $\mathcal{L}_0$ is a second-order elliptic operator with constant coefficients
that depend only on $A$.
Our primary interest in this paper is in
 the convergence rates for $\| u_\varep -u_0\|_{L^2(\Omega)}$.
 
 In the case that $A$ is uniformly a.p. (almost-periodic in the sense of H. Bohr),
 the problem of convergence rates and uniform H\"older estimates
  for the Dirichlet problem (\ref{DP-1}) were studied recently 
 by the first author in \cite{Shen-2015} (also see earlier work \cite{Kozlov-1979, Dungey-2001,Bondarenko-2005}
 as well as \cite{Ishii-2000, Lions-2005, Caffarelli-2010} for homogenization of
 nonlinear differential equations in the a.p. setting).
 The results in \cite{Shen-2015} were subsequently used
 by  S.N. Armstrong  and the first author in \cite{Armstrong-Shen-2016}
 to establish the uniform Lipschitz estimates, up to the boundary, for 
 solutions of $\mathcal{L}_\e (u_\varep)=F$ with either Dirichlet or Neumann conditions.
 In particular, it follows from \cite{Armstrong-Shen-2016} that 
 the so-called approximate correctors $\chi_T$, defined in (\ref{c}) below, satisfy
 the uniform Lipschitz estimate $\|\nabla \chi_T\|_\infty\le C$,
 if $A$ is H\"older continuous  and satisfies an almost-periodicity condition.
 Under some additional assumptions, the uniform boundedness of $\chi_T$,
 $\|\chi_T\|_\infty\le C$ and the existence of (true) correctors  were obtained recently in \cite{AGK-2015}.
 
 In this paper we carry out a comprehensive study of quantitative homogenization of 
 second-order elliptic systems with coefficients in $APW^2(\R^d)$.
 Our results improve and extend those in \cite{Shen-2015, AGK-2015} to a much broader class of a.p.
 functions, which allows bounded measurable coefficients (for comparison, uniformly a.p.
 functions are uniformly continuous in $\R^d$).
 Notice that the semi-norm $\| F\|_{W^2}$ in (\ref{cond_W2})
 is translation and dilation invariant.
 As such the class of  coefficients $A\in APW^2(\R^d)$ seems to be a natural choice for 
 studying quantitative properties in a.p. homogenization
 without smoothness assumptions. 
 
  As in the case of uniformly a.p. (or random) coefficients, to obtain the convergence rates,
 the key step is to establish  estimates for the approximate correctors
 $\chi_T$, defined by the elliptic system
 \begin{equation}\label{c}
 -\text{\rm div} \big(A\nabla \chi_T \big) 
 + T^{-2}\chi_T =\text{\rm div} \big(A\nabla P \big),
 \end{equation}
 where  $T\ge 1$ and $P$ is an affine  function.
 To quantify the almost-periodicity of the coefficient matrix $A$,
 we introduce a function $\rho_k (L,R)$, defined by (\ref{rho-k}) in Section 2.
 It is known that  a bounded function $A$ is a.p. in the sense of
H. Weyl if and only if $\rho_1(L, R)\to 0$ as $L, R\to \infty$ (see Section 2).
 We remark that the function $\rho_1$, which only involves the first-order difference,
 was used in \cite{Shen-2015}.
 Our definition of the higher-order version $\rho_k(L, R)$, as well as one of main steps 
 in the proof of Theorems \ref{main-theorem-1} and \ref{main-theorem-2}, is inspired 
 by \cite{AGK-2015}, where a similar function was used to give a sufficient condition 
 for the existence of (true) correctors. 
 
 The following is one of main results of the paper.

\begin{thm}\label{main-theorem-1}
	Suppose that $A\in APW^2(\R^d)$ and satisfies the ellipticity condition (\ref{cond_ellipticity}). 
	Fix $k\ge 1$ and $\sigma\in (0,1)$.
	Then there exists a constant $c>0$, depending only on $d$ and $k$, such that
	 for any $T\ge 2$, 
	\begin{equation}\label{main-estimate-0}
	\|\nabla \chi_T\|_{S^2_1} \le C_\sigma T^\sigma,
	\end{equation}
	and
	\begin{equation}\label{main-estimate-1}
	\|\chi_T\|_{S^2_1}
	\le C_\sigma \int_1^T \inf_{1\le L\le t}
	\left\{ \rho_k (L, t) +\exp\left(-\frac{c\, t^2}{L^2} \right) \right\}
	\left(\frac{T}{t}\right)^\sigma dt,
	\end{equation}
	where  $C_\sigma$ depends only on $\sigma$, $k$ and $A$.
	 	\end{thm}

In the theorem above we have used the notation
\begin{equation}\label{S-norm}
\| F\|_{S^p_R}:=
\sup_{x\in \R^d} 
\left(\fint_{B(x,R)} |F|^p\right)^{1/p}
\end{equation}
for $1\le p<\infty$ and $0<R<\infty$.
Since $\rho_k (L, R)\to \infty$ as $L, R\to \infty$,
it follows from (\ref{main-estimate-1}) that
$T^{-1} \|\chi_T\|_{S^2_1} \to 0$, as $T\to \infty$.
In particular, if there exist some $k\ge 1$ and $\alpha\in (0,1]$ such that
\begin{equation}\label{decay-0}
\rho_k (L, L)\le {C}{L}^{-\alpha} \quad \text{ for any }L\ge 1,
\end{equation}
then $\|\chi_T\|_{S^2_1} \le C_\beta\, T^\beta$ for any $\beta>1-\alpha$.
Our next two theorems provide sufficient conditions for the existence of true correctors 
in $APW^2(\R^d)$ and for the boundedness of $\|\nabla \chi_T\|_{S^2_1}$,
respectively.

\begin{thm}\label{main-theorem-2}
Suppose $A$ satisfies the same conditions as in Theorem \ref{main-theorem-1}.
Also assume that there exist some $k\ge 1$ and $\alpha>1$ such that (\ref{decay-0}) holds.
Then $\|\chi_T\|_{S^2_1}\le C$.
Moreover,  for each affine function $P$,
the system for the (true) corrector 
$$
-\text{\rm div}\big(A\nabla \chi)=\text{\rm div} \big(A\nabla P) \quad \text{ in } \R^d
$$
has a weak solution $\chi$ such that  $\chi, \nabla \chi\in APW^2(\R^d)$.
\end{thm}

\begin{thm}\label{main-theorem-Lip}
Suppose $A$ satisfies the same conditions as in Theorem \ref{main-theorem-1}.
Also assume that there exist some $k\ge 1$ and $\alpha>3$ such that
\begin{equation}\label{decay-Lip}
\rho_k (L, L)\le C \big\{ \log (L)\big\}^{-\alpha} \quad \text{ for any } L\ge 2.
\end{equation}
Then for any $T\ge 1$,
\begin{equation}\label{Lip}
\|\nabla \chi_T\|_{S^2_1} \le C,
\end{equation}
where $C$ is independent of $T$.
\end{thm}

Using the estimates in Theorems \ref{main-theorem-1} and \ref{main-theorem-2}
 for  the approximate correctors $\chi_T$
as well as similar estimates for the dual approximate correctors, 
we are able to establish a convergence rate in $L^2(\Omega;\R^m)$ under the condition 
that $u_0\in H^2(\Omega; \R^m)$.
In the following theorem, the function $\psi$, defined in (\ref{cond_Vpot}), is the limit of
$\nabla \chi_T$ in $B^2(\R^d)$,
while $\psi^*$ and $ \chi_T^*$ are the corresponding functions for the adjoint operator  $\mathcal{L}^*_\e$.
Also, we use $\Theta_{k, \sigma}(T)$ to denote the integral in the r.h.s of (\ref{main-estimate-1}).

\begin{thm}\label{main-theorem-3}
	Suppose  $A$ satisfes the same conditions as in Theorem \ref{main-theorem-1}.
	Let $\Omega$ be a bounded $C^{1,1}$ domain in $\R^d$. 
	Let $u_\e$, $u_0$ be weak solutions of (\ref{DP-1}) and (\ref{DP-H-1}), respectively, in $\Omega$.
	Assume further that $u_0\in H^2(\Omega; \R^m)$.
	 Then, for any $0<\e<1$,
	\begin{equation}\label{ineq_L2}
	\Norm{u_\e - u_0}_{L^2(\Omega)} \le C_\sigma \Big\{ \norm{\nabla \chi_T - \psi}_{B^2} +
	 \norm{\nabla \chi^*_T - \psi^*}_{B^2} +  T^{-1} \Theta_{k, \sigma} (T)  \Big\} \norm{u_0}_{H^2(\Omega)},
	\end{equation}
	where $T = \e^{-1}$.
	 The constant $C_\sigma $ depends only on $\sigma$, $k$, $A$ and $\Omega$.
	 Furthermore, if (\ref{decay-0}) holds for some $\alpha>1$ and $k\ge 1$, then
	 \begin{equation}\label{optimal}
	 \| u_\e -u_0 \|_{L^2(\Omega)} \le C\, \e \, \| u_0\|_{H^2(\Omega)}.
	 \end{equation}
\end{thm}

We now describe the outline of this paper and some of key ideas used in the proof of Theorems
\ref{main-theorem-1}-\ref{main-theorem-3}.
In Section 2 we provide a brief review of the qualitative homogenization theory of
second-order elliptic systems with coefficients that are a.p. in the sense of Besicovitch.
In Sections 3 and 4 we introduce the approximate correctors $\chi_T$
and establish some preliminary estimates for $\chi_T$.
Using  Tartar's method of test functions as well as the estimates
of $\chi_T$ obtained in Section 4,  we prove a compactness theorem in Section 5 on a sequence of 
elliptic operators $\mathcal{L}_{\e_\ell}^\ell  +\lambda_\ell =-\text{\rm div} \big( A^\ell (x/\e_\ell)\nabla \big)
+\lambda_\ell$,
where each $A^\ell (y)$ is obtained from $A(y)$ through a translation.
With this compactness theorem at our  disposal, an $L^2$-based H\"older estimates at large scale 
for solutions of $\mathcal{L}_\e (u_\e) +\lambda u_\e=F+\text{\rm div} (f)$ 
are obtained in Section 6. This is done by using a compactness argument,
introduced to the study of homogenization problems by Avellaneda and Lin \cite{AL-1987}.
As a corollary of the H\"older estimates at large scale, we obtain  the estimate
(\ref{main-estimate-0}) as well as
a Liouville property for solutions of $\mathcal{L}_1 (u)=0$ in $\R^d$.

In Section 7 we establish some general estimates for functions $g$ in $APW^2(\R^d)$.
These estimates, which formalize and extend a quantitate ergodic argument in \cite{AGK-2015},
allow us to control the norm $\| g\|_{S^2_1}$
by $\|\nabla g\|_{S^2_t}$ for $t\ge 1$ and the higher-order differences of $\nabla g$
(see Theorem \ref{general-theorem}).
The estimate (\ref{main-estimate-1}) in Theorem \ref{main-theorem-1}  as well 
as Theorem \ref{main-theorem-2} is proved in Section 8
by combining  estimates in Section 7 and the large-scale H\"older estimates in Section 6.
In particular, the existence of correctors in $APW^2(\R^d)$ under the condition (\ref{decay-0})
	  for some $\alpha>1$ is obtained by showing that
$$
	  \|\chi_T-\chi_{2T}\|_{S^2_1}\le C\, T^{-\beta},
	  $$
	  for all $T\ge 1$ and some $\beta>0$.
In Section 9 we introduce the dual approximate correctors $\phi_T$, defined by the
elliptic system
\begin{equation}\label{dual}
-\Delta \phi_T + T^{-2} \phi_T =b_T -\langle b_T\rangle,
\end{equation}
where $b_T =A +A\nabla \chi_T-\langle A\rangle$.
Estimates for $\phi_T$ and their derivatives are obtained by using a line of argument similar to
that for $\chi_T$.
In particular, we show that
\begin{equation}\label{c-estimate-1}
 T^{-1}
\|\phi_T\|_{S^2_1}
+ \|\nabla \phi_T\|_{S^2_1}\le C_\sigma \, \Theta_{k, \sigma} (T),
\end{equation}
for any $\sigma \in (0,1)$ and $T\ge 2$.
Using the estimates for $\chi_T$ and $\phi_T$,
we give the proof of Theorem \ref{main-theorem-3} in Section 10.
To do this we adapt  a line of argument for establishing sharp $L^2$ convergence rates in periodic 
homogenization in \cite{Shen-boundary-2015, SZ-2015}, which was motivated by 
the approach used in \cite{Suslina-2012}.
The idea is to first establish the following error estimate in $H^1$,
\begin{equation}
	\begin{aligned}\label{ineq_H1}
	&\Norm{ u_\e - u_0 -\e \chi_{T}(x/\e) K_{\e, \delta} 
	\big( \nabla u_0 \big) }_{H^1(\Omega)} \\
	&\qquad\qquad
	 \le C \Big\{ \norm{\nabla\chi_T - \psi}_{B^2} + T^{-1} \Theta_{k, \sigma} (T)
	  \Big\}^{1/2} \norm{u_0}_{H^2(\Omega)},
	\end{aligned}
	\end{equation}
	where $T=\e^{-1}$ and $K_{\e, \delta}$ is a smoothing operator defined
	by $K_{\e, \delta} (f)=\xi_\e * (\eta_\delta f)$.
	The desired  estimate for $\| u_\varep -u_0\|_{L^2(\Omega)}$
	follows from (\ref{ineq_H1}) by a duality argument.
	In  Section 10 we formalize this approach in the a.p. setting
	so that further improvement on the estimates of approximate and dual approximate correctors
	automatically leads to improvement on the rate of convergence in $L^2$.
	
	Finally, Theorem \ref{main-theorem-Lip} is proved in Section 11.
	To do this, we first establish an $L^2$-based large-scale Lipschitz estimate  for 
	$\mathcal{L}_\e (u_\e)=F$ under the  condition
	(\ref{decay-Lip}) for some $k\ge 1$ and $\alpha>3$.
	 The proof, which uses the convergence rates in Theorem
	 \ref{main-theorem-3},
	  is based on an approach developed in \cite{Armstrong-Smart-2014} and further improved in \cite{Armstrong-Shen-2016,
	  Armstrong-Mourrat, Shen-boundary-2015}.
	  Estimate (\ref{Lip}) follows readily from the $L^2$-based Lipschitz estimate.
	 
	Throughout this paper we will use $\fint_E f =\frac{1}{|E|}\int_E f$ to denote the $L^1$ average of a function $f$
	over a set $E$, and $C$ to denote constants that  depend at most on $A$, $\Omega$ and other relevant 
	parameters, but never on $\e$ or $T$.

%%%%%%%%%%%%%%%%%%%%%%%%%%%%%%%%%%%%%

%%%%%%%%%%%%%%%%%%%%%%%%%%%%%%%%%%%%%%%

\section{Almost-periodic homogenization}

In this section we give a brief review of the qualitative homogenization theory for elliptic systems with
a.p. coefficients. A detailed presentation may be found in \cite{Jikov-1994}.

Let $\text{Trig}(\R^d)$ denote the set of real trigonometric polynomials in $\R^d$.
 A function $f$ in $L^2_{\text{loc}}(\R^d) $ is said to belong to $B^2(\R^d)$ 
 if $f$ is the limit of a sequence of functions in $\text{Trig}(\R^d)$ with respect to the semi-norm
 
\begin{equation}\label{cond_B2}
\norm{f}_{B^2} : = \limsup_{R\to\infty} \left( \fint_{B(0,R)} |f|^2 \right)^{1/2}.
\end{equation}
Functions in $B^2(\R^d)$ are said to be a.p. in the sense of Besicovitch. 
It is not hard to see that if $g\in L^\infty(\R^d) \cap B^2(\R^d)$ and $f\in B^2(\R^d)$, then $fg\in B^2(\R^d)$.
 
Let $f\in L^1_{\text{loc}}(\R^d)$. A number $\ag{f}$ is called the mean value of $f$ if

\begin{equation}
\lim_{\e\to 0^+} \int_{\R^d} f(x/\e) \varphi(x) dx = \ag{f} \int_{\R^d} \varphi
\end{equation}
for any $\varphi \in C_0^\infty(\R^d)$. It is known that if $f,g\in B^2(\R^d)$, then $fg$ has a mean value.
 Under the equivalent relation that $f\sim g$ if $\norm{f-g}_{B^2} = 0$, the set $B^2(\R^d)$
  becomes a Hilbert space with the inner product defined by $(f,g) = \ag{fg}$. Furthermore,
  if $B^2(\mathbb{R}^d; \mathbb{R}^k)
  =B^2(\mathbb{R}^d)\times \cdots \times B^2(\mathbb{R}^d)$, then
  the following Weyl's orthogonal decomposition
\begin{equation}
B^2(\R^d;\R^{m\times d}) = V^2_{\text{pot}} \oplus V^2_{\text{sol}} \oplus \R^{m\times d}
\end{equation}
holds,
where $V^2_{\text{pot}}$ (resp., $V^2_{\text{sol}}$) denotes the closure of potential (resp., solenoidal) trigonometric polynomials with mean value zero in $B^2(\R^d;\R^{m\times d})$. 

Suppose that $A=(a_{ij}^{\alpha\beta})$ satisfies the ellipticity condition (\ref{cond_ellipticity}) and
$A\in B^2(\R^d)$, i.e., each entry $a_{ij}^{\alpha\beta}\in B^2(\R^d)$.
For each $1\le j\le d$ and $1\le \beta \le m$, let $\psi_j^\beta = (\psi_{ij}^{\alpha\beta})$ 
be the unique function in $V^2_{\text{pot}}$ such that

\begin{equation}\label{cond_Vpot}
(a_{ik}^{\alpha\gamma} \psi_{kj}^{\gamma\beta}, \phi_i^\alpha) 
= - (a_{ij}^{\alpha\beta}, \phi_i^\alpha) \qquad \text{for any } \phi = (\phi_i^\alpha) \in V^2_{\text{pot}}.
\end{equation}
Let $\widehat{A} = (\widehat{a}_{ij}^{\alpha\beta})$ be the homogenized matrix of $A$, where
\begin{equation}\label{cond_Ahat}
\widehat{a}_{ij}^{\alpha\beta} = \ag{a_{ij}^{\alpha\beta}} + \ag{a_{ik}^{\alpha\gamma} \psi_{kj}^{\gamma\beta}}.
\end{equation}
Using

\begin{equation}
\widehat{a}_{ij}^{\alpha\beta} = (a_{lk}^{t\gamma} (\psi_{kj}^{\gamma\beta} 
+ \delta_{kj} \delta^{\gamma\beta}), \psi_{li}^{t\alpha} + \delta_{li} \delta^{t\alpha}),
\end{equation}
where we have used $\delta_{ij}$ and $\delta^{\alpha\beta}$ for Kronecker's delta, it can be proved that
\begin{equation}\label{ellipticity-1}
\mu|\xi|^2 \le \widehat{a}_{ij}^{\alpha\beta} \xi_i^\alpha \xi_j^\beta \le \mu_1 |\xi|^2
\end{equation}
for any $\xi = (\xi_i^\alpha) \in \R^{m\times d}$, where $\mu_1$ depends only on $d,m$ and $\mu$. 
Moreover, $\widehat{A^*} = \big(\widehat{A}\big)^*$, where $A^*$ denotes the adjoint of $A$.

The next theorem, whose proof may be found in \cite{Jikov-1994}, 
shows that the homogenized operator for $\cL_\e$ is given by $\cL_0 = -\text{div} (\widehat{A} \nabla )$.

\begin{thm}\label{homo-theorem}
Suppose that $A$ is real, bounded measurable, and satisfies (\ref{cond_ellipticity}). Also assume that  $A\in B^2(\R^d)$. 
Let $\Omega$ be a bounded Lipschitz domain in $\R^d$ and $F\in H^{-1}(\Omega;\R^m)$. 
Let $u_{\e_\ell} \in H^1(\Omega;\R^m)$ be a weak solution of $\cL_{\e_\ell} (u_{\e_\ell}) = F$ in $\Omega$,
where $\e_\ell \to 0$. 
Suppose that $u_{\e_\ell} \rightharpoonup u_0$ weakly in $H^1(\Omega;\R^m)$.
Then $A(x/\e_\ell) \nabla u_{\e_\ell}  \rightharpoonup \widehat{A} \nabla u_0$ weakly in $L^2(\Omega;\R^{m\times d})$. 
Consequently, if $f\in H^{1/2}(\partial \Omega;\R^m)$ and $u_\e$ is the weak solution to the Dirichlet problem:
\begin{equation}
\cL_\e(u_\e) = F \quad \text{ in } \Omega \quad \text{and} \quad u_\e=f \quad \text{on } \partial\Omega,
\end{equation}	
	then, as $\e\to 0$, $u_\e  \rightharpoonup u_0$ weakly in $H^1 (\Omega;\R^m)$, where $u_0$ is the weak solution to
\begin{equation}
\cL_0(u_0) = F \quad \text{ in } \Omega \quad \text{and} \quad u_0=f \quad \text{on } \partial\Omega.
\end{equation}  
\end{thm}

Let $P_j^\beta(x) = x_j e^\beta$, where $e^\beta = (0,\cdots,1,\cdots,0)$
 with $1$ in the $\beta^{\text{th}}$ position. In the periodic case the homogenized coefficients 
 $\widehat{a}_{ij}^{\alpha\beta}$ are obtained by solving the elliptic system
\begin{equation}\label{cond_homo_chi}
-\text{div}(A(x) \nabla u) = \text{div}(A(x) \nabla P_j^\beta),
\end{equation}
with periodic boundary conditions in a periodic cell $Y$ and then extending the solutions to $\mathbb{R}^d$ 
by periodicity.
 Let $\chi_j^\beta = (\chi_j^{1\beta}, \cdots, \chi_j^{m\beta})$ denote the periodic solution 
 of (\ref{cond_homo_chi}) in $\R^d$ such that $\int_Y \chi_j^\beta = 0$. 
 Such solutions are called the correctors for $\cL_\e$. Then
\begin{equation}
\widehat{a}_{ij}^{\alpha\beta} = \fint_Y \left\{ a_{ij}^{\alpha\beta} + a_{ik}^{\alpha\gamma} \frac{\partial \chi_j^{\gamma\beta}}{\partial y_k}\right\}.
\end{equation}
In the a.p. case, solutions of the elliptic system (\ref{cond_homo_chi}) in $\R^d$ in general do not exist
in the class of a.p. functions.
  Although the existence of correctors is not needed for qualitative homogenization 
  ($\psi_{kj}^{\gamma\beta}$ plays the role of $\frac{\partial \chi_j^{\gamma\beta}}{\partial y_k}$ in the definition of the homogenized coefficients),
 the study of the so-called approximate correctors is fundamental in understanding the quantitative properties in homogenization of $\cL_\e$.
 As indicated in the Introduction, in this paper
 we will carry out a systematic study of the approximate correctors 
 $\chi_T$ under the assumption that $A$ is a.p.
 in the sense of H. Weyl.
  
We end this section with a few definitions and observations that will be useful to us.
 Let $1\le p<\infty$.
 We say $F\in L^p_{\loc, \unif} (\R^d)$ if $ F\in L^p_{\loc}(\R^d)$ and 
 \begin{equation}\label{unif}
 \sup_{x\in \R^d} \int_{B(x, 1)} |F|^p<\infty.
 \end{equation}
 For $F\in L^p_{\loc, \unif}(\R^d)$ and $R>0$, we define  the  norm,
 \begin{equation}\label{S-p}
 \| F\|_{S_R^p} : =\sup_{x\in \R^d} \left(\fint_{B(x, R)} |F|^p\right)^{1/p}.
 \end{equation}
 Note that if $0<r<R<\infty$, then
 \begin{equation}\label{r-R}
 \| F\|_{S_R^p} \le C\, \| F\|_{S_r^p},
 \end{equation}
 where $C$ depends only on $d$ and $p$.
Let
 \begin{equation}\label{W-p}
 \| F\|_{W^p}:=\limsup_{R\to \infty} \| F\|_{S^p_R}.
 \end{equation}
 It follows from (\ref{r-R}) that $\| F\|_{W^p} \le C_p\, \| F\|_{S_R^p}$ for any $R>0$.
A function $F$ is said to be $W^p$ (resp., $S^p_R$) a.p. if  $F\in L^p_{\loc, \unif}(\R^d)$ and
 there exists a sequence  of trigonometric polynomials $\{ t_n\}$ such that
 $\| F-t_n\|_{W^p} \to 0$ (resp., $\| F-t_n\|_{S^p_R}\to 0$), as $n\to \infty$.

For $y, z\in \R^d$, define the difference operator
\begin{equation}\label{Diff}
\Delta_{yz}  g (x):=  g(x+y)- g(x+z) .
\end{equation}
 
 \begin{prop}\label{C-W-theorem}
 Let $g\in L^p_{\loc, \unif}(\R^d)$ for some $1\le p<\infty$.
 Then $g$ is $W^p$ a.p. if and only if 
 \begin{equation}\label{char}
 \sup_{y\in \R^d} \inf_{|z|\le L} \| \Delta_{yz} (g)\|_{S^p_R} \to 0 \quad \text{ as } L, R\to \infty.
 \end{equation}
 \end{prop}
 
 \begin{proof}
 A subset $E$ of $\R^d$ is said to be relatively dense in $\R^d$ if
 there exists $L>0$ such that $\R^d =E + B(0, L)$; i.e.,
 any $x$ in $\R^d$ can be written as $y+z$ for some $y\in E$ and $z\in B(0, L)$.
 It is known  that 
 $g\in L^p_{\loc, \unif}$ is $W^p$ a.p.  if and only if for any $\varepsilon>0$,
 there exists $R=R_\e>0$ such that
 the set
 $$
 \Big\{ \tau \in \R^d: \ \| g(\cdot +\tau) - g(\cdot) \|_{S_{R}^p} < \e \Big\}
 $$
 is relatively dense in $\R^d$ \cite{B}.
 It is not hard to see that this is equivalent to (\ref{char}).
 \end{proof}
 
Let $P=P_k =\big\{ (y_1, z_1), \dots, (y_k, z_k) \big\}$,
where $(y_i, z_i)\in \R^d\times \R^d$, and
 $$
 Q=\big\{ (y_{i_1}, z_{i_1}), \dots, (y_{i_\ell}, z_{i_\ell}) \big\}
 $$
  be a subset of $P$ with 
$i_1<i_2<\dots <i_\ell$. 
Define
$$
\Delta_Q (g)=\Delta_{y_{i_1} z_{i_1}} \cdots \Delta_{y_{i_\ell} z_{i_\ell}} (g).
$$
To quantify the almost periodicity of the coefficient matrix $A$,
we introduce 
\begin{equation}\label{rho-k}
\rho_{k} ( L, R)
=\sup_{y_1\in \R^d}\inf_{|z_1|\le L} \cdots\sup_{y_k\in \R^d} \inf_{|z_k|\le L}
\sum
\|\Delta_{Q_1} (A)\|_{S^p_R} 
\cdots \|\Delta_{Q_\ell}  (A)\|_{S^p_R},
\end{equation}
where the sum is taken over all partitions of $P=Q_1\cup Q_2 \cup \cdots \cup Q_\ell$ with $1\le \ell\le k$.
The exponent $p$ in  (\ref{rho-k}) depends on $k$ and is given by
\begin{equation}\label{p}
\frac{k}{p} =\frac{1}{2}-\frac{1}{\bar{q}},
\end{equation}
where $\bar{q}>2$ is the exponent in the reverse H\"older estimate (\ref{RH}) below and
depends only on $d$, $m$ and $\mu$.

 %%%%%%%%%%%%%%%%%%%%%%%%%%%%%%%%%%%%%%%%%%%%%%%%

 %%%%%%%%%%%%%%%%%%%%%%%%%%%%%%%%%%%%%%%%%%%%%%
 
\section{Definition of approximate correctors}

In this section and next  we construct the approximate correctors and establish some preliminary estimates, 
under the assumptions that $A=(a_{ij}^{\alpha\beta})$
 satisfies the ellipticity condition (\ref{cond_ellipticity}) and  $A
\in APW^2(\R^d)$.
As in \cite{Shen-2015}, the existence of the approximate correctors is based on the following lemma.
 
\begin{lemma}\label{lem_appr_u}
	Suppose $A$ satisfies the ellipticity condition (\ref{cond_ellipticity}).
	 Let $F\in L^2_{\loc, \unif} (\R^d;\R^m)$ and $f \in L^2_{\loc,\unif}(\R^d; \R^{m\times d})$. 	Then, for any $T>0$, there exists a unique $u = u_T\in H^1_{\text{loc}}(\R^d;\R^{m})$ such that $u, \nabla u\in L^2_{\loc, \unif} (\R^d)$ and
	\begin{equation}\label{cond_appr_eq}
	-\text{\rm div} (A \nabla u) + T^{-2} u = F + \text{\rm div}(f) \quad \text{in } \R^d.
	\end{equation}
	Moreover, the solution $u$ satisfies the estimate
	\begin{equation}\label{ineq_appr_udu}
	\| \nabla u\|_{S^2_T}
	+ T^{-1} \| u\|_{S^2_T} \le C\Big\{ \| g\|_{S^2_T} + T \| f\|_{S^2_T} \Big\},
		\end{equation}
	where $C$ depends only on $d, m$ and $\mu$.
\end{lemma}

\begin{proof}
	See e.g.  \cite{Shen-2015}.
\end{proof}

The weak solution of (\ref{cond_appr_eq}), given by Lemma \ref{lem_appr_u}, satisfies
	\begin{equation}\label{ineq_appr_dup}
	\| \nabla u\|_{S^q_T} \le C \Big\{ \| f\|_{S^q_T} + T\| F\|_{S^2_T} \Big\},
		\end{equation}
	for $2\le q\le \bar{q}$, where $\bar{q}>2$ and $C>0$ depend only on $d,m$ and $\mu$. 
	This follows from the reverse H\"{o}lder estimate \cite{Giaquinta}
	for weak solutions of $-\text{div}(A \nabla v) = F + \text{div} (f)$
	in $2B=B(x_0, 2R)$,
	\begin{equation}\label{RH}
	\left(\fint_{B} |\nabla v|^{\bar{q}}\right)^{1/\bar{q}}
	\le C \left\{ \left(\fint_{2B} |\nabla v|^2\right)^{1/2}
	+R \left(\fint_{2B} |F|^2\right)^{1/2}
	+\left(\fint_{2B} |f|^{\bar{q}}\right)^{1/\bar{q} }\right\}.
	\end{equation}
	
For $T>0$, let $u = \chi^\beta_{T,j} = (\chi^{1\beta}_{T,j},\cdots,\chi^{m\beta}_{T,j})$ be the weak solution of
\begin{equation}\label{cond_eq_corrector}
-\text{div} (A \nabla u) + T^{-2} u =  \text{div}(A\nabla P_j^\beta) \quad \text{in } \R^d,
\end{equation}
given in Lemma \ref{lem_appr_u}, where $P_j^\beta=x_j e^\beta$.
 The matrix-valued functions $\chi_T = (\chi_{T,j}^\beta) = (\chi_{T,j}^{\alpha\beta})$ are
  called the approximate correctors for $\cL_\e $.
  It follows from (\ref{ineq_appr_udu}) that
  \begin{equation}\label{cor-L-2}
  \|\nabla \chi_T\|_{S^2_T} + T^{-1} \|\chi_T\|_{S^2_T} \le C,
  \end{equation}
  where $C$ depends only on $d$, $m$ and $\mu$.
  Also, by (\ref{ineq_appr_dup}),
\begin{equation}\label{ineq_chiT_p}
\|\nabla \chi_T\|_{S^{\bar{q}}_T} \le C
\end{equation}
for some $\bar{q}>2$. Note that by Sobolev imbedding,
 if $d\ge 3$,
\begin{equation}\label{ineq_chiT_pbar}
T^{-1} \|\chi_T\|_{S_T^{{p}}}\le C
\end{equation}
for some ${p} > 2d/(d-2)$.
If $d=2$, we have $\norm{\chi_T}_\infty \le CT$. Furthermore, by the De Giorgi - Nash estimates, we also have $\norm{\chi_T}_\infty \le CT$, if $m=1$.

\begin{lemma}\label{W-B-lemma}
Let $u\in H^1_{\loc}(\R^d;\R^m)$ be a weak solution of (\ref{cond_appr_eq}) in $\R^d$, given by
Lemma \ref{lem_appr_u}.
Then 
	\begin{equation}\label{3-3}
	\norm{\nabla u}_{S^2_R} + T^{-1} \norm{u}_{S^2_R} \le C \Big\{ \norm{f}_{S^2_R} + T \norm{F}_{S^2_R} \Big\}
	\end{equation}
for any $R\ge T$,
	where $C$ depends on $d$, $m$ and $\mu$.
\end{lemma}	

\begin{proof}

It follows by Caccioppoli's inequality that
	\begin{align}\label{ineq_appr_Caccio}
	\begin{aligned}
	& \fint_{B(x_0,R)} |\nabla u|^2 + \fint_{B(x_0,R)} T^{-2} |u|^2 \\
	&\le  \frac{C}{R^2} \fint_{B(x_0,2R)} | u|^2 + C\fint_{B(x_0,2R)} |f|^2 　+ C\fint_{B(x_0,2R)} T^2 |F|^2
	\end{aligned}
	\end{align}
	for any $x_0\in \R^d$ and $R>0$, where $C$ depends only on $d,m$ and $\mu$. 	
	This, together with the observation
	$\| F\|_{S_{2R}^2} \le C_d\,  \| F\|_{S^2_R}$, gives
	$$
	\|\nabla u\|_{S^2_R}  +T^{-1} \|u\|_{S^2_R}
	\le C\Big\{ \| f\|_{S^2_R} + T \| F\|_{S^2_R} \Big\} + C R^{-1} \| u\|_{S^2_R},
	$$
	from which the estimate (\ref{3-3}) follows if $R\ge 2C T$.
	Finally, we observe that the case $T\le R<2CT$ follows directly  from (\ref{ineq_appr_udu}).
	\end{proof}	

\begin{lemma} \label{lem_chiT_A}
	Suppose $A$ satisfies (\ref{cond_ellipticity}). 
	Then there exists some $2<p<\infty$, depending only on $d,m$ and $\mu$, such that for any $y,z\in \R^d$,
	\begin{equation}\label{ineq_chiT_B2Bq}
	\|\Delta_{yz} (\nabla \chi_T)\|_{S^2_R}
	+ T^{-1}\|\Delta_{yz} (\chi_T)\|_{S^2_R}
	\le C\, \| \Delta_{yz} (A)\|_{S^p_R},
	\end{equation}
	where $R\ge T$ and  $C$ depends only on $d$, $m$ and $\mu$.
	\end{lemma}

\begin{proof}
	Fix $1\le j\le d, 1\le \beta \le m$, and $y,z\in \R^d$. Let
	\begin{equation*}
	u(x) = \chi_{T,j}^\beta(x+y) - \chi_{T,j}^\beta(x+z) \quad \text{and} \quad v(x) = \chi_{T,j}^\beta(x+z).
	\end{equation*}
	Then
	\begin{align}\label{cond_uv_eq}
	\begin{aligned}
	-\text{div}(A(x+y) \nabla u) + T^{-2} u = \text{div}&( (A(x+y) - A(x+z)) \nabla P_j^\beta ) \\
	&+ \text{div} ( (A(x+y) - A(x+z)) \nabla v ).
	\end{aligned}	
	\end{align}
	It follows from Lemma \ref{W-B-lemma} that for any $R\ge T$,
	\begin{align}\label{ineq_uv_2}
	\begin{aligned}
	& \| \nabla u\|_{S_R^2} + T^{-1} \| u \|_{S^2_R}\\
	& \le C\, \| \Delta_{yz} (A) \|_{S^2_R} + C \sup_{x_0 \in \R^d}
	 \left( \fint_{B(x_0,R)} |\Delta_{yz} (A) |^2 |\nabla v |^2\, dx \right)^{1/2}.
	\end{aligned}	
	\end{align}
	By  (\ref{ineq_chiT_p}) we have 
	$$
	\|\nabla v\|_{S_R^{\bar{q}}}\le C\, \|\nabla v\|_{S_T^{\bar{q}}} \le C.
	$$
This, together with H\"{o}lder's inequality, allows us to bound the last term in the r.h.s.
	 of (\ref{ineq_uv_2}) by
	\begin{equation*}
	C \sup_{x_0 \in \R^d} \left( \fint_{B(x_0,R)} |\Delta_{yz}(A)|^{p}\, dx \right)^{1/{p}},
	\end{equation*}
	where $\frac{1}{p} +\frac{1}{\bar{q}}=\frac12$ and $\bar{q}$ is given by (\ref{ineq_chiT_p}). 
	In view of (\ref{ineq_uv_2}) we have proved the estimate (\ref{ineq_chiT_B2Bq}).
	\end{proof}
	
\begin{theorem}\label{lem_corrector_B2W2}
	Suppose that $A$ satisfies  (\ref{cond_ellipticity}) and $A\in APW^2(\R^d)$. Then $\chi_T,\nabla \chi_T \in APW^2(\R^d)$.
\end{theorem}

\begin{proof}
	 By Lemma \ref{lem_chiT_A} we obtain
	\begin{align}\label{ineq_chiT_W2W2}
	\begin{aligned}
	& \sup_{y\in \R^d}\inf_{|z|\le L}\norm{\Delta_{yz} (\nabla \chi_T) }_{S^2_R} 
	 + T^{-1} \sup_{y\in \R^d}\inf_{|z|\le L} \norm{\Delta_{yz} (\chi_T)}_{S^2_R} \\
	&\le C \norm{A}_\infty^{1-2/p} \sup_{y\in \R^d}
	\inf_{|z|\le L}\norm{\Delta_{yz} (A)}^{2/p}_{S_R^2},
	\end{aligned}
	\end{align}
	where $R\ge T$, $0<L<\infty$,
	 and $C$ depends only on $d$, $m$ and $\mu$.
	Since $A \in APW^2(\R^d)$, the r.h.s. of (\ref{ineq_chiT_W2W2}) goes to zero
	as $L, R\to \infty$. It follows that the l.h.s. of (\ref{ineq_chiT_W2W2}) goes to zero
	as $L, R\to \infty$.
	Since $\chi_T, \nabla \chi_T \in L^2_{\loc, \unif}$, by Proposition \ref{C-W-theorem},
	this implies that $\chi_T, \nabla \chi_T\in APW^2(\R^d)$.
\end{proof}

It follows from the equation (\ref{cond_eq_corrector}), and Theorem \ref{lem_corrector_B2W2} that if $A\in APW^2(\R^d)$ and $u = \chi_{T,j}^\beta$ for some $1\le j\le d$ and $1\le \beta\le m$, then
\begin{equation}\label{cond_uv_eq1}
\Big\langle{a_{ik}^{\alpha\gamma} \frac{\partial u^\gamma}{\partial x_k}\frac{\partial v^\alpha}{\partial x_i}} \Big\rangle
+ T^{-2} \ag{u^\alpha v^\alpha} = - \Big\langle{a_{ij}^{\alpha\beta} \frac{\partial v^\alpha}{\partial x_i}}\Big\rangle,
\end{equation}
for any $v = (v^\alpha) \in H^1_{\text{loc}}(\R^d;\R^m)$ such  that $v^\alpha,\nabla v^\alpha \in B^2(\R^d)$.
This implies that
	\begin{equation}\label{cond_dchi2psi}
	\frac{\partial}{\partial x_i} \left( \chi_{T,j}^{\alpha\beta}\right) \to \psi_{ij}^{\alpha\beta} \qquad \text{strongly in } B^2(\R^d) \text{ as } T \to \infty,
	\end{equation}
	where $\psi = (\psi_{ij}^{\alpha\beta})$ is defined by (\ref{cond_Vpot}), and that
	\begin{equation}\label{cond_chi2zero}
	T^{-2} \ag{|\chi_T|^2} \to 0 \quad \text{ as } T \to \infty.
	\end{equation}
In fact, it was observed in  \cite{Shen-2015},
	\begin{equation} \label{ineq_psi_dchi}
	\mu \ag{|\psi - \nabla\chi_T|^2} + T^{-2} \ag{|\chi_T|^2} \le \Big\langle{a_{ik}^{\alpha\gamma} \left[ \psi_{kj}^{\gamma\beta} - \frac{\partial }{\partial x_k} \left( \chi_{T,j}^{\gamma\beta} \right)\right] \psi_{ij}^{\alpha\beta}}\Big\rangle.
	\end{equation}

%%%%%%%%%%%%%%%%%%%%%%%%%%%%%%%%%%%%

%%%%%%%%%%%%%%%%%%%%%%%%%%%%%%%%%%%%%%%%

\section{Estimates of approximate correctors, part I}

The goal of this section is to establish the following. 

\begin{thm}\label{thm_chiT_2}
	Suppose that $A\in APW^2(\R^d)$ and satisfies the ellipticity condition (\ref{cond_ellipticity}).  
	Then
		\begin{equation}\label{ineq_chiT_2}
	T^{-1} \|\chi_T\|_{S^2_T}
	\le C\inf_{0<L< T} \left\{ \rho_1 ( L, T) + \frac{L}{T} \right\},
	\end{equation}
	where $\rho_1 (L,T)$ is defined by (\ref{rho-k}) and $C$ depends only on 
	$d$, $m$ and $\mu$.
\end{thm}

The  estimate (\ref{ineq_chiT_2}) follows from a general inequality (\ref{general-inequality}), which 
may be of independent interest.
The inequality  allows us 
to bound the local (uniform) $L^2$ norm of a function at a specific scale by its oscillation and gradient. 

We use  $Q(x, R)$ to denote the (closed) cube centered at $x$  with side length $R$.

\begin{thm}\label{general-theorem-4}
Let $u\in H^1_{\loc}(\R^d)$.
	Suppose that $u, \nabla u \in L^2_{\loc, \unif}(\R^d)$ and
$$
M=\lim_{r\to \infty} \fint_{Q(0,r)} u \quad \text{ exists.}
$$
 Then there exists $C>0$, depending only on $d$, such that for any
 $0<L\le R<\infty$,
\begin{equation}\label{general-inequality}
\| u\|_{S_R^2}
\le |M|
+ C 
\left\{ \sup_{y\in \R^d} \inf_{|z|\le L}
\| \Delta_{yz} (u)  \|_{S^2_R}
+L\,  \| \nabla u\|_{S^2_R} \right\}.
\end{equation}
\end{thm}

The proof of Theorem \ref{general-theorem-4} relies on the following two lemmas.

\begin{lem}\label{lem_u_p}
	Let $u\in H^1_{\loc}(\R^d)$.
	Suppose that $u, \nabla u \in L^2_{\loc, \unif}(\R^d)$. Then, for any
	$0<L, R <\infty$,
	\begin{equation}
	\begin{aligned}
	&\sup_{x\in\R^d} \left( \fint_{Q(x,R)} |u|^2 \right)^{1/2} \\ 
	& \le  C\, \sup_{y\in\R^d}\inf_{|z|\le L}
	 \| \Delta_{yz} (u) \|_{S^2_R} 
	+ C  L \|\nabla u\|_{S^2_R}  
	  +  \sup_{x\in\R^d} \left|\fint_{Q(x,R)} u \right|,
	\end{aligned}
	\end{equation}
	where $C$ depends only on $d$.
\end{lem}

\begin{proof}
	Note that for any $x\in\R^d$,
	\begin{align*}
	\left(\fint_{Q(x,R)} |u|^2 \right)^{1/2} 
	&\le \left(\fint_{Q(x,R)} \Abs{u(t) - \fint_{Q(x,R)} u(t+y) dy}^2 dt \right)^{1/2} +\sup_{z\in\R^d} \Abs{\fint_{Q(z,R)} u} \\
	&= \left(\fint_{Q(x,R)} \Abs{\fint_{Q(x,R)} [u(t) - u(t+y)] dy}^2 dt \right)^{1/2} +\sup_{z\in\R^d} \Abs{\fint_{Q(z,R)} u} \\
	& \le \fint_{Q(x,R)} \left( \fint_{Q(x,R)} |u(t) - u(t+y)|^2 dt\right)^{1/2} dy +\sup_{z\in\R^d} \Abs{\fint_{Q(z,R)} u}\\
	&\le C \sup_{y\in \R^d}\| u(\cdot) - u(\cdot +y)\|_{S^2_R} +\sup_{z\in\R^d} \Abs{\fint_{Q(z,R)} u},
	\end{align*}
	where we have used Minkowski's inequality for the second inequality.
	
	Next, using
	$$
	\aligned
	\| u(\cdot) -u(\cdot +y)\|_{S^2_R}
	&\le \| u(\cdot +y) -u(\cdot +z) \|_{S^2_R} +\| u(\cdot +z) -u(\cdot)\|_{S^2_R}\\
	&\le \| u(\cdot +y) -u(\cdot +z)\|_{S^2_R} +  L \|\nabla u\|_{S^2_R},
	\endaligned
	$$
	where $z\in \R^d$ and $|z|\le L$, we obtain 
	$$
	\sup_{y\in \R^d}\| u(\cdot) - u(\cdot +y)\|_{S^2_R}
	\le  \sup_{y\in \R^d} \inf_{|z|\le L}
	\| \Delta_{yz} (u)  \|_{S^2_R} +  L \, \|\nabla u\|_{S^2_R}.
	$$
		This completes the proof.
\end{proof}

\begin{lem}\label{lem_Uniform_Ave}
	Let $u\in L^p_{\text{\loc, \unif}}(\R^d)$ for some $p>1$. Suppose that 
	\begin{equation*}
	 M = \lim_{r\to\infty} \fint_{Q(0,r)} u \text{ exists.}
	\end{equation*}
	Then, for any $0<L \le R<\infty$,
	\begin{align}
	\begin{aligned}
	\sup_{x\in\R^d} \Abs{\fint_{Q(x,R)} u - M} &\le 2\sup_{y\in\R^d} \inf_{\substack{z\in\R^d \\ |z|\le L}} \fint_{Q(0,R)} |u(t+y) - u(t+z)| dt\\
	& \qquad + C \left(\frac{L}{R}\right)^{1/{p'}} \sup_{x\in\R^d} \left( \fint_{Q(x,R)} |u|^p \right)^{1/p},
	\end{aligned}
	\end{align}
	where $C$ depends only on $d$ and $p$.
\end{lem}

\begin{proof}
	Observe that if $0<L\le  R$ and $|z|\le L$,
	\begin{equation}
	\aligned\label{4-100}
	\Abs{ \fint_{Q(z,R)} u - \fint_{Q(0,R)} u} 
	&\le \frac{1}{R^d} \int_{Q(z,R)\setminus Q(0,R)} |u| + \frac{1}{R^d} \int_{Q(0,R)\setminus Q(z,R)}|u| \\
	& \le C \left(\frac{L}{R}\right)^{1/{p'}} \sup_{x\in\R^d} \left( \fint_{Q(x,R) } |u|^p\right)^{1/p},
	\endaligned
	\end{equation}
	where $C$ depends only on $d$ and $p$. It follows that for any $y\in\R^d$,
	\begin{align*}
	\Abs{ \fint_{Q(y,R)} u - \fint_{Q(0,R)} u} 
	&\le \Abs{ \fint_{Q(y,R)} u - \fint_{Q(z,R)} u}  + \Abs{ \fint_{Q(z,R)} u - \fint_{Q(0,R)} u}  \\
	& \le \Abs{ \fint_{Q(y,R)} u - \fint_{Q(z,R)} u} + C \left(\frac{L}{R}\right)^{1/{p'}} \sup_{x\in\R^d} \left( \fint_{Q(x,R) } |u|^p\right)^{1/p}.
	\end{align*}
	Hence,
	\begin{align*}
	&\sup_{y\in\R^d} \Abs{ \fint_{Q(y,R)} u - \fint_{Q(0,R)} u} \\
	& \le \sup_{y\in\R^d} \inf_{\substack{z\in\R^d \\ |z|\le L}} \fint_{Q(0,R)} |u(t+y) - u(t+z)| dt + C \left(\frac{L}{R}\right)^{1/{p'}} \sup_{x\in\R^d} \left( \fint_{Q(x,R) } |u|^p\right)^{1/p}.
	\end{align*}
	Using
	\begin{equation*}
	\Abs{\fint_{Q(0,R)} u - \fint_{Q(0,kR)} u} \le \frac{1}{k^d}\sum_{i=1}^{k^d} \Abs{\fint_{Q(0,R)} u - \fint_{Q(x_i,R)} u}, 
	\end{equation*}
	where $k\ge 1$ and $Q(0,kR) = \cup_{i=1}^{k^d} Q(x_i,R)$, we then obtain
	\begin{align*}
	&\sup_{y\in\R^d} \Abs{ \fint_{Q(y,R)} u - \fint_{Q(0,kR)} u} \\
	&\le \sup_{y\in\R^d} \Abs{ \fint_{Q(y,R)} u - \fint_{Q(0,R)} u} + \Abs{ \fint_{Q(0,R)} u - \fint_{Q(0,kR)} u} \\
	& \le 2 \sup_{y\in\R^d} \Abs{ \fint_{Q(y,R)} u - \fint_{Q(0,R)} u} \\
	& \le 2\sup_{y\in\R^d} \inf_{\substack{z\in\R^d \\ |z|\le L}} \fint_{Q(0,R)} |u(t+y) - u(t+z)| dt 
	+ C \left(\frac{L}{R}\right)^{1/{p'}} \sup_{x\in\R^d} \left( \fint_{Q(x,R) } |u|^p\right)^{1/p},
	\end{align*}
	from which the lemma follows by letting $k\to\infty$.
\end{proof}

\begin{rmk}\label{gradient-remark}
{\rm
In the place of (\ref{4-100}) we may also use
$$
\aligned
\Abs{ \fint_{Q(z,R)} u - \fint_{Q(0,R)} u} 
&\le \Abs{ \fint_{Q(z,R)} u - \fint_{Q(0,2R)} u} 
+\Abs{ \fint_{Q(0,R)} u - \fint_{Q(0,2R)} u} \\
&\le C  R \left(\fint_{Q(0, 2R)} |\nabla u|^2\right)^{1/2},
\endaligned
$$
where the last step follows by Poincar\'e inequality.
 This would lead to the estimate
 \begin{equation}\label{4-200}
 \aligned 
 \sup_{x\in\R^d} \Abs{\fint_{Q(x,R)} u - M} &\le 2\sup_{y\in\R^d} \inf_{\substack{z\in\R^d \\ |z|\le R}} 
 \fint_{Q(0,R)} |u(t+y) - u(t+z)| dt\\
	& \qquad + C  R  \sup_{x\in\R^d} \left( \fint_{Q(x,R)} |\nabla u|^2 \right)^{1/2}
	\endaligned
 \end{equation}
for any $0<R<\infty$.
}
\end{rmk}

\begin{rmk}\label{ineq_A_ave}
{\rm
	It follows from Lemma \ref{lem_Uniform_Ave} that functions in $APW^2(\R^d)$ have uniform mean
	values. In particular,
	\begin{equation}
	\sup_{x\in\R^d} \Abs{\fint_{Q(x,R)} A - \ag{A} } \le C \inf_{0<L<R} \left\{ \rho_1(L, R) + \frac{L}{R} \right\},
	\end{equation}
	where $C$ depends only on $d$, $m$ and $\mu$.
	}
\end{rmk}

\begin{rmk}\label{ineq_AdchiT_ave}
{\rm
	Let $u(x) = A(x)\nabla \chi_T(x)$, $R = T$ and $p=2$ in  Lemma \ref{lem_Uniform_Ave}. Observe that 
	by (\ref{ineq_chiT_B2Bq}),
	\begin{align*}
	&\fint_{Q(0,T)} |u(t+y) - u(t+z)| dt \\
	&\le C\left( \fint_{Q(0,T)} |A(t+y) - A(t+z)|^2 dt\right)^{1/2} + C \fint_{Q(0,T)} |\nabla \chi_T(t+y) - \nabla\chi_T(t+z)| dt \\
	& \le C \sup_{x\in\R^d} \left( \fint_{Q(x,T)} |A(t+y) - A(t+z)|^q dt \right)^{1/q},
	\end{align*}
	for some $q\in (2,\infty)$. It follows by Lemma \ref{lem_Uniform_Ave} that
	\begin{equation}
	\sup_{x\in\R^d} \Abs{ \fint_{Q(x,T)} A\nabla \chi_T - \ag{A\nabla \chi_T}} \le C \inf_{0<L<T} \left\{ \rho_1( L,T) + \left(\frac{L}{T}\right)^{1/2} \right\},
	\end{equation}
	where $C$ depends only on $d$, $m$ and $\mu$. This estimate will be used in the proof of Theorem \ref{thm_compactness}.
	}
\end{rmk}

We are now in a position to give the proof of Theorem \ref{general-theorem-4}.

\begin{proof}[\bf Proof of Theorem \ref{general-theorem-4}]
By considering the function $u-M$ we may assume that $\langle u \rangle=0$.
It follows from Lemmas  \ref{lem_u_p} and  \ref{lem_Uniform_Ave} that for any $0<L\le R<\infty$,
$$
\aligned
\| u\|_{S^2_R}
& \le C\, \sup_{y\in\R^d}\inf_{|z|\le L}
	 \| \Delta_{yz} (u) \|_{S^2_R} 
	+ C  L \|\nabla u\|_{S^2_R}  
		  + \sup_{x\in\R^d} \left|\fint_{Q(x,R)} u \right|\\
	  &\le C \sup_{y\in\R^d}\inf_{|z|\le L}
	 \| \Delta_{yz} (u) \|_{S^2_R} 
	+ C  L \|\nabla u\|_{S^2_R}  
	 +C \left(\frac{L}{R}\right)^{1/2} \| u\|_{S^2_R},
	  \endaligned
	  $$
	  where $C$ depends only on $d$.
	  Since $u, \nabla u \in L^2_{\loc, \unif}$, it follows that
	  $\|u\|_{S^2_R} $ and $\|\nabla u\|_{S^2_R}$ are finite for any $R>0$.
	Next, we fix a constant $\theta\in (0,1)$ so small that $C\theta^{1/2} < 1/2$. Then,
	if $0<L\le \theta R$, we have
	$$
\| u\|_{S^2_R} \le C\, \sup_{y\in\R^d}\inf_{|z|\le L}
	 \| \Delta_{yz} (u) \|_{S^2_R} 
	+ C  L \|\nabla u\|_{S^2_R}.
	$$
	Finally,  we observe that
	if $\theta R<L \le R$, the estimate (\ref{general-inequality}) follows from
	 Lemma \ref{lem_u_p} and (\ref{4-200}).	
\end{proof}

\begin{proof}[\bf Proof of Theorem \ref{thm_chiT_2}]
To see (\ref{ineq_chiT_2}), we
		 let $L = T$ in  (\ref{general-inequality}) and use (\ref{ineq_chiT_B2Bq}) and
		  the fact that $\|\nabla \chi_T\|_{S^2_T}\le C$.
\end{proof}

%%%%%%%%%%%%%%%%%%%%%%%%%%%%%%%%%%%%%

%%%%%%%%%%%%%%%%%%%%%%%%%%%%%%%%%%%%%%%%

\section{A compactness  theorem}

In this section we establish a compactness theorem, which extends Theorem \ref{homo-theorem} 
in the case $A\in APW^2(\R^d)$. It will play
  a key role in the compactness argument in the next section.
Throughout this section we will assume that 
$A\in APW^2(\R^d)$ and satisfies the ellipticity condition (\ref{cond_ellipticity}).

\begin{thm}\label{thm_compactness}
	Let $\Omega$ be a bounded Lipschitz domain in $\R^d$.
	Suppose that $\{u_\ell \}_{
	\ell =1}^\infty \subset H^1(\Omega;\R^m)$ and
	\begin{equation*}
	-\text{\rm div} (A^\ell (x/\e_\ell) \nabla u_\ell) + \lambda_\ell  u_\ell = F_\ell \quad \text{in } \Omega,
	\end{equation*}
	where $\e_\ell \to 0$, $\lambda_\ell \to \lambda$,
	and $A^\ell(y) = A(y+x_\ell)$ for some $x_\ell \in \R^d$. 
	Assume that $u_\ell \rightharpoonup u_0$ weakly in $H^1(\Omega; \R^m)$ and $F_\ell \to F_0$
	 strongly in $H^{-1}(\Omega;\R^m)$. Then $A^\ell(x/\e_\ell) \nabla u_\ell \rightharpoonup \widehat{A} \nabla u_0$ weakly in $L^2(\Omega;\R^{m\times d})$, and $u_0\in H^1(\Omega;\R^m)$ is a weak solution of
	\begin{equation*}
	-\text{\rm div}(\widehat{A} \nabla u_0) + \lambda u_0 = F_0 \quad \text{in } \Omega,
	\end{equation*}
	where $\widehat{A}$ is the homogenized matrix of $A$.
\end{thm}

We will prove Theorem \ref{thm_compactness} by using Tartar's method of test functions and the following
lemma.

\begin{lem}\label{lem_echiT2zero}
	 Let $\{x_\ell\} \subset \R^d, \e_\ell \to 0$, and $T_\ell = \e_\ell^{-1}$. Then for any bounded domain $\Omega$,
	\begin{equation}\label{cond_echiT2zero}
	\e_\ell  \chi_{T_\ell} ((x/\e_\ell) + x_\ell) \rightharpoonup 0 \text{ weakly in } H^1(\Omega),
	\end{equation}
	and
	\begin{equation}\label{cond_adchi2ahat}
	a_{ik}^{\alpha\gamma} ((x/\e_\ell) + x_\ell) \frac{\partial}{\partial x_k} \left\{ x_j\delta^{\gamma\beta} 
	+ \e_\ell \chi_{{T_\ell},j}^{\gamma\beta} ((x/\e_\ell) + x_\ell) \right\} \rightharpoonup \widehat{a}_{ij}^{\alpha\beta} \text{ weakly in } L^2(\Omega),
	\end{equation}
	as $\ell\to\infty$.
\end{lem}

\begin{proof}
	We start with the proof of (\ref{cond_adchi2ahat}).
	 Let $\{g_\ell\}$ denote the sequence in  (\ref{cond_adchi2ahat}). 
	 Since $T_\ell = \e_\ell^{-1}$, in view of estimate (\ref{ineq_chiT_p}), $\{g_\ell\}$ is bounded in $L^2(\Omega)$. 
	 Thus, by a density argument, it suffices to show that
	\begin{equation}
	\fint_{B(z,r)} g_\ell \to \widehat{a}_{ij}^{\alpha\beta} \qquad \text{as } \ell\to\infty,
	\end{equation}
	for any ball $B(z,r)$ in $\R^d$. Recall that $\widehat{A} = \ag{A} + \ag{A\psi}$, where $\psi = (\psi_{ij}^{\alpha\beta})$ is defined by (\ref{cond_Vpot}). Hence,
	\begin{align*}
	&\Abs{\fint_{B(z,r)} g_\ell - \widehat{a}_{ij}^{\alpha\beta}} 
	= \Abs{\fint_{B(x_\ell+T_\ell z,rT_\ell)} a_{ik}^{\alpha\gamma}(y) \left\{ \delta_{jk} \delta^{\gamma\beta} + \frac{\partial}{\partial y_k} (\chi_{{T_\ell},j}^{\gamma\beta}) \right\} dy - \widehat{a}_{ij}^{\alpha\beta} } \\
	&\le \sup_{x\in\R^d} \Abs{\fint_{B(x,rT_\ell)} A - \ag{A} } 
	+\sup_{x\in\R^d} \Abs{\fint_{B(x,rT_\ell)} A\nabla \chi_{T_\ell} -\ag{A\nabla \chi_{T_\ell}}} + |\ag{A(\nabla \chi_{T_\ell} - \psi)}|.
	\end{align*}
	This, together with Remark \ref{ineq_A_ave}, Remark \ref{ineq_AdchiT_ave} and (\ref{cond_dchi2psi}), gives (\ref{cond_adchi2ahat}).
	
	Finally, let $\{f_\ell\}$ denote the sequence in (\ref{cond_echiT2zero}). It follows from estimates (\ref{ineq_chiT_pbar}) and (\ref{ineq_chiT_p}) that $\{f_\ell\}$ is bounded in $H^1(\Omega)$. 
	Also, by the estimate (\ref{ineq_chiT_2}), 
	$f_\ell\to 0$ strongly in $L^2(\Omega)$. This implies that $f_\ell \rightharpoonup 0$ weakly in $H^1(\Omega)$.
\end{proof}

\begin{proof}[\bf Proof of Theorem \ref{thm_compactness}]
	 Let
	$
	p_\ell(x) = A^\ell(x/\e_\ell) \nabla u_\ell(x).
	$
	Observe that since $\{ u_\ell\}$ is bounded in $H^1(\Omega; \R^m)$,
	$\{p_\ell\}$ is bounded in $L^2(\Omega;\R^{m\times d})$. We will show that if a subsequence $\{p_{\ell^\prime}\}$ converges weakly in $L^2(\Omega;\R^{m\times d})$ to $p_0$, then $p_0 = \widehat{A} \nabla u_0$. This would imply that the full sequence converges weakly to $\widehat{A} \nabla u_0$. As a result, we also obtain $-\text{div} (\widehat{A} \nabla u_0) + \lambda u_0 = F_0$ in $\Omega$.
	
	Without loss of generality let us assume that $p_\ell$ converges weakly in $L^2(\Omega;\R^{m\times d})$ to $p_0$, 
	as $\ell\to\infty$. Let $\psi\in C_0^1(\Omega)$. 
	Fix $1\le j \le d$ and $1\le \beta\le m$. Let $T_\ell = \e_\ell^{-1}$.
	Note that
	\begin{align}\label{cond_chiT_*}
	\begin{aligned}
	&\ag{F_\ell, (P_j^\beta + \e_\ell \chi_{T_\ell,j}^{\ell * \beta} (x/\e_\ell)) \psi}_{H^{-1} (\Omega) \times H_0^1(\Omega)} \\
	& = \lambda_\ell \int_\Omega u_\ell (P_j^\beta + \e_\ell \chi_{T_\ell,j}^{\ell *\beta} (x/\e_\ell)) \psi dx \\
&\qquad\qquad
	+ \int_\Omega A^\ell(x/\e_\ell) \nabla u_\ell \cdot \nabla \Big\{ (P_j^\beta + \e_\ell \chi_{T_\ell,j}^{\ell *\beta} (x/\e_\ell)) \psi \Big\} 
	dx \\
	& = \lambda_\ell\int_\Omega u_\ell (P_j^\beta + \e_\ell \chi_{T_\ell,j}^{\ell *\beta} (x/\e_\ell)) \psi dx\\
&\qquad\qquad
	 + \int_\Omega A^\ell(x/\e_\ell) \nabla u_\ell \cdot \nabla (P_j^\beta + \e_\ell \chi_{T_l,j}^{\ell *\beta} (x/\e_\ell)) \psi dx \\
	&\qquad\qquad + \int_{\Omega} A^l(x/\e_\ell) \nabla u_\ell \cdot (P_j^\beta 
	+ \e_\ell \chi_{T_\ell,j}^{\ell *\beta} (x/\e_\ell)) (\nabla \psi) dx \\
	&= \lambda_\ell \int_\Omega u_\ell (P_j^\beta 
	+ \e_\ell \chi_{T_\ell,j}^{\ell *\beta} (x/\e_\ell)) \psi dx 
	- \int_\Omega u_\ell \cdot \e_\ell \chi_{T_\ell,j}^{\ell * \beta} (x/\e_\ell) \psi dx\\
	& \qquad\qquad - \int_{\Omega} u_\ell\cdot (A^\ell)^* (x/\e_\ell) \nabla (P_j^\beta 
	+ \e_\ell \chi_{T_\ell,j}^{\ell * \beta} (x/\e_\ell)) (\nabla \psi) dx\\
	& \qquad \qquad + \int_{\Omega} A^\ell(x/\e_\ell) \nabla u_\ell \cdot (P_j^\beta 
	+ \e_\ell \chi_{T_\ell,j}^{\ell *\beta} (x/\e_\ell)) (\nabla \psi) dx,
	\end{aligned}
	\end{align}
	where $\chi_{T,j}^{\ell *\beta}$ denote the approximate correctors
	 for the adjoint matrix $(A^\ell)^*$. We point out that the following equation
	\begin{equation*}
	\text{div} \left\{ (A^\ell)^* (x/\e_\ell) \nabla (P_j^\beta + \e_\ell \chi_{T_\ell,j}^{\ell*\beta} (x/\e_\ell)) \right\}
	 = -\e_\ell \chi_{T_\ell,j}^{\ell*\beta} (x/\e_\ell) \qquad \text{in } \R^d,
	\end{equation*}
	was used for the last equality  in (\ref{cond_chiT_*}). Since $A^\ell(y) = A(y+x_\ell)$, we have
	\begin{equation}
	\chi_{T,j}^{\ell*\beta} (y) = \chi_{T,j}^{*\beta} (y+x_\ell),
	\end{equation}
	where $\chi_{T,j}^{*\beta}$ denote the approximate correctors for $A^*$.	
	
	We now let $\ell \to\infty$ in (\ref{cond_chiT_*}) and use Lemma \ref{lem_echiT2zero} 
	(with $A^*$ in the place of $A$) to find the limit on each side. 
	Since $F_\ell \to F_0$ strongly in $H^{-1}(\Omega;\R^m)$, by (\ref{cond_echiT2zero}), the l.h.s.
	 of (\ref{cond_chiT_*}) converges to $\ag{F_0,P_j^\ell\psi}$. 
	 Since $u_\ell \to u_0$ strongly in $L^2(\Omega;\R^m)$,
	it also follows from (\ref{cond_echiT2zero}) that the sum of first two terms in the r.h.s. of (\ref{cond_chiT_*}) converges to $\lambda \int_{\Omega} u_0 P_j^\beta \psi dx$, and the fourth term converges to
	\begin{equation*}
	\int_{\Omega} p_0\cdot P_j^\beta (\nabla \psi) \, dx.
	\end{equation*}
	Similarly, using (\ref{cond_adchi2ahat}) and the fact that $u_\ell \to u_0$ strongly in $L^2(\Omega;\R^m)$, we see that the third term in the r.h.s. of (\ref{cond_chiT_*}) converges to
	\begin{equation*}
	-\int_{\Omega} u_0^\alpha \cdot \widehat{a^*}_{ij}^{\alpha\beta} \frac{\partial \psi}{\partial x_i}\, dx 
	= \int_{\Omega} \widehat{a^*}_{ij}^{\alpha\beta} \frac{\partial u_0^\alpha}{\partial x_i} \psi\,  dx = \int_{\Omega} \widehat{a}_{ji}^{\beta\alpha} \frac{\partial u_0^\alpha}{\partial x_i} \psi \, dx,
	\end{equation*}
	where $\widehat{A^*}= \big(\widehat{a^*}^{\alpha\beta}_{ij}\big)$ and we have used the fact $(\widehat{A})^* = \widehat{A^*}$ for the last step. 
	As a result, we have proved that
	\begin{equation}\label{cond_inner_FP}
	\ag{F_0,P_j^\beta \psi} = \lambda \int_{\Omega} u_0\cdot P_j^\beta \psi \, dx + \int_{\Omega} \widehat{a}_{ji}^{\beta\alpha} \frac{\partial u_0^\alpha}{\partial x_i} \psi \, dx + \int_{\Omega} p_0\cdot P_j^\beta (\nabla \psi) \, dx.
	\end{equation}
	
	Finally, by taking limits in the equation 
	\begin{equation*}
	\ag{F_\ell,P_j^\beta \psi } = \lambda_\ell \int_{\Omega} u_\ell \cdot P_j^\beta \psi \, dx 
	+ \int_{\Omega} p_\ell\cdot \nabla(P_j^\beta \psi) \, dx,
	\end{equation*}
	we obtain
	\begin{align*}
	\ag{F_0,P_j^\beta \psi} &=\lambda \int_{\Omega} u_0\cdot P_j^\beta \psi \, dx 
	+ \int_{\Omega} p_0\cdot \nabla(P_j^\beta \psi) \, dx \\
	& = \lambda \int_{\Omega} u_0\cdot P_j^\beta \psi \, dx + \int_{\Omega} p_0\cdot \nabla(P_j^\beta) \psi \, dx + \int_{\Omega} p_0\cdot P_j^\beta (\nabla\psi) \, dx.
	\end{align*}
	Since $\psi\in C_0^1(\Omega)$ is arbitrary, this, together with (\ref{cond_inner_FP}), implies that $p_0\cdot (\nabla P_j^\beta) =\widehat{a}_{ji}^{\beta\alpha} \frac{\partial u_0^\alpha}{\partial x_i}$, i.e., $p_0 = \widehat{A} \nabla u_0$. The proof is complete.
\end{proof}

%%%%%%%%%%%%%%%%%%%%%%%%%%%%%%%%%%%%%%%%%%

%%%%%%%%%%%%%%%%%%%%%%%%%%%%%%%%%%%%%%%%%%%%

\section{H\"older estimates at large scale}

In this section we establish  an $L^2$-based H\"older estimate at large scale for the elliptic system,
\begin{equation}\label{cond_comp_Ff}
-\text{\rm div}(A(x/\e) \nabla u_\e) + \lambda u_\e = F + \text{\rm div} ( f ).
\end{equation}
Throughout this section we will assume 
that $A\in APW^2(\R^d)$ and satisfies the ellipticity condition (\ref{cond_ellipticity}).

\begin{thm}\label{thm_ue_Ff}
	Fix $\sigma \in (0,1)$ and $B=B(x_0, R)$ for some $x_0\in \R^d$.
	 Let $u_\e \in H^1(B; \R^m)$ be a weak solution of
	\begin{equation}\label{cond_ue_BR}
	-\text{\rm div}(A(x/\e) \nabla u_\e) + \lambda u_\e = F + \text{\rm div} (f) \quad \text{in } B,
	\end{equation}
	for $0<\e < R$ and $\lambda \in [0,R^{-2}]$.
	Then, if $\e \le r \le R/2$,
	\begin{align}\label{ineq_due_Br}
	\begin{aligned}
	&\left(\fint_{B(x_0,r)} |\nabla u_\e|^2 \right)^{1/2} + \sqrt{\lambda} \left(\fint_{B(x_0,r)} |u_\e|^2 \right)^{1/2}\\
	& \le C_\sigma\bigg( \frac{R}{r}\bigg)^{\sigma} \Bigg\{  \frac{1}{R}\left( \fint_{B(x_0,R)} | u_\e|^2 \right)^{1/2} + 
	\sup_{\substack{x\in B(x_0, R/2)\\ r\le t\le R/2}}  t \left( \fint_{B(x,t)} |F|^2 \right)^{1/2}  \\
	&\qquad \qquad \qquad\qquad\qquad\qquad+ \sup_{x\in B(x_0,R/2)} \left( \fint_{B(x,r)} |f|^2 \right)^{1/2}\Bigg\},
	\end{aligned}
	\end{align}
	where $C_\sigma$ depends only on  $\sigma$ and $A$.
\end{thm}

\begin{rmk}\label{remark-6-1}
{\rm
One may regard the estimate (\ref{ineq_due_Br}) as a H\"older estimate at large scale, as the estimate
$$
\sup_{\substack{0<r<R/2 \\ x\in B(x_0, R/2)}}
r^{\sigma} \left(\fint_{B(x,r)} |\nabla u_\e|^2\right)^{1/2}<\infty
$$
would imply that $u_\e\in C^{1-\sigma} (B(x_0, R/2))$.
Note that since no smoothness condition is imposed on $A$,
estimate (\ref{ineq_due_Br}) may fail to hold for  $0<r<\e$.
}
\end{rmk}

As a corollary of Theorem \ref{thm_ue_Ff}, we obtain a Liouville property for the elliptic operator $\mathcal{L}_1$.

\begin{cor}\label{Corollary-L}
Suppose that $A\in APW^2(\R^d)$ and satisfies the ellipticity condition  (\ref{cond_ellipticity}).
Let $u\in H^1_{\loc} (\R^d; \R^m)$ be a weak solution of
$\text{\rm div} (A\nabla u)=0$ in $\R^d$.
Assume that there exist constants $\sigma \in (0,1)$ and $C_u>0$ such that
\begin{equation}\label{growth-condition}
\left(\fint_{B(0,R)} |u|^2\right)^{1/2} \le C_u \, R^\sigma
\end{equation}
for all $R\ge1$. Then $u$ is constant in $\R^d$.
\end{cor}

\begin{proof} Choose $\sigma_1\in (\sigma,1)$. It follows from Theorem \ref{thm_ue_Ff} that
$$
\aligned
\left(\fint_{B(0,r)} |\nabla u|^2\right)^{1/2}
&\le C \left(\frac{R}{r}\right)^{1-\sigma_1} \frac{1}{R} \left(\fint_{B(0,R)} |u|^2\right)^{1/2}\\
&\le  \frac{CR^{\sigma-\sigma_1}}{r^{1-\sigma_1}}
\endaligned
$$
for any $1\le r\le R/2$.
By letting $R\to \infty$, we obtain $\nabla u=0$ in $B(0,r)$ for any $r>1$.
Thus $u$ is constant in $\R^d$.
\end{proof}

Theorem \ref{thm_ue_Ff} will be proved by using a compactness argument introduced by Avellaneda and Lin \cite{AL-1987} to the study of uniform regularity estimates in homogenization.

We begin with a Caccioppoli's inequality.

\begin{lem} \label{lem_Cacci_Ff}
	Let $u$ be a weak solution of $-\text{\rm div}(A\nabla u) + \lambda u = F+\text{\rm div} (f)$ in $B(x_0,2R)$ for some $x_0\in\R^d, R>0$, and $\lambda \ge 0$. Then
	\begin{align}\label{ineq_Cacci}
	\begin{aligned}
	\left(\fint_{B} |\nabla u|^2 \right)^{1/2} \le & \frac{C}{R} \left(\fint_{2B} \Big| u - \fint_{2B} u \Big|^2 \right)^{1/2} +  C \left( \fint_{2B} |f|^2 \right)^{1/2}\\
	& \qquad + CR \left\{ \lambda \left( \fint_{2B} |u|^2 \right)^{1/2} + \left( \fint_{2B} |F|^2 \right)^{1/2}\right\},
	\end{aligned}
	\end{align}
	where $B=B(x_0, R)$ and $C$ depends only on $d$, $m$ and $\mu$.
\end{lem}

\begin{proof}
This is well known.
\end{proof}

To assure that our estimates are translation invariant in the compactness argument, 
we introduce the set of all matrices obtained from $A$ by translation, 
\begin{equation*}
\cA = \Big\{ M=M(y): M(y) = A(y+z) \text{ for some } z \in\R^d \Big\}.
\end{equation*}

\begin{lem}\label{lem_comp_first}
	Fix $\sigma\in (0,1)$. There exist $\e_0 \in (0,1/2)$ and $\theta \in (0,1/8)$, depending at most on
	$\sigma$ and $A$, such that
	\begin{equation}\label{ineq_ue_lambda}
	\begin{aligned}
	&\left(\fint_{B(0,\theta)} \Big| u_\e - \fint_{B(0,\theta)} u_\e \Big|^2 \right)^{1/2} + \theta\sqrt{\lambda}  \left(\fint_{B(0,\theta)} | u_\e|^2 \right)^{1/2} \\
	& \qquad \le \theta^{\sigma} \Bigg\{ \left(\fint_{B(0,1)} \Big| u_\e - \fint_{B(0,1)} u_\e \Big|^2 \right)^{1/2} + \sqrt{\lambda} \left(\fint_{B(0,1)} \left| u_\e \right|^2 \right)^{1/2} \\
	&\qquad \qquad  \qquad\qquad
	+ \e_0^{-1} \left( \fint_{B(0,1)} |F|^2 \right)^{1/2} + \e_0^{-1} \left( \fint_{B(0,1)} |f|^2 \right)^{1/2} \Bigg\},
	\end{aligned}
	\end{equation}
	whenever $0<\e<\e_0$ and $u_\e \in H^1(B(0,1);\R^m)$ is a  weak solution of
	\begin{equation}\label{cond_ue_Ff}
	-\text{\rm div}(M(x/\e) \nabla u_\e) + \lambda u_\e = F  + \text{\rm div} (f) \quad \text{in } B(0,1)
	\end{equation}
	for some $M\in\cA$ and $\lambda\in [0,\e_0^2]$.
\end{lem}

\begin{proof}
	The lemma is proved by contradiction, along a line of argument used in \cite{AL-1987} for periodic coefficients. 
	We will show that there exist $\e_0 \in (0,1/2)$ and $\theta \in (0,1/8)$, 
	depending at most on  $\sigma$ and $A$, such that whenever $0<\e<\e_0$ and
	$u_\e$ is a solution of (\ref{cond_ue_Ff}), then
	\begin{equation}\label{ineq_theta}
	\begin{aligned}
	&\left(\fint_{B(0,\theta)} \Big| u_\e - \fint_{B(0,\theta)} u_\e \Big|^2 \right)^{1/2} + \theta \left( \fint_{B(0,\theta)} | u_\e|^2 \right)^{1/2}\\
	& \le \frac{\theta^\sigma}{2} \left\{  \left(\fint_{B(0,1)} \left| u_\e \right|^2 \right)^{1/2} + \e_0^{-1}\left( \fint_{B(0,1)} |F|^2 \right)^{1/2} + \e_0^{-1} \left( \fint_{B(0,1)} |f|^2 \right)^{1/2} \right\}.
	\end{aligned}
	\end{equation}
	We claim that (\ref{ineq_theta}) implies (\ref{ineq_ue_lambda}). In fact, assume (\ref{ineq_theta}) is true, we set
	$
	E = \fint_{B(0,1)} u_\e
	$
	and apply (\ref{ineq_theta}) to $v_\e = u_\e - E$,
	 which is a solution of (\ref{cond_ue_Ff}) with the r.h.s. $F$ replaced by $F-\lambda E$. As a result, it follows that
	\begin{align}
	\begin{aligned}\label{ineq_theta_1}	
	& \left(\fint_{B(0,\theta)} \Big| u_\e - \fint_{B(0,\theta)} u_\e \Big|^2 \right)^{1/2}
	 + \theta \left(\fint_{B(0,\theta)} \Big|u_\e - \fint_{B(0,1)} u_\e \Big|^2\right)^{1/2} \\
	&\qquad  \le \frac{\theta^\sigma}{2} \Bigg\{  \left(\fint_{B(0,1)} \Big| u_\e - \fint_{B(0,1)} u_\e \Big|^2 \right)^{1/2} \\
	&\qquad \qquad + \e_0^{-1}\left( \fint_{B(0,1)} \Big|F-\lambda \fint_{B(0,1)} u_\e\Big|^2 \right)^{1/2} + \e_0^{-1} \left( \fint_{B(0,1)} |f|^2 \right)^{1/2} \Bigg\} \\
	&\qquad \le \frac{\theta^\sigma}{2} \Bigg\{  \left(\fint_{B(0,1)} \Big| u_\e - \fint_{B(0,1)} u_\e \Big|^2 \right)^{1/2} + \sqrt{\lambda} \left( \fint_{B(0,1)} |u_\e|^2 \right)^{1/2} \\
	&\qquad \qquad \qquad+ \e_0^{-1}\left( \fint_{B(0,1)} |F|^2 \right)^{1/2} + \e_0^{-1} \left( \fint_{B(0,1)} |f|^2 \right)^{1/2} \Bigg\},
	\end{aligned}
	\end{align}
	where we have used the assumption $\lambda\le \e_0^2$ in the last step.
	Using the fact that $\theta,\lambda \in [0,1]$ as well as the triangle inequality, we have
	\begin{equation*}
	\theta \left(\fint_{B(0,\theta)} \Big|u_\e - \fint_{B(0,1)} u_\e \Big|^2 \right)^{1/2} \ge \theta\sqrt{\lambda} \left(\fint_{B(0,\theta)} \Abs{u_\e }^2 \right)^{1/2} - \frac{\theta^\sigma}{2}\sqrt{\lambda} \left(\fint_{B(0,1)} \Abs{u_\e }^2 \right)^{1/2},
	\end{equation*}
	where we also assume that $\theta$ is small enough so that $\theta \le \theta^\sigma/2$. This, together with (\ref{ineq_theta_1}), gives the desired estimate (\ref{ineq_ue_lambda}). 
	
	It remains to prove (\ref{ineq_theta}). 
	By normalizing the r.h.s. of (\ref{ineq_theta}),
	without loss of generality,  
	 it suffices to show that there exist $\e_0 \in (0,1/2)$ and $\theta \in (0,1/8)$ so that if
	\begin{equation}\label{cond_uFf_norm}
	\left(\fint_{B(0,1)} \left| u_\e \right|^2 \right)^{1/2} + \e_0^{-1}\left( \fint_{B(0,1)} |F|^2 \right)^{1/2} 
	+\e_0^{-1} \left( \fint_{B(0,1)} |f|^2 \right)^{1/2} \le 1,
	\end{equation}
	then
	\begin{equation}\label{ineq_theta_norm}
	\left(\fint_{B(0,\theta)} \Big| u_\e - \fint_{B(0,\theta)} u_\e \Big|^2 \right)^{1/2} + \theta \left( \fint_{B(0,\theta)} | u_\e|^2 \right)^{1/2} \le \frac{\theta^\sigma}{2}.
	\end{equation}
	To this end we first 
	note that if $u\in H^1(B(0,1/2);\R^m)$ is a weak solution of
	\begin{equation}\label{constant-system}
	-\text{div}(A^0 \nabla u) + \lambda u = 0 \quad \text{in } B(0,1/2),
	\end{equation}
	where $\lambda\in [0,1]$ and $A^0$ is a constant matrix satisfying the ellipticity condition 
	(\ref{ellipticity-1}), then, for any $\theta\in (0,1/8)$,
	\begin{equation}\label{ineq_Ff_Lip}
	\left(\fint_{B(0,\theta)} \Big| u - \fint_{B(0,\theta)} u \Big|^2 \right)^{1/2} 
	+ \theta \left(\fint_{B(0,\theta)} |u|^2 \right)^{1/2}
	\le C_0 \theta \left(\fint_{B(0,1/2)} \left| u \right|^2 \right)^{1/2} ,
	\end{equation}
	where $C_0$ depends only on $d,m$ and $\mu$. This follows  from the interior Lipschitz estimate
	$$
	\|\nabla u\|_{L^\infty(B(0,1/4))} +\| u\|_{L^\infty(B(0,1/4))}
	\le C\, \| u\|_{L^2(B(0,1/2))} 
	$$
	 for solutions of the elliptic system (\ref{constant-system}) with constant coefficients.
	
		We now choose $\theta\in (0,1/4)$ so small that $2^{d/2}C_0\theta < \theta^{\sigma}/2$. We claim that (\ref{ineq_theta_norm}) holds for this $\theta$ and for some $\e_0\in (0,1/2)$, 
		which depends at most on $\sigma$ and $A$.
	
	Suppose this is not the case. Then there exist sequences 
	$\{ M^\ell \} \subset \cA, \{\e_\ell\} \subset \R_+$, $\{\lambda_\ell \} \subset [0,1]$,
	$\{ F_\ell\}\subset L^p(B(0,1); \R^m)$,
	$\{ f_\ell\}\subset L^2(B(0,1); \R^{m\times d})$,
	and $ \{u_\ell\} \subset H^1(B(0,1);\R^m)$, 
	 such that $\e_\ell\to 0$, $0\le \lambda_\ell \le\e_\ell^2$,
	\begin{equation}\label{cond_ul_Ff}
	-\text{div}(M^\ell(x/\e_\ell) \nabla u_\ell) + \lambda_\ell u_\ell = F_\ell + \text{div} (f_\ell) \quad \text{in } B(0,1),
	\end{equation}
	and
	\begin{equation}\label{cond_f_0}
	\left(\fint_{B(0,1)} \left| u_\ell \right|^2 \right)^{1/2} + \e_\ell^{-1}\left( \fint_{B(0,1)} |F_\ell|^2 \right)^{1/2} 
	+\e_\ell^{-1} \left( \fint_{B(0,1)} |f_\ell|^2 \right)^{1/2} \le 1,
	\end{equation}
	but
	\begin{equation}\label{ineq_theta_false}
	\left(\fint_{B(0,\theta)} \left| u_\ell - \fint_{B(0,\theta)} u_\ell \right|^2 \right)^{1/2} 
	+ \theta \left( \fint_{B(0,\theta)} | u_\ell |^2 \right)^{1/2} > \frac{\theta^\sigma}{2}.
	\end{equation}
		By passing to subsequences, we may assume that 
	\begin{align}
	\begin{aligned}\label{cond_weakC}
	 u_\ell \rightharpoonup u \text{ weakly in } L^2(B(0,1);\R^m).
	\end{aligned}
	\end{align}
	By Caccioppoli's inequality  the sequence $\{u_\ell \}$ is bounded in $H^1(B(0,1/2); \R^m)$. 
	By passing to a subsequence, we may further assume that
	 $u_\ell$ converges to $u$ weakly in $H^1(B(0,1/2);\R^m)$ and hence strongly in $L^2(B(0,1/2);\R^m)$.
	  Also note  that 
	  $\lambda_\ell \to 0$, and
	  $F_\ell + \text{div} f_\ell$ converges to zero strongly in $H^{-1}(B(0,1/2); \R^m)$. This allows us to apply Theorem \ref{thm_compactness} 
	  to the system (\ref{cond_ul_Ff}) in $B(0,1/2)$. It follows that $u$ is a weak solution 
	  of $-\text{div}(\widehat{A} \nabla u)  = 0$ in $B(0,1/2)$.
	  
	  Finally, since $u_\ell \to u$ strongly in $L^2(B(0,1/2); \R^m)$,
	  by  (\ref{ineq_theta_false}), we obtain
	\begin{equation}\label{C-0}
	\left(\fint_{B(0,\theta)} \Big| u - \fint_{B(0,\theta)} u \Big|^2 \right)^{1/2} 
	+ \theta \left( \fint_{B(0,\theta)} | u|^2 \right)^{1/2} \ge \frac{\theta^\sigma}{2}.
	\end{equation}
	Similarly, by (\ref{cond_weakC}) and (\ref{cond_f_0}),
	\begin{equation}\label{C-1}
	\left(\fint_{B(0,1)} |u|^2\right)^{1/2}  \le 1.
	\end{equation}
	On the other hand, it follows from (\ref{ineq_Ff_Lip}) and (\ref{C-1}) that
	\begin{equation}
	\left(\fint_{B(0,\theta)} \Big| u - \fint_{B(0,\theta)} u \Big|^2 \right)^{1/2} + \theta \left( \fint_{B(0,\theta)} | u|^2 \right)^{1/2} \le C_0 2^{d/2}\theta.
	\end{equation}
	This, together with (\ref{C-0}), gives $C_0 2^{d/2}\theta \ge \theta^\sigma/2 $, 
	which is in contradiction with our choice of $\theta$. The proof is now complete.
\end{proof}

\begin{lem}\label{lem_comp_iter}
	Let $\sigma \in (0,1)$ and
	$\e_0, \theta$ be given by Lemma \ref{lem_comp_first}. If $0<\e<\e_0\theta^{k-1}$ for some $k\ge 1$
	and $u_\e$ is a weak solution of (\ref{cond_ue_Ff}) in $B(0,1)$ for some $M\in \cal{A}$ and
	$\lambda\in [0,\e_0^2]$, then
	\begin{align*}
	&\left(\fint_{B(0,\theta^k)} \Big| u_\e - \fint_{B(0,\theta^k)} u_\e \Big|^2 \right)^{1/2} + \theta^{k}\sqrt{\lambda} \left( \fint_{B(0,\theta^k)} |u_\e|^2 \right)^{1/2} \\
	&\qquad \le \theta^{k\sigma} \Bigg\{ \left(\fint_{B(0,1)} \Big| u_\e - \fint_{B(0,1)} u_\e \Big|^2 \right)^{1/2} + \sqrt{\lambda} \left(\fint_{B(0,1)} \left| u_\e \right|^2 \right)^{1/2}  +I_k+J_k \Bigg\},
		\end{align*}
	where $I_k$ and $J_k$ are defined by
	\begin{equation}\label{I-J}
	\aligned
	&I_k =\e_0^{-1}\sum_{\ell=0}^{k-1} \theta^{\ell (2-\sigma)}\left(\fint_{B(0,\theta^\ell)} |F|^2\right)^{1/2},\\
	& J_k =\varep^{-1}_0\sum_{\ell=0}^{k-1}\theta^{\ell (1-\sigma)} \left(\fint_{B(0, \theta^\ell)}
	|f|^2\right)^{1/2}.
	\endaligned
	\end{equation}
\end{lem}
\begin{proof}
	The lemma is proved by an induction argument on $k$. The case $k=1$ is given by Lemma \ref{lem_comp_first}. Suppose now that the lemma holds for some $k\ge 1$. Let $u_\e$ be a weak solution of (\ref{cond_ue_Ff}) in $B(0,1)$ for some $0<\e<\e_0 \theta^k$. Consider the function $v(x) = u_\e(\theta^k x)$. Observe that $v$ satisfies the system
	\begin{equation*}
	-\text{div}(M(x/(\theta^{-k}\e)) \nabla v) + \theta^{2k} \lambda v = G+  \text{div}(H) \quad \text{in } B(0,1),
	\end{equation*}
	where $G(x)=\theta^{2k} F(\theta^k x)$ and $H(x)=\theta^k f(\theta^k x)$.
	Since $\theta^{-k} \e < \e_0$ and $\theta^{2k} \lambda \in [0,\e_0^2]$, it follows from Lemma \ref{lem_comp_first} that
	\begin{align*}
	&\left(\fint_{B(0,\theta^{k+1})} \Big| u_\e - \fint_{B(0,\theta^{k+1})} u_\e \Big|^2 \right)^{1/2} + \theta^{k+1} \sqrt{\lambda} \left( \fint_{B(0,\theta^{k+1})} |u_\e|^2 \right)^{1/2} \\
	& = \left(\fint_{B(0,\theta)} \Big| v - \fint_{B(0,\theta)} v \Big|^2 \right)^{1/2} + \theta ( \sqrt{\theta^{2k}\lambda}) \left( \fint_{B(0,\theta)} |v|^2 \right)^{1/2} \\
	& \le \theta^\sigma \Bigg\{  \left(\fint_{B(0,1)} \Big| v - \fint_{B(0,1)} v \Big|^2 \right)^{1/2} +\sqrt{\theta^{2k} \lambda} \left(\fint_{B(0,1)} \left| v \right|^2 \right)^{1/2} \\
	&\qquad \qquad+ \e_0^{-1}\theta^{2k} \left( \fint_{B(0,1)} |F(\theta^k x)|^2 \, dx\right)^{1/2} + \e_0^{-1} \theta^k \left( \fint_{B(0,1)} |f(\theta^k x)|^2 \, dx \right)^{1/2} \Bigg\} \\
	& \le \theta^\sigma \Bigg\{  \left(\fint_{B(0,\theta^k)} \Big| u_\e - \fint_{B(0,\theta^k)} u_\e \Big|^2 \right)^{1/2} +\theta^{k}\sqrt{ \lambda} \left(\fint_{B(0,\theta^k)} \left| u_\e \right|^2 \right)^{1/2} \\
	&\qquad \qquad+ \e_0^{-1}\theta^{2k} \left( \fint_{B(0,\theta^k)} |F|^2  \right)^{1/2} 
	+ \e_0^{-1} \theta^k \left( \fint_{B(0,\theta^{k})} |f|^2  \right)^{1/2} \Bigg\} .
	\end{align*}
	By the induction assumption this is bounded by
	\begin{align*}
	 & \theta^{(k+1)\sigma} \Bigg\{  \left(\fint_{B(0,1)} \Big| u_\e - \fint_{B(0,1)} u_\e \Big|^2 \right)^{1/2} +\sqrt{ \lambda} \left(\fint_{B(0,1)} \left| u_\e \right|^2 \right)^{1/2} \\
	&\qquad \qquad + I_k+J_k + \e_0^{-1}\theta^{k(2-\sigma)}
	 \left( \fint_{B(0,\theta^k)} |F|^2\right)^{1/2} 
	 + \e_0^{-1} \theta^{k(1-\sigma)}  \left( \fint_{B(0,\theta^{k})} |f|^2 \right)^{1/2}\Bigg\} \\
	&=  \theta^{(k+1)\sigma} \Bigg\{  \left(\fint_{B(0,1)} \Big| u_\e - \fint_{B(0,1)} u_\e \Big|^2 \right)^{1/2} 
	+\sqrt{ \lambda} \left(\fint_{B(0,1)} \left| u_\e \right|^2 \right)^{1/2}  +I_{k+1} +J_{k+1} \Bigg\},
		\end{align*}
	where we have used the definitions of $I_k$ and $J_k$.
	The proof is complete.
\end{proof}

We are now ready to give the proof of Theorem \ref{thm_ue_Ff}.

\begin{proof}[\bf Proof of Theorem \ref{thm_ue_Ff}]
Let $\e_0$ and $\theta$ be given by Lemma \ref{lem_comp_first}.
We may   assume that $0\le \lambda\le \e_0^2 R^{-2}$.
The case $\varep_0^2  R^{-2}<\lambda\le R^{-2}$ follows easily 
from the case $\lambda=\e_0^{2}R^{-2}$.
	By translation and dilation we may also assume that $x_0 = 0$ and $R = 1$. 
	Thus $u_\e \in H^1(B(0,1);\R^m)$ is a weak solution of $-\text{div}(M(x/\e)\nabla u_\e) + \lambda u_\e
	 = F+\text{div} ( f)$ in $B(0,1)$ for some $M\in\cA, 0<\e<1$ and $\lambda\in [0,\e_0^2]$. 
	 Let $\e < r < 1$.  We may assume that $r<\e_0 \theta$,
	 as the case $r\ge \e_0 \theta$ follows directly from Caccioppoli's inequality.
	
	Now we choose $k\ge 1$ so that $\e_0 \theta^{k+1} \le r < \e_0 \theta^{k}$. 
	It follows from Lemma \ref{lem_comp_iter} and (\ref{ineq_Cacci}) that
	\begin{align*}
	&\left(\fint_{B(0,r)} |\nabla u_\e|^2 \right)^{1/2} + \sqrt{\lambda} \left(\fint_{B(0,r)} | u_\e|^2 \right)^{1/2} \\
	&\qquad \le C \Bigg\{ \left(\fint_{B(0,\theta^{k}/2)} |\nabla u_\e|^2 \right)^{1/2} + \sqrt{\lambda}\left(\fint_{B(0,\theta^{k}/2)} | u_\e|^2 \right)^{1/2} \Bigg\} \\
	&\qquad \le C \left\{  \theta^{-k}\left(\fint_{B(0,\theta^k)} \left| u_\e - \fint_{B(0,\theta^k)} u_\e \right|^2 \right)^{1/2} + \sqrt{\lambda} \left(\fint_{B(0,\theta^{k})} | u_\e|^2 \right)^{1/2} \right. \\
	&\qquad\qquad \qquad \left.+  \theta^k \left( \fint_{B(0,\theta^k)} |F|^2 \right)^{1/2} 
	+ \left( \fint_{B(0,\theta^k)} |f|^2 \right)^{1/2} \right\} \\
	&\qquad \le C \theta^{k(\sigma-1)}\Bigg\{  \left(\fint_{B(0,1)} \left| u_\e  \right|^2 \right)^{1/2} 
	+\theta^{k(2-\sigma)} \left( \fint_{B(0,\theta^k)} |F|^2\right)^{1/2} \\
& \qquad\qquad\qquad\qquad\qquad
	+I_k +J_k
	+\theta^{k(1-\sigma)} \left( \fint_{B(0,\theta^{k})} |f|^2 \right)^{1/2}  \Bigg\}.
		\end{align*}
		
		Finally,  note that by (\ref{I-J}),
		$$
		\aligned
		 &I_k\le C \sup_{\substack{x\in B(0,1/2) \\ r\le t\le 1/2} } t \left(\fint_{B(x,t)} |F|^2\right)^{1/2},\\
		& J_k \le C \sup_{x\in B(0, 1/2)} \left(\fint_{B(x, r)}  |f|^2\right)^{1/2}.
		\endaligned
		$$
		We obtain
		$$
		\aligned
		&\left(\fint_{B(0,r)} |\nabla u_\e|^2 \right)^{1/2} + \sqrt{\lambda} \left(\fint_{B(0,r)} | u_\e|^2 \right)^{1/2}\\
		 &\le C_\sigma  r^{\sigma-1}
		\left\{ \left(\fint_{B(0,1)} \left| u_\e  \right|^2 \right)^{1/2} 
	+\sup_{\substack{ x\in B(0,1/2)\\ r\le t\le (1/2} } t \left( \fint_{B(x,t)} |F|^2\right)^{1/2}\right.\\
&\qquad\qquad\qquad\qquad\qquad\qquad
\left.	+ \sup_{x\in B(0, 1/2)} \left(\fint_{B(x, r)}  |f|^2\right)^{1/2}\right\}.
	\endaligned
	$$
	This finishes the proof (with $1-\sigma$ in the place of $\sigma$).
\end{proof}

As a corollary of Theorem \ref{thm_ue_Ff}, we obtain the following.

\begin{thm}\label{H-theorem-global}
Suppose that $A\in APW^2(\R^d)$ and satisfies the ellipticity condition (\ref{cond_ellipticity}).
Let $F\in L^2_{\loc, \unif}(\R^d; \R^m)$, $f\in L^2_{\loc, \unif}(\R^{m\times d})$, and $u$ be the solution of
\begin{equation}\label{global-equation}
-\text{\rm div} (A\nabla u) +T^{-2} u = F +\text{\rm  div} (f) \quad \text{ in } \R^d,
\end{equation}
given by Lemma \ref{lem_appr_u}. Then there exists $\bar{q}>2$, depending only on  $d$, $m$ and $\mu$,
such that for any $1\le r\le T$ and $\sigma \in (0,1)$,
\begin{equation}\label{G-estimate-r-6}
\|\nabla u\|_{S^q_r} + T^{-1} \| u\|_{S^2_r}
\le C_\sigma \left(\frac{T}{r} \right)^{\sigma}
\left\{ \sup_{r\le t\le T}  t \| F\|_{S^2_t} + \| f\|_{S^q_r} \right\},
\end{equation}
where $2\le q\le \bar{q}$ and
$C_\sigma$ depends only on $\sigma$ and $A$.
\end{thm}

\begin{proof}
The case $(T/2)\le r\le T$ follows from Lemma \ref{lem_appr_u}
and does not use the almost periodicity of $A$.
To treat the case $1\le r< (T/2)$,
we use Theorem \ref{thm_ue_Ff} with $\e=1$, $\lambda=T^{-2}$ and $R=T$.
This, together with the reverse H\"older estimate (\ref{RH}), gives
$$
	\aligned
	&\left(\fint_{B(x_0,r)} |\nabla u|^q \right)^{1/q} + \sqrt{\lambda} \left(\fint_{B(x_0,r)} |u|^2 \right)^{1/2}\\
	& \le C_\sigma\bigg( \frac{T}{r}\bigg)^{\sigma} \Bigg\{  \frac{1}{T}\left( \fint_{B(x_0,T)} | u|^2 \right)^{1/2} + 
	\sup_{\substack{x\in B(x_0, T/2)\\ r\le t\le T/2}} t \left( \fint_{B(x,t)} |F|^2 \right)^{1/2}  \\
	&\qquad \qquad \qquad\qquad\qquad\qquad+ \sup_{x\in B(x_0,T/2)} \left( \fint_{B(x,r)} |f|^q \right)^{1/q}\Bigg\},
	\endaligned
$$
where $2\le q\le \bar{q}$ and
$\bar{q}>2$ depends only on $d$, $m$ and $\mu$. It follows that
$$
\aligned
\|\nabla u\|_{S^q_r} + T^{-1} \| u\|_{S^2_r}
&\le C_\sigma \left(\frac{T}{r}\right)^{\sigma}
\Big\{ T^{-1} \| u\|_{S^2_T} + \sup_{r\le t\le T} t \| F\|_{S^2_t} + \| f\|_{S^q_r} \Big\}\\
&\le C_\sigma \left(\frac{T}{r}\right)^{\sigma}
\Big\{\sup_{r\le t\le T}  t\| F\|_{S^2_t} +  +\| f\|_{S^2_T} 
+ \| f\|_{S^q_r} \Big\},
\endaligned
$$
which leads to (\ref{G-estimate-r-6}), using $\| f\|_{S^2_T} \le C \| f\|_{S^q_r}$.
\end{proof}

\begin{cor}\label{thm_chiT_BL}
	Suppose that $A\in APW^2(\R^d)$ and satisfies the ellipticity condition (\ref{cond_ellipticity}). 
	Let $T>1$ and $\sigma \in (0,1)$.
	Then
	\begin{equation}\label{ineq_dchi_r}
	\aligned
	 \|\nabla \chi_T\|_{S^{\bar{q}}_r}  & \le C_\sigma \left( \frac{T}{r}\right)^{\sigma},\\
	 \|\nabla \big(\chi_T -\chi_{2T}\big) \|_{S^{\bar{q}}_r}
	  & \le C _\sigma \left( \frac{T}{r}\right)^{\sigma}
	  \sup_{r\le t\le T} t \| T^{-2} \chi_{2T} \|_{S^2_t},
	  \endaligned
	\end{equation}
	for any $1\le r\le T$,
 where  $C_\sigma$ depends only on $\sigma$ and $A$. 
 \end{cor}

\begin{proof}
The first inequality in (\ref{ineq_dchi_r}) follows directly from Theorem \ref{H-theorem-global}
with $F=0$ and $f=A\nabla P_j^\beta$.
The same argument also gives a rough estimate,
\begin{equation}\label{small-cor-1}
T^{-1} \|\chi_T\|_{S^2_r} \le C_\sigma \left(\frac{T}{r}\right)^\sigma.
\end{equation}
To see the second inequality we let $u=\chi_T -\chi_{2T}$.
Then 
$$
-\text{\rm div} \big(A\nabla  u \big)
+T^{-2} u= -(3/4)T^{-2} \chi_{2T}.
$$
By Theorem \ref{H-theorem-global} we obtain the second inequality in (\ref{ineq_dchi_r}).
		\end{proof}
		
%%%%%%%%%%%%%%%%%%%%%%%%%%%%%%

%%%%%%%%%%%%%%%%%%%%%%%%%%

\section{A quantitative ergodic argument}

In this section we establish some general  estimates, which
formalize and extend the quantitative ergodic argument in \cite{AGK-2015},
for functions in $APW^2(\R^d)$.
These estimates allow us to control the norm $\| g\|_{S^2_1}$ by
$\| \nabla g\|_{S^2_t}$ for $t\ge 1$ and the function $\omega_k (g; L, R)$,
defined by 
\begin{equation}\label{omega-k}
\omega_k (g; L, R)
=\sup_{y_1\in \R^d} \inf_{|z_1|\le L} 
\cdots \sup_{y_k\in \R^d} \inf_{|z_k|\le L}
\| \Delta_{y_1z_1}\Delta_{y_2z_2} \cdots \Delta_{y_k z_k} (g)\|_{S^2_R},
\end{equation}
where $0<L, R<\infty$ and $k\ge 1$.
Throughout this section we will assume that
$g, \nabla g \in L^2_{\loc, \unif} (\R^d)$ and
\begin{equation}\label{mean}
\langle g \rangle=\lim_{R\to\infty} \fint_{B(0,R)} g =0.
\end{equation}
Let
\begin{equation}\label{u}
u(x,t)=g * \Phi_t ( x)= \int_{\R^d} g(y) \Phi_t (x-y)\, dy,
\end{equation}
where 
$$
\Phi_t (y)=t^{d/2}\Phi(y/\sqrt{t}) =c_d t^{-d/2} \exp (-|y|^2/(4 \sqrt{t}))
$$
is the standard heat kernel.

We begin with a lemma that reduces the estimate of $\| g\|_{S^2_1}$ to that of $\| u(\cdot, 1)\|_\infty$.

\begin{lem}\label{heat-lemma-3}
Let $u (x, t)=g * \Phi_t (x)$. Then, for $0<R<\infty$,
\begin{equation}\label{heat-estimate-10}
\| g\|_{S^2_R}
\le C \left\{ \| u(\cdot, R^2)\|_\infty +R\,  \| \nabla g\|_{S^2_R} \right\},
\end{equation}
where $C$ depends only on $d$.
\end{lem}

\begin{proof}
Note that if $g(x)=f(Rx)$, then
$\| g\|_{S^2_1} =\| f\|_{S^2_R}$ and $\Phi_t *g (x) =\Phi_{tR^2} *f(Rx)$.
Thus, by rescaling, we may assume that $R=1$.
We will show that for any $r\ge 1$ and $x\in \R^d$,
\begin{equation}\label{heat-3-1}
\Big| u(x, 1)-\fint_{B(x, r)} g \Big|
\le C \,  r^{\frac{d}{2}+2}
 \|\nabla g\|_{S^2_r}+ C\, e^{-c\, r^2} \| g\|_{S^2_r},
\end{equation}
where $C>0$ and $c>0$ depend only on $d$.
Assume (\ref{heat-3-1}) holds for a moment.
Then, by Poincar\'e inequality, for any $r\ge 1$,
$$
\aligned
\| g\|_{S^2_r}   &\le C\, r \|\nabla g\|_{S^2_r} +\sup_{x\in \R^d}
\Big| \fint_{B(x,r)} g \Big|\\
&\le C\, r^{\frac{d}{2}+2} \|\nabla g\|_{S^2_r}
+\| u(\cdot, 1)\|_\infty +C\,e^{-c\, r^2} \| g\|_{S^2_r}.
\endaligned
$$
We now fix $r>1$ such that $C e^{-c\, r^2}\le (1/2)$.
Since $\| g\|_{S^2_r}<\infty$,
it follows that
$$
\aligned
\|g\|_{S^2_1}
&\le C \, \| g\|_{S^2_r}
\le C \, \| u(\cdot, 1)\|_\infty
+C\, \| \nabla g\|_{S^2_r}\\
&\le C \Big\{ \| u(\cdot, 1)\|_\infty
+\| \nabla g\|_{S^2_1} \Big\}.
\endaligned
$$

It remains to prove (\ref{heat-3-1}).
To this end we first note that
$$
\aligned
\Big| u(x, 1) -\int_{B(0,r)} g(x-y) \Phi (y)\, dy \Big|
&\le \int_{\R^d\setminus B(0,r)}  |g(x-y)|\, \Phi (y)\, dy\\
&\le \sum_j \int_{Q_j} |g(x-y)| \Phi (y)\, dy\\
&\le \sum_j \left(\fint_{Q_j} |g(x-\cdot)|^2\right)^{1/2}
\left(\fint_{Q_j} |\Phi|^2\right)^{1/2}\\
&\le C \| g\|_{S^2_r} \int_{|y|\ge cr} e^{-c|y|^2}\, dy\\
&\le C\, e^{-c\, r^2} \| g\|_{S^2_r},
\endaligned
$$
where $\{ Q_j\}$ is a collection of non-overlapping cubes with side length $c\, r$
such that 
$$
\R^d\setminus B(0,r)\subset \cup_j Q_j \subset \R^d \setminus B(0,r/2).
$$
It follows that
\begin{equation}\label{heat-3-3}
\aligned
& \Big| u(x, 1) -\fint_{B(x,r)} g \Big|\\
&\le C e^{-c\, r^2} \| g\|_{S^2_r}
+\Big| \fint_{B(x,r)} g -\int_{B(x,r)} g(y)\Phi (x-y) \, dy \Big|\\
&\le C e^{-c\, r^2}\| g\|_{S^2_r}
+ \Big| \int_{B(x,r)} \left( g-\fint_{B(x,r)} g \right)
\big( \Phi (x-y) -E_r \big) \, dy \Big|\\
&\qquad\qquad
+ \Big| \fint_{B(x,r)} g \Big| \, \Big| \int_{B(0,r)} \Phi  -1 \Big|,
\endaligned
\end{equation}
where $E_r$ is the average of $\Phi$ over $B(0,r)$.
By H\"older's and Poincar\'e  inequalities the second term in the r.h.s. of (\ref{heat-3-3}) is 
bounded by 
$$
C r^{d+2} \|\nabla g\|_{S^2_r} \left(\fint_{B(0,r)} |\nabla \Phi|^2\right)^{1/2}
\le C r^{\frac{d}{2} +2} \|\nabla g\|_{S^2_r}.
$$
Finally,  since $\int_{\R^d} \Phi =1$,
the last term in the r.h.s. of (\ref{heat-3-3}) is bounded by
$$
C\,  \| g\|_{S^2_r} \int_{\R^d\setminus B(0,r)} \Phi
\le Ce^{-c\, r^2} \| g\|_{S^2_r}.
$$
This completes the proof of (\ref{heat-3-1}).
\end{proof}

To control $\| u(\cdot, t)\|_\infty$, we use a quantitative ergodic result
from \cite{AGK-2015}. We mention  that the explicit dependence of constants in $k$
is not used in this paper.

\begin{lem}\label{heat-lemma-0}
Let $u(x,t)= g* \Phi_t (x)$, where $g, \nabla g\in L^2_{\loc, \unif} (\R^d)$ and
$\langle g \rangle=0$.
Then, for any $t\ge  k R^2$ and $0<L<\infty$,
\begin{equation}\label{heat-estimate-1}
\aligned
\| u(\cdot, t)\|_\infty
& \le C^k \left\{ \omega_k (g; L, R) + \exp\left(-\frac{c\,  t}{k L^2}\right) \| g\|_{S^2_R}\right\},\\
\|\nabla_x u (\cdot , t)\|_\infty
& \le \frac{C^k}{\sqrt{t}}
 \left\{ \omega_k (g; L, R) + \exp\left({-\frac{c \, t}{kL^2}}\right) \| g\|_{S^2_R} \right\},
\endaligned
\end{equation}
where $C$ and $c$ depend only on $d$.
\end{lem}

\begin{proof}
By rescaling we may reduce the general case  to the case
where $R=1$ and $t\ge k$. In this case
the proposition was proved in \cite{AGK-2015}.
We point  out that the condition (\ref{mean}), together with the assumption that $g\in L^2_{\loc, \unif}(\R^d)$,
implies 
$
 \fint_{B(x,R)} g \to 0,
$
as $R\to \infty$,
for any $x\in \R^d$. It follows by the Lebesgue dominated convergence theorem that
$\fint_{B(0, R)} u(x, t)\, dx \to 0$, as $R\to \infty$, for any  $t>0$.
Hence, $\| u(\cdot, t)\|_\infty \le \sup_{x,y\in \R^d} |u(x,t)-u(y,t)|$.
\end{proof}

We are ready to state and prove the main result of this section.

\begin{thm}\label{general-theorem}
Let $g\in H^1_{\loc} (\R^d)$.
Suppose that $g, \nabla g\in L^2_{\loc, \unif} (\R^d)$ and
$\langle g \rangle=0$.
Then, for any $T\ge 2$ and $k\ge 1$,
\begin{equation}\label{general-estimate}
\aligned
\| g\|_{S^2_1}
 &\le C  \inf_{1\le L \le T}
 \left\{ \omega_k (g; L, T) +\exp \left(-\frac{c\, T^2}{L^2}\right) \| g\|_{S^2_T}  \right\}\\
 &\qquad
 + C \int_1^{T} \inf_{1\le L\le t}
 \left\{ \omega_k (\nabla g; L, t)
 +\exp\left(-\frac{c\, t^2}{L^2} \right)  \|\nabla g\|_{S^2_t}
 \right\}\, dt,
 \endaligned
\end{equation}
where $C>0$ and $c>0$ depend only on $d$ and $k$.
\end{thm}

\begin{proof}
We first note that  $\|\nabla g\|_{S^2_1}$ is bounded by
the second integral in the r.h.s. of (\ref{general-estimate})
over the interval $[1,2]$.
Thus, in view of Lemma \ref{heat-lemma-3},
it suffices to show that $\| u(\cdot, 1)\|_\infty$ is dominated by the r.h.s. of (\ref{general-estimate}),
where $u(x,t)=g * \Phi_t (x)$.
To this end we use the heat equation $\partial_t u=\Delta_x u$ to obtain 
\begin{equation}\label{g-1}
\aligned
\| u(\cdot, 1)\|_\infty
&\le \| u(\cdot, T^2)\|_\infty +\int_1^{T^2} \|\partial_s u (\cdot, s)\|_\infty\, ds\\
&\le \| u(\cdot, T^2)\|_\infty +\int_1^{T^2} \| \nabla^2_x  u(\cdot, s)\|_\infty\, ds.
\endaligned
\end{equation}
By the first inequality in (\ref{heat-estimate-1}) with $R=c\, T$,
\begin{equation}\label{g-2}
 \| u(\cdot, T^2)\|_\infty \le 
C  \inf_{1\le L \le T }
 \left\{ \omega_k (g; L, T) +\exp \left(-\frac{c\, T^2}{L^2}\right) \| g\|_{S^2_T}  \right\}.
 \end{equation}
To handle $\|\nabla_x^2 u(\cdot, s)\|_\infty$, we use
$$
\nabla_x^2 u=\nabla_x \big(\nabla g * \Phi_t\big)
$$
and the second inequality in (\ref{heat-estimate-1})  with $R=c\sqrt{s}$
to obtain 
\begin{equation}\label{g-3}
\|\nabla^2 u(\cdot, s)\|_\infty
\le  
C \inf_{1\le L \le \sqrt{s}}
\left\{ \omega_k (\nabla g; L, \sqrt{s})
 +\exp\left(-\frac{c\, s}{L^2} \right)  \|\nabla g\|_{S^2_{\sqrt{s}}}
 \right\}\frac{1}{\sqrt{s}}.
\end{equation}
The estimate (\ref{general-estimate})
follows by combining (\ref{g-1}), (\ref{g-2}) and (\ref{g-3}) and using 
a change of variable $t=\sqrt{s}$ in the integral.
\end{proof}

\begin{rmk}\label{g-remark}
{\rm
Suppose that $g\in APW^2(\R^d)$ and $\langle g \rangle =0$.
Then $\omega_k (g, L, L)\to 0$ as $L\to \infty$.
It follows that the first term in (\ref{general-estimate})
goes to zero as $ T\to \infty$. This gives 
\begin{equation}\label{g-estimate-2}
\| g\|_{S^2_1}
 \le C \int_1^\infty \inf_{1\le L\le t}
 \left\{ \omega_k (\nabla g; L, t)
 +\exp\left(-\frac{c\, t^2}{L^2} \right)  \|\nabla g\|_{S^2_t}
 \right\}\, dt ,
\end{equation}
where $C>0$ and $c>0$ depend only on $d$ and $k$.
}
\end{rmk}

%%%%%%%%%%%%%%%%%%%%%%%%%%%%%%%%%%%%%%%%%%%%%%%%%%%%%%%%%

%%%%%%%%%%%%%%%%%%%%%%%%%%%%%%%%

\section{Estimates of approximate correctors, part II}

In this section we give the proof of Theorems \ref{main-theorem-1} and \ref{main-theorem-2}.
Let $$
 P=P_k=\big\{ (y_1, z_1),  (y_2, z_2), \dots, (y_k, z_k)\big\},
 $$
  where $(y_i, z_i)\in \R^d\times \R^d$. Recall that
\begin{equation}\label{P}
\Delta_P  (f) =\Delta_{y_1z_1} \Delta_{y_2z_2} \cdots \Delta_{y_k z_k} (f)
\end{equation}
(if $k=0$, then $P=\emptyset$ and $\Delta_P (f)=f$).
Using the observation that
\begin{equation}\label{product-rule}
\Delta_{yz} (fg) (x)
=\Delta_{yz} (f) (x) \cdot g (x+y)
+ f(x+z) \cdot \Delta_{yz} (g) (x),
\end{equation}
an induction argument yields
\begin{equation}\label{product-rule-1}
\Delta_P (fg) (x)
=\sum_{Q\subset P}
\Delta_Q (f) (x +z_{j_1} +\cdots +z_{j_t}) \cdot
\Delta_{P\setminus Q} (g) (x +y_{i_1} +\cdots +y_{j_{i_\ell}}),
\end{equation}
where the sum is taken over all $2^k$ subsets 
$Q=\big\{ (y_{i_1}, z_{i_1}), \dots, (y_{i_\ell}, z_{i_\ell})\big\}$ of $P$,
with $P\setminus Q=\big\{(y_{j_1,}, z_{j_1}), \dots, (y_{j_t}, z_{j_t}) \big\}$. Here,
$i_1<i_2<\cdots i_\ell$, $j_1<j_2<\dots <j_t$, and $\ell +t=k$.
It follows from (\ref{product-rule-1}) by H\"older's inequality that
\begin{equation}\label{product-rule-2}
\| \Delta_P (fg)\|_{S^{q_1}_R}
\le \sum _{Q\subset P} \|\Delta_Q (f)\|_{S^p_R} \|\Delta_{P\setminus Q} (g)\|_{S^q_R},
\end{equation}
where 
$\frac{1}{q_1} \ge \frac{1}{p} +\frac{1}{q}$.

\begin{lem}\label{Higher-order-theorem}
Suppose that $A\in APW^2(\R^d)$ and satisfies the condition (\ref{cond_ellipticity}).
Let $F\in L^2_{\loc, \unif}(\R^d; \R^m)$, $f\in L^2_{\loc, \unif}(\R^{m\times d})$, and $u$ be the solution of
(\ref{global-equation}), 
given by Lemma \ref{lem_appr_u}.  Let $k\ge 0$ and $P=P_k$.
Then
there exists $\bar{q}>2$, depending only on $d$, $m$ and $\mu$, 
such that  any $1\le r\le T$ and $\sigma \in (0,1)$,
\begin{equation}\label{G-estimate-r}
\aligned
&\|\Delta_P (\nabla u)\|_{S^q_r} + T^{-1} \| \Delta_P (u) \|_{S^2_r}\\
 &\le C_\sigma \left(\frac{T}{r} \right)^{\sigma}
 \left\{  \sup_{r\le t\le T} t \| \Delta_{P} (F)\|_{S^2_t} + \| \Delta_{P} (f) \|_{S^{{q}_0}_r} \right\}\\
& \qquad  +C_\sigma \left(\frac{T}{r} \right)^{\sigma}
\sum_{P=Q_0\cup Q_1\cup\cdots \cup Q_\ell}
\|\Delta_{Q_1} A\|_{S^p_r}
\cdots \|\Delta_{Q_\ell} A \|_{S^p_r}\\
&\qquad\qquad\qquad\qquad
\cdot
\left\{ \sup_{r\le t\le T}
t \| \Delta_{Q_0} (F)\|_{S^2_t} + \| \Delta_{Q_0} (f) \|_{S^{{q}_0}_r} \right\},
\endaligned
\end{equation}
where $2\le q< q_0\le \bar{q}$, $\frac{1}{q}-\frac{1}{q_0} \ge \frac{k}{p}$,
and $C_\sigma$ depends only on $d$, $m$, $k$, $\sigma$ and $A$.
The sum in (\ref{G-estimate-r}) is taken over all partitions of $P=Q_0\cup Q_1\cup \cdots \cup Q_\ell$
with $1\le \ell \le k-1$ and $Q_j\neq \emptyset$.
\end{lem}

\begin{proof}
Let $\bar{q}>2$ be the same as in Theorem \ref{H-theorem-global}.
We prove the estimate (\ref{G-estimate-r})
by an induction argument on $k$.
Note that the case $k=0$ with $P=\emptyset$  is given by Corollary \ref{H-theorem-global}.
Let $k\ge 1$ and
suppose the estimate (\ref{G-estimate-r}) holds for  $P=P_\ell$ with $0\le \ell \le k-1$.
Let $2\le q<q_0\le \bar{q}$ and $\frac{1}{q}-\frac{1}{q_0} \ge \frac{k}{p}$.
By applying $\Delta_P $ to the system (\ref{global-equation}) and using (\ref{product-rule-1}),
we obtain 
$$
\aligned
& -\text{\rm div} \big( A(\cdot + z_1+\cdots z_k) \nabla \Delta_P (u)\big)
+T^{-2} \Delta_P (u)\\
&=\Delta_P(F)
+\text{\rm div} \big( \Delta_P(f))
+\text{\rm div} \Big( 
\sum_{Q\subset P, Q\neq \emptyset}
\Delta_Q (A) \cdot \Delta_{P\setminus Q} (\nabla u) \Big).
\endaligned
$$
It follows from  Theorem \ref{H-theorem-global} and H\"older's inequality  that
\begin{equation}\label{G-r-100}
\aligned
&\| \Delta_P (\nabla u) \|_{S^q_r} +T^{-1} \|\Delta_P (u)\|_{S^2_r}\\
&\le C_\sigma \left(\frac{T}{r} \right)^{\frac{\sigma}{2}}
\Bigg\{ \sup_{r\le t\le T}
t \|\Delta_P (F)\|_{S^2_t}
+ \|\Delta_P (f)\|_{S^q_r}\\
&\qquad\qquad\qquad\qquad
+\sum_{Q\subset P, Q\neq \emptyset}
\|\Delta_Q (A)\|_{S^p_r} \|\Delta_{P\setminus Q} (\nabla u)\|_{S^{q_1}_r} \Bigg\},
\endaligned
\end{equation}
where $q_1$ is chosen so that $2\le q<q_1<q_0\le  \bar{q}$, $\frac{1}{q}-\frac{1}{q_1} \ge \frac{1}{p}$ and 
$\frac{1}{q_1}-\frac{1}{q_0} \ge \frac{k-1}{p}$.
By  the induction assumption,
\begin{equation}\label{G-r-200}
\aligned
&\|\Delta_{P\setminus Q} (\nabla u)\|_{S^{q_1}_r} + T^{-1} \| \Delta_{P\setminus Q} (u) \|_{S^2_r}\\
 &\le C_\sigma \left(\frac{T}{r} \right)^{\frac{\sigma}{2}}
 \left\{  \sup_{r\le t\le T} t \| \Delta_{P\setminus Q} (F)\|_{S^2_t} 
 + \| \Delta_{P\setminus Q} (f) \|_{S^{{q}_0}_r} \right\}\\
&\quad
+ C_\sigma \left(\frac{T}{r} \right)^{\frac{\sigma}{2}}
\sum_{P\setminus Q=Q_0\cup Q_1\cup\cdots \cup Q_\ell}
\|\Delta_{Q_1} A\|_{S^p_r}
\cdots \|\Delta_{Q_\ell} A \|_{S^p_r}\\
&\qquad\qquad\qquad\qquad\qquad
\cdot
\left\{  \sup_{r\le t\le T}
t \| \Delta_{Q_0} (F)\|_{S^2_t} + \| \Delta_{Q_0} (f) \|_{S^{{q}_0}_r} \right\}.
\endaligned
\end{equation}
The desired estimate now follows by combining (\ref{G-r-100}) and
(\ref{G-r-200}).
\end{proof}

\begin{rmk}\label{higher-order-remark-1}
{\rm
If $r\ge T$, the argument in the proof of Lemma \ref{Higher-order-theorem},
together with the estimate in Lemma \ref{W-B-lemma}, gives
\begin{equation}\label{higher-order-large-estimate}
\aligned
&\|\Delta_P (\nabla u)\|_{S^q_r} + T^{-1} \| \Delta_P (u) \|_{S^2_r}\\
 &\le C
 \Big\{ T \| \Delta_{P}( F) \|_{S^2_r} + \| \Delta_{P} (f) \|_{S^{{q}_0}_r} \Big\}\\
& \qquad  +C
\sum_{P=Q_0\cup Q_1\cup\cdots \cup Q_\ell}
\|\Delta_{Q_1} A\|_{S^p_r}
\cdots \|\Delta_{Q_\ell} A \|_{S^p_r}\\
&\qquad\qquad\qquad\qquad
\cdot
\Big\{ T \| \Delta_{Q_0}( F) \|_{S^2_r} +  \| \Delta_{Q_0} (f) \|_{S^{{q}_0}_r} \Big\},
\endaligned
\end{equation}
}
where $2\le q<q_0\le \bar{q}$, $\frac{1}{q}-\frac{1}{q_0} \ge \frac{k}{p}$, and
$C$ depends only on $d$, $m$ and $\mu$.
\end{rmk}

Let $\rho_k (L,R)$ be the function defined by (\ref{rho-k}).

\begin{lem}\label{rho-k-lemma}
Suppose that $A\in APW^2(\R^d)$ and satisfies  the  condition (\ref{cond_ellipticity}).
Let $T\ge 1$.
Then, for any $\sigma \in (0,1)$ and $k\ge 1$,
\begin{equation}\label{omega-k-estimate}
 \omega_k (\nabla \chi_T; L, R)
 +\omega_k (T^{-1}\chi_T; L, R)
   \le C_\sigma \left(\frac{T}{R} \right)^\sigma
 \rho_k  (L, R),
 \end{equation}
 where $1\le R\le T$, $0<L<\infty$, and
 $C_\sigma$ depends only on $\sigma$, $k$, and $A$.
 If $R\ge T$, we have
 \begin{equation}\label{omega-k-estimate-large}
 \omega_k (\nabla \chi_T; L, R)
 +\omega_k (T^{-1}\chi_T; L, R)
   \le C\,
 \rho_k (L, R),
 \end{equation}
 where $C$ depends only on $d$, $m$ and $\mu$.
\end{lem}

\begin{proof}
In view of the definition of $\chi_T$, estimates (\ref{omega-k-estimate})
and (\ref{omega-k-estimate-large}) follow
 directly from Lemma \ref{Higher-order-theorem} and Remark \ref{higher-order-remark-1},
 respectively, with 
$q=2$, $q_0=\bar{q}$, and $\frac{k}{p}=\frac{1}{2}-\frac{1}{\bar{q}}$.
\end{proof}

We are now in a position to give the proof of Theorems \ref{main-theorem-1}
and \ref{main-theorem-2}.

\begin{proof}[\bf Proof of Theorem \ref{main-theorem-1}]
We use Theorem \ref{general-theorem} with $g=\chi_T$ and $t=T^2$.
Note that by (\ref{cor-L-2}) and (\ref{ineq_dchi_r}), 
$$
\| g\|_{S^2_T}  \le C\, T,
$$
and for any $\sigma \in (0,1)$,
$$
\|\nabla g\|_{S^2_r}  \le C_\sigma \left(\frac{T}{r}\right)^\sigma, 
$$
for $1\le r\le T$. In particular, we have
$\| \nabla \chi_T\|_{S^2_1} \le C_\sigma T^\sigma$.
Also, by Lemma \ref{rho-k-lemma},
$$
T^{-1}\omega_k (g; L, T) \le C\, \rho_k (L, T),
$$
and for any $\sigma \in (0,1)$ and $1\le t\le T$,
$$
\omega_k (\nabla g; L, t)  \le C_\sigma \left(\frac{T}{t}\right)^\sigma \rho_k (L, t).
$$
It follows by Theorem \ref{general-theorem} that
\begin{equation}\label{main-1-100}
\aligned
\|\chi_T\|_{S^2_1}
&\le C\, T \inf_{1\le L\le T} \left\{ \rho_k (L, T)
+\exp\left(-\frac{c\, T^2}{L^2}\right) \right\}\\
&\qquad + C\int_1^{T}\inf_{1\le L\le {t}}
\left\{ \rho_k (L, {t}) +\exp\left(-\frac{c\, t^2}{L^2} \right)\right\}
 \left(\frac{T}{{t}}\right)^\sigma dt.
\endaligned
\end{equation}
It is not hard to see that   the first term in the r.h.s. of (\ref{main-1-100})
is bounded by the integral in (\ref{main-1-100}) from $(T/2)$ to $T$.
As a result, the estimate (\ref{main-estimate-1}) follows.
\end{proof}

\begin{rmk}\label{power-decay-remark}
{\rm 
Suppose that there exist some $k\ge 1$, $0<\alpha\le 1$ and $C>0$ such that
\begin{equation}\label{power-decay-condition}
\rho_k  (L, L)\le {C}{L^{-\alpha}}\quad \text{ for any } L\ge 1.
\end{equation}
By choosing $L=t^\delta$ for some $\delta\in (0,1)$, we see that
$$
\inf_{1\le L\le t}
\left\{ \rho_k (L, L)
+\exp\left(-\frac{c\, t^2}{L^2}\right) \right\}
\le \frac{C}{t^{\delta \alpha}}.
$$
It  follows from (\ref{main-estimate-1}) that
$\|\chi_T\|_{S^2_1}\le C \, T^{1-\delta \alpha}$.
Since $\delta\in (0,1)$ is arbitrary, we obtain 
\begin{equation}\label{power-decay-2}
\| \chi_T\|_{S^2_1} \le C_\beta \,T^{1-\beta},
\end{equation}
for any $\beta\in (0, \alpha)$ and $T\ge 1$.
}
\end{rmk}

\begin{proof}[\bf Proof of Theorem \ref{main-theorem-2}]
Suppose that there exist $k\ge 1$, $\delta>0$ and $C>0$ such that
\begin{equation}\label{decay-fast}
\rho_k (L, L) \le {C}{L^{-1-\delta}}\quad \text{ for any } L\ge 1.
\end{equation}
It follows by Remark \ref{power-decay-remark} that
$\|\chi_T\|_{S^2_1}\le C_\sigma T^\sigma$ for any $T\ge 1$ and $\sigma \in (0,1)$.
Let $g=\chi_T -\chi_{2T}$. Note that by Corollary \ref{thm_chiT_BL},
\begin{equation}\label{m-0}
\|\nabla g\|_{S^2_1}\le C_\sigma \, T^{\sigma -1}\quad \text{ for any } T\ge 1 \text{ and } \sigma \in (0,1).
\end{equation}
We will show that there exists some $\beta>0$ such that
\begin{equation}\label{m-1}
\|g\|_{S^2_1} \le C\, T^{-\beta} \quad \text{ for any } T\ge 1.
\end{equation}
This would imply that $\{ \chi_{2^j}, j=1,2, \dots, \}$ is a Cauchy sequence in 
the Banach space $S^2_1 =\{ F\in L^2_{\loc} (\R^d):\, \| F\|_{S^2_1}<\infty\}$.
Let $\chi$ be the limit of $\chi_{2^j}$ in $S^2_1$.
It is easy to see that $\|\chi\|_{S^2_1} +\|\nabla \chi\|_{S^2_1}\le C$.
Since $\chi_T\in APW^2(\R^d)$ and
$\|g\|_{W^2}\le \| g\|_{S^2_1}$, we also obtain $\chi \in APW^2(\R^d)$.
Note that (\ref{m-1}) also gives $\|\chi_T\|_{S^2_1}\le C$.

To see (\ref{m-1}), we let $u(x,t)=g* \Phi_t (x)$.
In view of Lemma \ref{heat-lemma-3} and (\ref{m-0}),
it suffices to show that
\begin{equation}\label{m-40}
\| u(\cdot, 1)\|_\infty \le C T^{-\beta}
\end{equation}
for some $\beta>0$.
To this end we note that since $g\in APW^2(\R^d)$,
$\omega_1 (g; L, L) \to 0$ as $L\to \infty$.
It follows by Lemma \ref{heat-lemma-0} that $\| u(\cdot, t)\|_\infty \to 0$ as $t\to \infty$.
Thus, as in the proof of Theorem \ref{general-theorem},
\begin{equation}\label{m-4}
\aligned \| u(\cdot, 1)\|_\infty
&\le \int_1^\infty \|\partial_t u(\cdot, t)\|_\infty\, dt\\
 &\le  \int_1^\infty \|\nabla_x \big(\nabla g *\Phi_t\big)\|_\infty\, dt\\
&\le C t_0 \|\nabla g\|_{S^2_1} 
+\int_{t_0^2} ^\infty
\|\nabla_x (\nabla g *\Phi_t)\|_\infty\, dt\\
&\le C t_0 T^{\sigma-1}
+\int_{t_0^2} ^\infty
\Big\{ \|\nabla_x (\nabla \chi_T *\Phi_t)\|_\infty
+
\|\nabla_x (\nabla \chi_{2T} *\Phi_t)\|_\infty\Big\} \, dt,
\endaligned
\end{equation}
where $t_0>1$ is to be chosen and we have used the estimate 
$$
 \|\nabla_x (\nabla g *\Phi_t)\|_\infty
\le C t^{-1/2} \|\nabla g\|_{S^2_1}
$$
 for the third inequality and (\ref{m-0}) for the fourth.
 
 As in the proof of Theorem \ref{main-theorem-1},
 the  integral in the r.h.s. of (\ref{m-4}) is bounded by
 $$
 \aligned
 & C T^\sigma \int_{t_0^2}^\infty
 \inf_{1\le L\le \sqrt{t}}
 \left\{ \rho_k (L, \sqrt{t}) +\exp \left( -\frac{c\, t}{L^2}\right) \right\}
 \frac{dt}{\sqrt{t}}\\
& \qquad \le C T^\sigma \int_{t_0}^\infty
 \inf_{1\le L\le t}
 \left\{ \rho_k (L, {t}) +\exp \left( -\frac{c\, t^2}{L^2}\right) \right\}\, dt\\
 & \qquad \le C T^\sigma \int_{t_0 }^\infty
 \inf_{1\le L\le t}
 \left\{ \frac{1}{L^{1+\delta}} +\exp \left( -\frac{c\, t^2}{L^2}\right) \right\}\, dt,
\endaligned
 $$
 where we have used the condition (\ref{decay-fast}) for the last step.
 By choosing $L=t^\alpha$ for $\alpha\in (0,1)$, it follows that
  the integral in the r.h.s. of (\ref{m-4}) is bounded by
  $CT^\sigma t_0^{1-(1+\delta)\alpha}$.
  As a result, we have proved that
  $$
  \| u(\cdot, 1)\|_\infty
  \le C\,  t_0 T^{\sigma -1}
  +C \, T^\sigma t_0^{1-(1+\delta)\alpha}
  =C\, t_0 T^\sigma \Big\{ T^{-1} + t_0^{-(1+\delta)\alpha} \Big\}.
  $$
  Finally, we choose $t_0>1$ such that $t_0^{(1+\delta)\alpha} =T$.
  This gives
  $$
  \| u(\cdot, 1)\|_\infty
  \le C \, t_0 T^{\sigma -1}
  =C\, T^{\sigma -1 +\frac{1}{(1+\delta)\alpha}} =C\, T^{-\beta},
  $$
  where
  $$
  \beta=1-\sigma -\frac{1}{(1+\delta) \alpha} >0,
  $$
  if $\sigma >0$ is small and $\alpha $ is close to $1$.
  This completes the proof.
\end{proof}

%%%%%%%%%%%%%%%%%%%%%%%%%%%%%%%

%%%%%%%%%%%%%%%%%%%%%%%%%%%%%%%

%%%%%%%%%%%%%%%%%%%%%%%%%%%%%%%%%%%%%%%%

\section{Estimates of dual approximate correctors}

Let $\chi_T=(\chi_{T, j}^{\alpha\beta})$ be the approximate correctors defined by (\ref{cond_eq_corrector}). 
For $1\le i,j\le d$ and $1\le \alpha,\beta \le m$, let
$b_T=A+A\nabla \chi_T -\widehat{A}=\big(b_{T,ij}^{\alpha\beta}\big)$ with
\begin{equation}\label{cond_bT}
b^{\alpha\beta}_{T,ij}(y) = a^{\alpha\beta}_{ij}(y) + a^{\alpha\gamma}_{ik} (y)\frac{\partial}{\partial y_k} \big( \chi^{\gamma\beta}_{T,j}(y)\big) - \widehat{a}^{\alpha\beta}_{ij}.
\end{equation}
To establish the convergence rates in Theorem \ref{main-theorem-3},  
as in \cite{Shen-2015}, we introduce 
the matrix-valued function $\phi_T =(\phi_{T, ij}^{\alpha\beta})$, called the dual approximate correctors
 and defined by the following auxiliary equations:
\begin{equation}\label{cond_phiT}
-\Delta \phi^{\alpha\beta}_{T,ij} + T^{-2} \phi^{\alpha\beta}_{T,ij} = b^{\alpha\beta}_{T,ij} - \ag{b^{\alpha\beta}_{T,ij}},
\end{equation}
where $\phi^{\alpha\beta}_{T,ij} \in H^1_{\text{loc}}(\R^d)$ are the weak solutions given by Lemma \ref{lem_appr_u}.
 In this section we establish  the uniform local $L^2$ estimates for $\phi_T$ and its derivatives.
 
 Throughout the section we assume that
$A\in APW^2(\R^d)$ and satisfies the ellipticity condition (\ref{cond_ellipticity}).
It follows that $\nabla\chi_T\in APW^2(\R^d)$ and thus $b_T \in APW^2(\R^d)$. Moreover, by (\ref{ineq_dchi_r}),
  for any $\sigma\in (0,1)$ and $1\le R\le T$,
\begin{equation}\label{ineq_bT_BT}
\| b_T\|_{S^2_R} \le C_\sigma \left(\frac{T}{R}\right)^{\sigma},
\end{equation}
where $C_\sigma$ depends only on $\sigma$ and $A$.

\begin{lem}\label{B-lemma-0}
Let $k\ge 1$ and $\sigma \in  (0,1)$.
Then, for $0<L<\infty$ and $1\le R\le T$,
\begin{equation}\label{B-estimate-0}
\omega_k (b_T; L, R) \le C_\sigma
\left(\frac{T}{R}\right)^\sigma \rho_k (L, R),
\end{equation}
where $C_\sigma$ depends only on $\sigma$, $k$ and $A$.
\end{lem}

\begin{proof}
Let $\frac{k}{p}=\frac{1}{2} -\frac{1}{\bar{q}}$, where $\bar{q}>2$ is given by (\ref{RH}).
Note that
$$
\Delta_P (b_T)
=\Delta_P (A) +\sum_{Q\subset P}
\Delta_{ Q} (A) \cdot \Delta_{ P\setminus Q}  (\nabla \chi_T).
$$
It follows by H\"older's inequality that if $\frac{k-1}{p}+\frac{1}{q}=\frac{1}{2}$,
$$
\aligned
\|\Delta_P (b_T)\|_{S^2_R}
 &\le \| \Delta_P (A) \|_{S^2_R}+ \| A\|_\infty \|\Delta_P (\nabla \chi_T)\|_{S^2_R}\\
 & \qquad 
 + \sum_{Q\subset P, Q\neq \emptyset}
\| \Delta_Q (A)\|_{S^p_R}  \| \Delta_{P\setminus Q}  (\nabla \chi_T)\|_{S^q_R}\\
& \le C_\sigma \left(\frac{T}{R}\right)^\sigma 
\sum_{Q_1\cup Q_2 \cup \cdots \cup Q_\ell=P} 
\|\Delta_{Q_1} A\|_{S^p_R} \|\Delta_{Q_2} (A)\|_{S^p_R}\cdots \| \Delta_{Q_\ell} (A)\|_{S^p_R},
\endaligned
$$
where we have used Lemma \ref{Higher-order-theorem}  with $q_0=2$ for the last step.
By applying
$$
\sup_{y_1\in \R^d} \inf_{|z_1|\le L} \cdots \sup_{y_k\in \R^d} \inf_{|z_k|\le L}
$$
to the both sides of the inequality above,
we obtain (\ref{B-estimate-0}).
\end{proof}

\begin{lem}\label{lem_regu_Mor}
	Assume $F\in L^2_{\text{\loc, \unif}}(\R^d)$. Let $u\in H_{\text{\loc}}^2(\R^d)$ be the weak solution of
	\begin{equation}\label{cond_u_f}
	-\Delta u + T^{-2} u = F \quad \text{ in } \R^d,
	\end{equation}
	given by Lemma \ref{lem_appr_u}. Then for any $0<R<\infty$,
	\begin{equation}\label{ineq_ddu_Br}
	T^{-1} \| \nabla u\|_{S^2_R} +T^{-2}\| u\|_{S^2_R}
	\le C\,  \| F \|_{S^2_R},
		\end{equation}
	where $C$ depends only on $d$. Furthermore, 
	\begin{equation}\label{second-d-estimate}
	\| \nabla^2 u\|_{S^2_R} \le C\, \log \left(2+\frac{T}{R} \right) \| F\|_{S^2_R}.
	\end{equation}
\end{lem}

\begin{proof}
By rescaling we may assume that $T=1$. We may also assume that $d\ge 3$.
The case $d=2$ may be handled by the method of descending 
(introducing a dummy variable and considering the equation in 
$\R^3$).

To show (\ref{ineq_ddu_Br}), we write 
$$
u(x)=\int_{\R^d} \Gamma (y) F(x-y)\, dy,
$$
where $\Gamma (x)$ denotes the fundamental solution for the operator $-\Delta +1$ in 
$\R^d$, with pole at the origin. Using Minkowski's inequality, we see that
$$
\| u\|_{S^2_R} \le \int_{\R^d} |\Gamma(y)| \| F\|_{S^2_R}\, dy
\le C \, \| F\|_{S^2_R},
$$
where the last inequality follows from the estimate $|\Gamma (x)|\le C\, |x|^{2-d} e^{-c|x|}$.
Similarly, by using the estimate $|\nabla \Gamma (x)|\le C |x|^{1-d} e^{-c|x|}$, we obtain 
$$
\| \nabla u\|_{S^2_R} \le \int_{\R^d} |\nabla \Gamma(y)| \| F\|_{S^2_R}\, dy
\le C \, \| F\|_{S^2_R}.
$$

Finally, to see (\ref{second-d-estimate}), we fix $B=B(x_0, R)$ and choose $\varphi \in C_0^1(3B)$
such that $\varphi=1$ in $2B$.
Write
$$
\aligned
u(x) &=\int_{\R^d} \Gamma (x-y) \varphi (y) F(y)\, dy
+\int_{\R^d} \Gamma (x-y) (1-\varphi (y)) F(y)\, dy\\
&=u_1 (x) + u_2 (x).
\endaligned
$$
By the well known singular integral estimates,
\begin{equation}\label{8-1}
\left(\fint_{B} |\nabla^2 u_1|^2\right)^{1/2} 
\le C \left(\fint_{3B} |F|^2\right)^{1/2} \le C \| F\|_{S^2_R}.
\end{equation}
Using the estimate $|\nabla^2 \Gamma (x)|\le C |x|^{-d} e^{-c|x|}$, we obtain that, for any $x\in B$,
$$
\aligned
|\nabla^2 u_2 (x)|
 &\le  C\int_{(2B)^c} |y-x_0|^{-d} e^{-c |y-x_0|} |F(y)|\, dy\\
 & \le C \sum_{j=1}^\infty e^{-c 2^j R} \fint_{|y-x_0|\le 2^j R} |F(y)|\, dy\\
 &\le C \| F\|_{S^2_R} \sum_{j=1}^\infty e^{-c 2^j R}\\
 &\le C \| F\|_{S^2_R} \log (2+R^{-1}).
 \endaligned
 $$
This, together with (\ref{8-1}), gives
$$
\left(\fint_B |\nabla^2 u|^2\right)^{1/2} \le C \|F\|_{S^2_R} \log (2+R^{-1})
$$
for any $B=B(x_0, R)$.
The estimate (\ref{second-d-estimate}) now follows.
\end{proof}

\begin{rmk}\label{B-remark-1}
{\rm
It follows from Lemma \ref{lem_regu_Mor} and estimate (\ref{ineq_bT_BT}) that
for $1\le R\le T$ and $\sigma\in (0,1)$, 
\begin{equation}\label{B-estimate-1}
T^{-2} \|\phi_T\|_{S^2_R}
+T^{-1} \|\nabla \phi_T\|_{S^2_R}
+\|\nabla^2 \phi_T\|_{S^2_R}
\le C_\sigma \left(\frac{T}{R}\right)^\sigma,
\end{equation}
where $C_\sigma$ depends only on $\sigma$ and $A$.
}
\end{rmk}

\begin{lem}\label{B-lemma-2}
Let $k\ge  1$ and $\sigma \in (0,1)$.
Then, for $ 1\le R\le T$ and $0<L<\infty$,
\begin{equation}\label{B-estimate-2}
\aligned
& T^{-2}\omega_k (\phi_T; L, R)
+ T^{-1}\omega_k (\nabla \phi_T; L, R)
+\omega_k (\nabla^2 \phi_T; L, R)\\
&\qquad\qquad\qquad
\le C_\sigma \left(\frac{T}{R}\right)^\sigma 
\rho_k (L, R),
\endaligned
\end{equation}
where $C_\sigma$ depends only on $\sigma$, $k$ and $A$.
\end{lem}

\begin{proof}
Let $u=\phi_T$. Since the difference operator $\Delta_P$ commutes with $\Delta$,
in view of Lemma \ref{lem_regu_Mor}, we have
$$
T^{-2}\|\Delta_P (u)\|_{S^2_R} +T^{-1} \|\Delta_P (\nabla u)\|_{S^2_R}
\le C\, \| \Delta_P (b_T)\|_{S^2_R}.
$$
It follows that
$$
T^{-2}\omega_k(u; L, R)
+T^{-1} \omega_k (u; L, R) \le C\, \omega_k (b_T, L, R).
$$
Similarly,
$$
\omega_k (\nabla^2 u; L, R) \le C \log (2 +TR^{-1})\,  \omega_k (b_T; L, R).
$$
The desired estimates now follows from (\ref{B-estimate-0}).
\end{proof}

We are now ready to state and prove our main estimates for the dual approximate correctors.

\begin{thm}\label{appx-theorem-1}
	Let $\phi_T = (\phi^{\alpha\beta}_{T,ij})$ be defined in (\ref{cond_phiT}).
	Let $k\ge 1$ and $\sigma\in (0,1)$.
	Then there exists $c>0$, depending only on $d$ and $k$, such that
	  for any $T\ge 2$ and $\sigma \in (0,1)$,
	\begin{equation}\label{Phi-estimate-1}
	\aligned
	& T^{-1}\| \phi_T\|_{S^2_1} + \|\nabla \phi_T \|_{S^2_1} \\
	&\le C_\sigma\int_1^T
	\inf_{1\le L\le t} 
	\left\{ \rho_k (L, t) +\exp{ \left(\frac{-c\,t^2}{L^2}\right)} \right\}
	\left(\frac{T}{t} \right)^\sigma dt,
	\endaligned
	\end{equation}
		where 
		$C_\sigma$ depends only on  $k$, $\sigma$ and $A$.
\end{thm}

\begin{proof}
With estimates (\ref{B-estimate-1}) and (\ref{B-estimate-2}) at our disposal,
as in the case of (\ref{main-estimate-1}),
this theorem follows readily from Theorem \ref{general-theorem}.
\end{proof}

Let
	\begin{equation}\label{h-1}
	h^{\alpha\beta}_{T,j} = \frac{\partial}{\partial x_i} \left( \phi^{\alpha\beta}_{T,ij} \right),
	\end{equation}
	where $\phi_T = (\phi^{\alpha\beta}_{T,ij})$ is defined in (\ref{cond_phiT}).
	Note that the index $i$ is summed.

\begin{thm}\label{appx-theorem-2}
Let $h_T =(h_{T, j}^{\alpha\beta})$ be defined by (\ref{h-1}).
Then 
\begin{equation}\label{h-2}
T\|\nabla h_T\|_{S^2_1} \le C\,  \|\chi_T\|_{S^2_1},
\end{equation}
where $C$ depends only on $d$.
\end{thm}

\begin{proof}
	Observe that by the definition of $\chi_T$, 
	\begin{equation*}
	\frac{\partial}{\partial x_i} \big(b^{\alpha\beta}_{T,ij} \big) = T^{-2}  \chi^{\alpha\beta}_{T,j},
	\end{equation*}
	for each $1\le j\le d$ and $1\le \alpha,\beta \le m$ (index $i$ is summed). In view of (\ref{cond_phiT}) this gives
	$$
	-\Delta h_T + T^{-2} h_T = T^{-2} \chi_T  \quad \text{ in } \R^d.
	$$
	As a result, estimate (\ref{h-2}) follows readily from Lemma \ref{lem_regu_Mor}.
	  \end{proof}

%%%%%%%%%%%%%%%%%%%%%%%%%%%%%%%%%%%%%%%%%%%%%%%%%%%%%%

%%%%%%%%%%%%%%%%%%%%%%%%%%%%%%%%%%%%%%%%%%%%%%%%%

\section{Convergence rates}

In this section we give the proof of Theorem \ref{main-theorem-2}, which establishes the near optimal 
convergence rate in $L^2$. 
Our approach follows the same line of argument as in \cite{Shen-boundary-2015, SZ-2015},
which in turn use ideas from \cite{Suslina-2012}.
While the papers \cite{Suslina-2012, Shen-boundary-2015, SZ-2015} all deal with the case of periodic coefficients,
our argument relies on the estimates for approximate correctors in Theorem \ref{main-theorem-1} as well as 
estimates for dual approximate correctors in Section 8.

We begin by introducing smoothing operators $S_\e$ and $K_{\e, \delta}$. 
Let $\zeta \in C_0^\infty(B(0,1))$ be a nonnegative  function with 
$\int_{\R^d} \zeta  = 1$,  and $\zeta_\e(x) = \e^{-d} \zeta(x/\e) $. Define
\begin{equation}
S_\e f(x) = \zeta_\e* f(x) = \int_{\R^d} \zeta_\e(y) f(x-y) dy.
\end{equation}
Note that, for $1\le p\le \infty$,
\begin{equation}\label{ineq_S_p}
\norm{S_{\e} f}_{L^p(\R^d)} \le \norm{f}_{L^p(\R^d)}.
\end{equation}
It is known that if $f\in L^p(\R^d)$ and $g\in L^p_{\loc, \unif}(\R^d)$, then
\begin{equation}\label{ineq_gK_e}
	\norm{g(x/\e)S_\e(f)}_{L^p(\R^d)}\le \sup_{x\in\R^d} \left( \fint_{B(x,1)} |g|^p \right)^{1/p} \norm{f}_{L^p(\R^d)},
	\end{equation}
and for $f\in W^{1, p}(\R^d)$,
\begin{equation}\label{lem_S_e}
	\norm{S_\e(f) - f}_{L^p(\R^d)} \le C \e \norm{\nabla f}_{L^p(\R^d)},
	\end{equation}
	where $C$ depends only on $d$ (see e.g. \cite{Shen-boundary-2015} for a proof of (\ref{ineq_gK_e})-(\ref{lem_S_e})).	
	
Let $\delta \ge 2\e$ be a small parameter to be determined. 
Let $\eta_\delta \in C_0^\infty(\Omega)$ be a cut-off function 
so that $\eta_\delta(x) =0 $ in $\Omega_\delta 
= \{x\in\Omega; \text{dist}(x,\partial\Omega) < \delta\}$, $\eta_\delta(x) = 1$
 in $\Omega \setminus \Omega_{2\delta}$ and $|\nabla \eta_\delta | \le C\delta^{-1}$. Define
\begin{equation}
K_{\e,\delta} f(x) = S_\e (\eta_\delta f)(x).
\end{equation}
Note that $K_{\e,\delta} f \in C_0^\infty(\Omega)$, as $\delta \ge2\e$.

\begin{lem}\label{lem_Oeu}
	Let $\Omega$ be a bounded Lipschitz domain. Then, for any $u\in H^1(\R^d)$,
	\begin{equation}
	\int_{\Omega_{\e}} |u|^2 \le C \e \norm{u}_{H^1(\R^d)} \norm{u}_{L^2(\R^d)},
	\end{equation}
	where $\Omega_\e = \{x\in \Omega: \text{dist}(x,\partial\Omega) < \e\}$ and the constant $C$ depends only on  $\Omega$.	
\end{lem}

\begin{proof}
See e.g. \cite{SZ-2015}.
\end{proof}

\begin{lem}\label{lem_w_e}
	Let $u_\e, u_0 \in H^1(\Omega;\R^m)$ be 
	weak solutions of $\cL_\e (u_\e) = F$ and  $\cL_0 (u_0) = F$, respectively, in $\Omega$. Assume further that $u_0 \in H^2(\Omega; \R^m)$. Set
	\begin{equation}\label{w}
	w_\e = u_\e - u_0 -\e \chi^\beta_{T,k}(x/\e) K_{\e, \delta} \bigg( \frac{\partial u^\beta_0}{\partial x_k} \bigg),
	\end{equation}
	where $T = \e^{-1}$. Then
	\begin{align}\label{formula-L}
	\begin{aligned}
	\cL_\e(w_\e) = & \frac{\partial}{\partial x_i} \left\{ \left\{\widehat{a}_{ij}^{\alpha\beta} - a^{\alpha\beta}_{ij}(x/\e) \right\} \left\{ K_{\e,\delta} \bigg(\frac{\partial u^\beta_0}{\partial x_j}\bigg)  - \frac{\partial u^\beta_0}{\partial x_j} \right\}\right\} \\
	& + \frac{\partial}{\partial x_i} \left\{ b^{\alpha\beta}_{T,ij}(x/\e) K_{\e, \delta} \bigg( \frac{\partial u^\beta_0}{\partial x_j} \bigg) \right\} \\
	& + \e \frac{\partial}{\partial x_i} \left\{ a^{\alpha\beta}_{ij}(x/\e) \chi^{\beta\gamma}_{T,k}(x/\e) \frac{\partial}{\partial x_j}K_{\e, \delta}\bigg(  \frac{\partial u^\gamma_0}{\partial x_k }\bigg)  \right\},
	\end{aligned}
	\end{align}
	where the function $b^{\alpha\beta}_{T,ij}$ is given in (\ref{cond_bT}).
\end{lem}

\begin{proof}
This follows by some direct algebraic manipulation, using $\cL_\e (u_\e)=F=\cL_0(u_0)$.
\end{proof}

\begin{lem}\label{lem_b_T}
	Let $u_0\in H^2(\Omega;\R^m)$, then
	\begin{align*}
	&\frac{\partial}{\partial x_i} \left\{ b^{\alpha\beta}_{T,ij}(x/\e) K_{\e, \delta} \bigg( \frac{\partial u^\beta_0}{\partial x_j} \bigg) \right\} \\
	&= \ag{b_{T,ij}^{\alpha\beta}} \frac{\partial}{\partial x_i}K_{\e, \delta} \bigg( \frac{\partial u_0^\beta}{ \partial x_j} \bigg) + \frac{\partial}{\partial x_i} \left\{ T^{-2} \phi_{T,ij}^{\alpha\beta}(x/\e) K_{\e, \delta} \bigg(  \frac{\partial u_0^\beta}{\partial x_j}\bigg)\right\} \\
	&\qquad- \frac{\partial}{\partial x_i} \left\{ \frac{\partial}{\partial x_i} h_{T,j}^{\alpha\beta}(x/\e) K_{\e, \delta} \bigg( \frac{\partial u_0^\beta}{\partial x_j}\bigg)\right\} \\
	&\qquad +\e \frac{\partial}{\partial x_i} \left\{ \left[ \frac{\partial}{\partial x_k} (\phi_{T,ij}^{\alpha\beta})(x/\e) - \frac{\partial}{\partial x_i} (\phi_{T,kj}^{\alpha\beta})(x/\e) \right] \frac{\partial}{\partial x_k}K_{\e, \delta} \bigg( \frac{\partial u_0^\beta}{\partial x_j } \bigg) \right\},
	\end{align*}
	where the function $\phi_{T,ij}^{\alpha\beta} (y)$ is defined by (\ref{cond_phiT}) and
	\begin{equation}\label{h}
	h_{T,j}^{\alpha\beta} (y)=\frac{\partial}{\partial y_k} \phi_{T,kj}^{\alpha\beta}.
	\end{equation}
\end{lem}

\begin{proof}
	This follows  from the identity
	\begin{equation}\label{cond_bT_phiT}
	b_{T,ij}^{\alpha\beta} = \ag{b_{T,ij}^{\alpha\beta}} - \frac{\partial}{\partial y_k} \left( \frac{\partial}{\partial y_k} \phi_{T,ij}^{\alpha\beta} - \frac{\partial}{\partial y_i} \phi_{T,kj}^{\alpha\beta} \right) - \frac{\partial}{\partial y_i}\left( \frac{\partial}{\partial y_k} \phi_{T,kj}^{\alpha\beta} \right)  +T^{-2}\phi_{T,ij}^{\alpha\beta},
	\end{equation}
	as well as the fact that
	 the second term in the r.h.s. of (\ref{cond_bT_phiT}) is skew-symmetric with respect to $(i,k)$.
\end{proof}

The formulas in the previous two lemmas allow us to establish the following.

\begin{lem}\label{lem_Lw}
	Let $w_\e$ be the same as in Lemma \ref{lem_w_e}, $T = \e^{-1} > 1$ and $2\e\le \delta<2$.
	 Then, for any $\varphi\in H^1_0(\Omega;\R^m)$,
	\begin{equation}\label{ineq_Lw}
	\aligned
&	\Abs{\int_{\Omega} A(x/\e) \nabla w_\e \cdot \nabla \varphi} \\
	& 
	\le C \Big\{ \delta + \|\nabla \chi_T -\psi\|_{B^2}
	+ T^{-1} \|\chi_T\|_{S^2_1}
	+T^{-2} \| \phi_T\|_{S^2_1}
	+ T^{-1} \|\nabla \phi_T\|_{S^2_1} \Big\}\\
	&\qquad\qquad
	\cdot  \Big\{ \norm{\nabla\varphi}_{L^2(\Omega)} + \delta^{-1/2} 
	\norm{\nabla\varphi}_{L^2(\Omega_{4\delta})}\Big\}  \norm{u_0}_{H^2(\Omega)},
	\endaligned
	\end{equation}
	where $\psi$ is defined by (\ref{cond_Vpot}) and the constant $C$ depends only on
	$A$ and $\Omega$.
\end{lem}
\begin{proof}
	It follows from (\ref{formula-L}) that
	\begin{align}\label{ineq_Lw1}
	\begin{aligned}
	\int_{\Omega} A(x/\e) \nabla w_\e \cdot \nabla\varphi 
	&= - \int_{\Omega} \left\{\widehat{a}_{ij}^{\alpha\beta} - a^{\alpha\beta}_{ij}(x/\e) \right\} \left\{ K_{\e,\delta} \bigg(\frac{\partial u^\beta_0}{\partial x_j}\bigg)  - \frac{\partial u^\beta_0}{\partial x_j} \right\} \frac{\partial\varphi^\alpha}{\partial x_i}  \\
	&\qquad-  \e \int_{\Omega} a^{\alpha\beta}_{ij}(x/\e) \chi^{\beta\gamma}_{T,k}(x/\e) \frac{\partial}{\partial x_j}K_{\e, \delta}\bigg(  \frac{\partial u^\gamma_0}{\partial x_k }\bigg) \frac{\partial\varphi^\alpha}{\partial x_i} \\
	& \qquad - \int_{\Omega} b^{\alpha\beta}_{T,ij}(x/\e) K_{\e, \delta} \bigg( \frac{\partial u^\beta_0}{\partial x_j} \bigg) \frac{\partial\varphi^\alpha}{\partial x_i}.
	\end{aligned}
	\end{align}	
	Observe that
	\begin{equation}
	K_{\e, \delta}(\nabla u_0) - \nabla u_0 = S_{\e}( \eta_\delta \nabla u_0) - \eta_\delta \nabla u_0 + (\eta_\delta - 1) \nabla u_0,
	\end{equation}
	and $\eta_\delta -1=0$  in $\Omega\setminus\Omega_{2\delta}$. Thus, in view of
	 (\ref{lem_S_e}) and  Lemma \ref{lem_Oeu},
	  the first term in the r.h.s. of (\ref{ineq_Lw1}) is bounded by
	\begin{equation}
	\begin{aligned}\label{ineq_Lw11}
	&C\norm{ S_{\e}( \eta_\delta \nabla u_0) - \eta_\delta \nabla u_0}_{L^2(\Omega)} \norm{\nabla \varphi}_{L^2(\Omega)} + C\norm{\nabla u_0}_{L^2(\Omega_{2\delta})} \norm{\nabla \varphi}_{L^2(\Omega_{2\delta})}  \\
	&\qquad \le C \Big\{ \e \norm{\nabla\varphi}_{L^2(\Omega)} + \delta^{1/2}\norm{\nabla \varphi}_{L^2(\Omega_{2\delta})} \Big\}\norm{u_0}_{H^2(\Omega)}\\
	&\qquad \le C\, \delta \Big\{ \|\nabla\varphi\|_{L^2(\Omega)}
	+\delta^{-1/2} \|\nabla \varphi\|_{L^2(\Omega_{2\delta})}\Big\} \| u_0\|_{H^2(\Omega)}.
	\end{aligned}
	\end{equation}
	
	Next, note that
	\begin{equation}\label{ineq_dK}
	\frac{\partial}{\partial x_j}K_{\e, \delta}\bigg(  \frac{\partial u^\gamma_0}{\partial x_k }\bigg) = S_\e \bigg( \eta_\delta \frac{\partial^2 u^\gamma_0}{\partial x_j \partial x_k } \bigg) + S_\e \bigg( \frac{\partial \eta_\delta}{\partial x_j} \frac{\partial u^\gamma_0}{\partial x_k } \bigg),
	\end{equation}
	and  the last term of (\ref{ineq_dK}) is zero in $\Omega\setminus \Omega_{4\delta}$, since $\delta>2\e$. 
	It follows from (\ref{ineq_gK_e}) and Lemma \ref{lem_Oeu} that the second term in the r.h.s. of (\ref{ineq_Lw1}) is bounded by
	\begin{equation}
	\begin{aligned}\label{ineq_Lw12}
	&C\norm{\e \chi_T S_\e(\eta_\delta \nabla^2 u_0)}_{L^2(\Omega)}\norm{\nabla\varphi}_{L^2(\Omega)} + C\norm{\e \chi_T S_\e(\nabla \eta_\delta\nabla u_0)}_{L^2(\Omega_{4\delta})} \norm{\nabla \varphi}_{L^2(\Omega_{4\delta})}\\
	&\qquad \le C\, \|\e \chi_T\|_{S^2_1} \norm{\nabla^2 u_0}_{L^2(\Omega)}\norm{\nabla\varphi}_{L^2(\Omega)} \\
&\qquad\qquad\qquad	+ C\,\delta^{-1} \| \e \chi_T\|_{S^2_1}
	 \norm{\nabla u_0}_{L^2(\Omega_{4\delta})} \norm{\nabla \varphi}_{L^2(\Omega_{4\delta})}\\
	&\qquad \le C\Big\{ \norm{\nabla\varphi}_{L^2(\Omega)} + 
	\delta^{-1/2} \norm{\nabla \varphi}_{L^2(\Omega_{4\delta})} \Big\}  \|\e \chi_T\|_{S^2_1} \norm{u_0}_{H^2(\Omega)}.
	\end{aligned}
	\end{equation}
	
	It remains to estimate the third term in the r.h.s. of (\ref{ineq_Lw1}). To this end we use  Lemma \ref{lem_b_T}
	to obtain 
	\begin{equation}
	\begin{aligned}\label{ineq_Lw2}
	&\int_{\Omega} b^{\alpha\beta}_{T,ij}(x/\e) K_{\e, \delta} \bigg( \frac{\partial u^\beta_0}{\partial x_j} \bigg) \frac{\partial\varphi^\alpha}{\partial x_i} \\
	&\qquad = \ag{b_{T,ij}^{\alpha\beta}} \int_{\Omega} K_{\e, \delta} \bigg( \frac{\partial u_0^\beta}{ \partial x_j} \bigg)\frac{\partial\varphi^\alpha}{\partial x_i}  + \int_{\Omega}  T^{-2} \phi_{T,ij}^{\alpha\beta}(x/\e) K_{\e, \delta} \bigg(  \frac{\partial u_0^\beta}{\partial x_j}\bigg) \frac{\partial\varphi^\alpha}{\partial x_i}\\
	&\qquad \qquad - \int_{\Omega} \frac{\partial}{\partial x_i} h_{T,j}^{\alpha\beta}(x/\e) K_{\e, \delta} \bigg( \frac{\partial u_0^\beta}{\partial x_j}\bigg) \frac{\partial\varphi^\alpha}{\partial x_i} \\
	&\qquad \qquad + \e \int_{\Omega} \left[ \frac{\partial}{\partial x_k} (\phi_{T,ij}^{\alpha\beta})(x/\e) - \frac{\partial}{\partial x_i} (\phi_{T,kj}^{\alpha\beta})(x/\e) \right] \frac{\partial}{\partial x_k}K_{\e, \delta} \bigg( \frac{\partial u_0^\beta}{\partial x_j } \bigg) \frac{\partial\varphi^\alpha}{\partial x_i}.
	\end{aligned}
	\end{equation}
	It follows from (\ref{cond_Ahat}) and (\ref{ineq_S_p}) that,
	\begin{equation}\label{ineq_Lw21}
	\Abs{ \ag{b_{T,ij}^{\alpha\beta}} \int_{\Omega} 
	K_{\e, \delta} \bigg( \frac{\partial u_0^\beta}{ \partial x_j} \bigg)\frac{\partial\varphi^\alpha}{\partial x_i} }
	 \le C \norm{\nabla \chi_T - \psi}_{B^2} \|\nabla \varphi\|_{L^2(\Omega)} \norm{u_0}_{H^1(\Omega)}.
	\end{equation}
	Also, the second and third terms  in the r.h.s. of (\ref{ineq_Lw2}) are bounded by
	\begin{equation}\label{9-10}
	 C \Big\{  T^{-2} \|\phi_T\|_{S^2_1} + \|\nabla h_T\|_{S^2_1} \Big\} \| \nabla u_0\|_{L^2(\Omega)} \| \nabla \varphi\|_{L^2(\Omega)},
	 \end{equation}
	 while the last term is bounded by
	 \begin{equation}\label{9-11}
	 C \, \|\e \nabla \phi_T\|_{S^2_1}
	 \Big\{ \|\nabla \varphi\|_{L^2(\Omega)} +\delta^{-1/2} \|\nabla \varphi\|_{L^2(\Omega_{4\delta})} \Big\} \| u_0\|_{H^2(\Omega)}.
	 \end{equation}	
	As a result, we obtain 
	\begin{equation}
	\begin{aligned}\label{ineq_Lw3}
	&\Abs{ \int_{\Omega} b^{\alpha\beta}_{T,ij}(x/\e) 
	K_{\e, \delta} \bigg( \frac{\partial u^\beta_0}{\partial x_j} \bigg) \frac{\partial\varphi^\alpha}{\partial x_i}} \\
	& \le C \Big\{\|\chi_T -\psi\|_{B^2}
	+T^{-2} \|\phi_T\|_{S^2_1} +T^{-1}\|\nabla \phi_T\|_{S^2_1}
	+\|\nabla h_T\|_{S^2_1} \Big\}\\
& \qquad\qquad\qquad\cdot
	 \Big\{ \norm{\nabla\varphi}_{L^2(\Omega)} + \delta^{-1/2} \norm{\nabla\varphi}_{L^2(\Omega_{4\delta})} \Big\} \norm{u_0}_{H^2(\Omega)}.
	\end{aligned}
	\end{equation}
	Finally, the estimate (\ref{ineq_Lw}) follows by combining 
	(\ref{ineq_Lw1}), (\ref{ineq_Lw11}), (\ref{ineq_Lw12}) and (\ref{ineq_Lw3}).
	The estimate (\ref{h-1}) is also used here.
\end{proof}	

The next theorem provides an error estimate for $u_\e$ in $H^1(\Omega)$.

\begin{thm}\label{theorem-9-1}
Suppose that $A\in APW^2(R^d)$ and satisfies  (\ref{cond_ellipticity}).
Let $\Omega$ be a bounded Lipschitz domain in $\R^d$ and $0<\varep<1$.
Let $u_\e, u_0\in  H^1(\Omega; \R^m)$ be weak solutions of $\mathcal{L}_\e (u_\e)=F$,
$\mathcal{L}_0 (u_0)=F$ in $\Omega$, respectively.
Assume that $u_\e =u_0$ on $\partial\Omega$ and $u_0\in H^2(\Omega; \R^m)$.
Then 
\begin{equation}\label{9-1-1}
\|  u_\e - u_0 -\e \chi^\beta_{T,k}(x/\e) K_{\e, \delta} \bigg( \frac{\partial u^\beta_0}{\partial x_k} \bigg) \|_{H_0^1(\Omega)} 
\le C\, \delta^{1/2} \| u_0\|_{H^2(\Omega)},
\end{equation}
where $T=\varep^{-1}$,
\begin{equation}\label{9-1-2}
\delta=2T^{-1}+\|\nabla  \chi_T-\psi\|_{B^2}
+T^{-1}\|\chi_T\|_{S^2_1}
+ T^{-2} \|\phi_T\|_{S^2_1}
+T^{-1} \|\nabla \phi_T\|_{S^2_1}, 
\end{equation}
and $C$ depends only on $\Omega$ and $A$.
\end{thm}

\begin{proof}
This follows from Lemma \ref{lem_Lw} by letting $\varphi=w_\e$, where $w_\e$ is
defined by  (\ref{w}) with $\delta$ given by (\ref{9-1-2}). 
Note that $w_\e\in H^1_0(\Omega; \R^m)$ and
$\delta\ge 2\e$.
\end{proof}

\begin{thm}\label{theorem-9-2}
Let $A$, $\Omega$, $u_\varep$ and $u_0$ be the same as in Theorem \ref{theorem-9-1}.
We further assume that $\Omega$ is a bounded $C^{1,1}$ domain.
Let $\delta^*$ be defined by (\ref{9-1-2}), but with $A$ replaced by $A^*$.
Then
\begin{equation}\label{9-2-1}
\| u_\varep -u_0\|_{L^2(\Omega)} \le C  \big\{ \delta +\delta^*\big\} \| u_0\|_{H^2(\Omega)},
\end{equation}
where $\delta$ is given by (\ref{9-1-2}) and $C$ depends only on $A$ and $\Omega$.
\end{thm}

\begin{proof}
The theorem is proved by a duality argument, following the approach in \cite{Suslina-2012}.
Consider the Dirichlet problem,
		\begin{equation}\label{cond_dual_e}
	\cL_\e^*(v_\e) = G \quad \text{ in } \Omega \quad \text{and} \quad v_\e=0 \quad \text{on } \partial\Omega,
	\end{equation}
	where $\cL_\e^*$ is the adjoint  operator of $\cL_\e$. The corresponding homogenized problem of (\ref{cond_dual_e}) is given by
	\begin{equation}\label{cond_dual_0}
	\cL_0^*(v_0) = G \quad \text{ in } \Omega \quad \text{and} \quad v_0=0 \quad \text{on } \partial\Omega.
	\end{equation}
	where $\cL_0^*$ is the adjoint operator of $\cL_0$. 
	It is known that if $\Omega$ is a bounded $C^{1,1}$ domain and $G\in L^2(\Omega;\R^m)$, then the unique weak solution $v_0$ of (\ref{cond_dual_0})
	with constant coefficients  is in $H^2(\Omega;\R^m)$ and satisfies the estimate
	\begin{equation}
	\norm{v_0}_{H^2(\Omega)} \le C\,  \norm{G}_{L^2(\Omega)},
	\end{equation}
	where $C$ depends only on $d$, $m$, $\mu$ and $\Omega$.
	Let
	\begin{equation}
	\rho_\e = v_\e - v_0 - \e \chi^{*\beta}_{T,k}(x/\e) K_{\e, \delta^*+5\delta} 
	\bigg( \frac{\partial v^\beta_0}{\partial x_k} \bigg),
	\end{equation}
	where $\chi_T^*$ denotes the approximate corrector for operators $\{ \cL_\e^*\}$.
	 It follows from  Theorem \ref{theorem-9-1} that
	\begin{equation}\label{ineq_rho_H1}
	\Norm{\rho_\e }_{H_0^1(\Omega)} \le C(\delta +\delta^*)^{1/2} \norm{v_0}_{H^2(\Omega)}.
	\end{equation}
	
	Next, we observe that to show estimate (\ref{9-2-1}), it suffices to prove
	\begin{equation}\label{ineq_we_L2}
	\norm{w_\e}_{L^2(\Omega)} \le C\big\{ \delta +\delta^* \big\} \norm{u_0}_{H^2(\Omega)},
	\end{equation}
	where $w_\e$ is defined by (\ref{w}). This is because 
	\begin{equation}
	\Norm{\e \chi^\beta_{T,k}(x/\e) K_{\e, \delta} 
	\bigg( \frac{\partial u^\beta_0}{\partial x_k} \bigg)}_{L^2(\Omega)} 
	\le C\, \delta\,  \norm{u_0}_{H^1(\Omega)},
	\end{equation}
	by  (\ref{ineq_gK_e}).
	To prove (\ref{ineq_we_L2}),  we use
	\begin{equation}\label{eq_weG}
	\begin{aligned}
	\int_{\Omega} w_\e \cdot G & = \int_{\Omega} A(x/\e)\nabla w_\e \cdot \nabla v_\e \\
	& = \int_{\Omega} A(x/\e)\nabla w_\e \cdot \nabla \rho_\e + \int_{\Omega} A(x/\e)\nabla w_\e \cdot \nabla v_0 \\
	&\qquad + \int_{\Omega} A(x/\e)\nabla w_\e \cdot \nabla \left( \e \chi^{*\beta}_{T,k}(x/\e) K_{\e, \delta^*+5\delta} \bigg( \frac{\partial v^\beta_0}{\partial x_k} \bigg) \right).
	\end{aligned}
	\end{equation}
	 It follows from (\ref{9-1-1}) and (\ref{ineq_rho_H1}) that
	\begin{equation}\label{ineq_ve1}
	\aligned
	\Abs{\int_{\Omega} A(x/\e)\nabla w_\e \cdot \nabla \rho_\e}
	& \le C \norm{\nabla w_\e}_{L^2(\Omega)} \norm{\nabla\rho_\e}_{L^2(\Omega)} \\
	&\le C\big\{ \delta +\delta^*\big\}
	 \norm{u_0}_{H^2(\Omega)} \norm{v_0}_{H^2(\Omega)}.
	\endaligned
	\end{equation}
	To handle the second integral in the r.h.s. of (\ref{eq_weG}), 
	we observe that  $v_0\in H_0^1(\Omega;\R^m)$ and thus by Lemma \ref{lem_Lw},
	\begin{equation}\label{ineq_ve2}
	\begin{aligned}
	\Abs{\int_{\Omega} A(x/\e)\nabla w_\e \cdot \nabla v_0} &\le C \Big\{\delta \norm{\nabla v_0}_{L^2(\Omega)} + \delta^{1/2} \norm{\nabla v_0}_{L^2(\Omega_{4\delta})}\Big\}  \norm{u_0}_{H^2(\Omega)} \\
	&\le C\, \delta \,
	 \norm{v_0}_{H^2(\Omega)} \norm{u_0}_{H^2(\Omega)},
	\end{aligned}
	\end{equation}
	where we have used  Lemma \ref{lem_Oeu}  for the last inequality. 
	
	Finally, to bound the last integral in the r.h.s. of (\ref{eq_weG}), 
	we apply Lemma \ref{lem_Lw} again to obtain
	\begin{equation}\label{ineq_ve3}
	\begin{aligned}
	&\Abs{\int_{\Omega} A(x/\e)\nabla w_\e \cdot \nabla 
	\left( \e \chi^{*\beta}_{T,k}(x/\e) K_{\e, \delta^*+5\delta} \bigg( \frac{\partial v^\beta_0}{\partial x_k} \bigg) \right)}  \\
	&\qquad \le C \norm{u_0}_{H^2(\Omega)} \cdot \left\{\delta \Norm{\nabla \left( \e \chi^{*\beta}_{T,k}(x/\e) 
	K_{\e, \delta^*+5\delta} \bigg( \frac{\partial v^\beta_0}{\partial x_k} \bigg) \right)}_{L^2(\Omega)} \right. \\
	&\qquad \qquad\qquad\qquad
	\left. + \delta^{1/2} \Norm{\nabla \left( \e \chi^{*\beta}_{T,k}(x/\e) K_{\e, \delta^*+5\delta} \bigg( \frac{\partial v^\beta_0}{\partial x_k} \bigg) \right)}_{L^2(\Omega_{4\delta})}\right\}.
	\end{aligned}
	\end{equation}
	By (\ref{ineq_rho_H1}) and the energy estimate,
	\begin{equation*}
	\aligned
	\Norm{\nabla \left( \e \chi^{*\beta}_{T,k}(x/\e) K_{\e, \delta^*+5\delta} \bigg( \frac{\partial v^\beta_0}{\partial x_k} \bigg) \right)}_{L^2(\Omega)} & \le \norm{\rho_\e}_{H^1(\Omega)}+\norm{v_\e}_{H^1(\Omega)}
	+\norm{v_0}_{H^1(\Omega)}\\
	&\le C\, \norm{v_0}_{H^2(\Omega)}.
	\endaligned
	\end{equation*}
	Also note that
	\begin{equation*}
	  K_{\e, \delta^*+5\delta} \bigg( \frac{\partial v^\beta_0}{\partial x_k} \bigg)  = 0
	 \quad \text{ in } \Omega_{4\delta}.
	\end{equation*}
	As a result, it follows from (\ref{ineq_ve3}) that
	\begin{equation}
	\Abs{\int_{\Omega} A(x/\e)\nabla w_\e \cdot \nabla \left( \e \chi^{*\beta}_{T,k}(x/\e) K_{\e, \delta^*+5\delta} \bigg( \frac{\partial v^\beta_0}{\partial x_k} \bigg) \right)} \le C \delta \norm{v_0}_{H^2(\Omega)} \norm{u_0}_{H^2(\Omega)}.
	\end{equation}
	This, together with (\ref{eq_weG}), (\ref{ineq_ve1}) and (\ref{ineq_ve2}), leads to
	\begin{equation}
	\aligned
	\Abs{\int_{\Omega} w_\e \cdot G} &\le  C\big\{  \delta +\delta^*\big\}
	 \norm{v_0}_{H^2(\Omega)} \norm{u_0}_{H^2(\Omega)} \\
	&\le  C\big\{ \delta +\delta^*\big\}
	 \norm{G}_{L^2(\Omega)} \norm{u_0}_{H^2(\Omega)}.
	\endaligned
	\end{equation}
	Therefore, by duality,
	\begin{equation}
	\norm{w_\e}_{L^2(\Omega)} \le C\big\{  \delta  +\delta^*\big\} \norm{u_0}_{H^2(\Omega)},
	\end{equation}
	which completes the proof of Theorem \ref{theorem-9-2}.
\end{proof}

\begin{proof}[\bf Proof of Theorem \ref{main-theorem-3}]
Let $\delta$, $\delta^*$ be the same as in Theorem \ref{theorem-9-2}.
Let $\Theta_{k, \sigma}(T)$ denote the integral in the r.h.s. of (\ref{main-estimate-1}).
It follows from Theorems \ref{main-theorem-1}, \ref{appx-theorem-1} and \ref{appx-theorem-2} that
$$
\aligned
  \delta & \le C_\sigma \left\{ \|\nabla \chi_T -\psi\|_{B^2} 
+ T^{-1} \Theta_{k, \sigma} (T) \right\},\\
 \delta^* & \le C_\sigma \left\{ \|\nabla \chi^*_T -\psi^*\|_{B^2} 
+T^{-1} \Theta_{k, \sigma} (T) \right\},
\endaligned
$$
for any $k\ge 1$ and $\sigma \in (0,1)$, where $C_\sigma$ depends only on $\sigma$, $k$ and $A$.
This, together with Theorem \ref{theorem-9-2},
gives  the estimate (\ref{ineq_L2}) in Theorem \ref{main-theorem-3}.

Now suppose that the condition (\ref{decay-0})
holds for some $\alpha>1$ and $k\ge 1$.
Then, by Theorem \ref{main-theorem-2},
$$
\|\chi_T\|_{S^2_1} +\|\chi_T^*\|_{S^2_1}\le C.
$$
To see (\ref{optimal}), we note that by the proof of Theorem 6.6 in \cite[p.1590]{Shen-2015},
$$
\|\nabla \chi_T-\psi \|_{B^2} +\|\nabla \chi^*_T-\psi^* \|_{B^2}
\le C \sum_{j=1}^\infty
\left\{ \Norm{\frac{\chi_{2^j T}}{2^j T} }_{B^2} +\Norm{\frac{\chi^*_{2^j T}}{2^j T} }_{B^2}\right\}
\le \frac{C}{T},
$$
for any $T>1$.
Finally, using the same argument as in the case of $\chi_T$,
we may show that
$$
T^{-1} \|\phi_T\|_{S^2_1} + \|\nabla \phi_T \|_{S^2_1}
+T^{-1} \|\phi_T^*\|_{S^2_1} + \|\nabla \phi_T ^*\|_{S^2_1}
\le C.
$$
As a result, we obtain $\delta +\delta^*\le C\,  T^{-1}$. In view of Theorem \ref{theorem-9-2},
this gives the $O(\e)$ estimate (\ref{optimal}).
\end{proof}

%%%%%%%%%%%%%%%%%%%%%%%%%%%%%%%%%%%%%%%%%%%%%%%%%%

%%%%%%%%%%%%%%%%%%%%%%%%%%%%%%%%%%%%%%%%%%%%

\section{Lipschitz estimates at large scale}

In this section we establish an interior $L^2$-based Lipschitz estimate at large scale
  under a general  condition:
there exists a nonnegative increasing function $\eta(t)$ on $[0, 1]$
with the Dini property
\begin{equation}\label{Dini}
\int_0^1 \frac{\eta( t)}{t}\, dt <\infty,
\end{equation}
such that
\begin{equation}\label{Lip-condition-11}
\| u_\varep -u_0 \|_{L^2(B(x_0, 1))} \le 
\big[ \eta(\e) \big]^2 \| u_0\|_{H^2(B(x_0, 1))}, 
\end{equation}
whenever $u_\varep , u_0 \in H^1(B(x_0, 1); \R^m)$,
$\mathcal{L}_\varep (u_\varep)=\mathcal{L}_0 (u_0)$ in $B(x_0, 1)$
and $u_\varep =u_0$ on $\partial B(x_0, 1)$, for some $x_0\in \R^d$ and $0<\e<1$.
We note that by rescaling,
 (\ref{Lip-condition-11}) continues to hold if $B(0,1)$ is replaced by $B(0,r)$ for $1\le r\le 2$.

\begin{thm}\label{Lip-theorem}
Suppose that $A\in B^2(\R^d)$ and satisfies the ellipticity condition (\ref{cond_ellipticity}).
Also assume that conditions (\ref{Dini})-(\ref{Lip-condition-11}) hold.
Let $u_\e\in H^1(B(x_0, R); \R^m)$ be a weak solution of
$\mathcal{L}_\e (u_\e)=F$ in $B(x_0, R)$ for some $x_0\in \R^d$ and $R>\e$.
Then for any $\varep\le r\le R$ and $\sigma \in (0,1)$,
\begin{equation}\label{Lip-estimate}
\aligned
& \left(\fint_{B(x_0, r)}
|\nabla u_\e|^2\right)^{1/2}\\
& \le C_\sigma  \left\{  \left(\fint_{B(x_0, R)} |\nabla u_\e|^2\right)^{1/2}
+\sup_{\substack{ x\in B(x_0, R/2)\\ \e\le t \le R/2}}
 t \left(\frac{R}{t}\right)^\sigma \left(\fint_{B(x, t)} |F|^2 \right)^{1/2} \right\},
\endaligned
\end{equation}
where $C_\sigma $ depends only on  $d$, $m$, $\mu$, $\sigma$,
and the function $\eta$ in (\ref{Lip-condition-11}).
\end{thm}

The proof of Theorem \ref{Lip-theorem} is based on a general approach originated in 
\cite{Armstrong-Smart-2014} and further developed in \cite{Armstrong-Shen-2016, Armstrong-Mourrat, Shen-boundary-2015}
for Lipschitz estimates.
The argument in this section follows closely that in \cite{Shen-boundary-2015}, where the large-scale  boundary
Lipschitz estimates for systems of linear elasticity 
with bounded measurable periodic coefficients  $\mathcal{L}_\e (u_\e)=F$ are obtained 
for $F\in L^p_{\loc}$, $p>d$.
However, in order to apply the estimates to $\chi_T$ for elliptic systems with bounded measurable 
a.p. coefficients,
we need to consider the case where $F\in L^2_{\loc}$.
As a result, modifications of the argument in \cite{Armstrong-Shen-2016,Shen-boundary-2015}
are needed for the proof of Theorem \ref{main-theorem-Lip}. 
Notice that if $F\in L^p_{\loc}$ for some $p>d$,
 the second term in the r.h.s. of (\ref{Lip-estimate})  with $\sigma =\frac{d}{p}$
 is bounded by
$$
C \, R \left(\fint_{B(x_0,R)} |F|^p\right)^{1/p}.
$$

We begin with a lemma that utilizes the condition (\ref{Lip-condition-11}).

\begin{lem}\label{lemma-11-1}
Assume $A$ satisfies the same conditions as in Theorem \ref{Lip-theorem}.
Let $u_\varep\in H^1(B(0,2); \R^m)$ be a weak solution of
$\mathcal{L}_\e (u_\e)=F$ in $B(0,2)$, where $0<\e<1$
and $F\in L^2(B(0,2); \R^m)$.
Then there exists  $v\in H^1(B(0,1);\R^m)$ such that $\mathcal{L}_0 (v)=F$ 
in $B(0,1)$ and
\begin{equation}\label{11-1-0}
\| u_\e -v\|_{L^2(B(0,1))}
\le C\,  \eta (\e) 
\Big\{ \| u_\e\|_{L^2(B(0,2))}
+\| F\|_{L^2(B(0,2))} \Big\},
\end{equation}
where $C$ depends only on $d$, $m$ and $\mu$.
\end{lem}

\begin{proof}
By Caccioppoli's inequality,
$$
\int_{B(0, 3/2)} |\nabla u_\varep|^2\le C \int_{B(0,2)} |u_\e|^2 + C \int_{B(0,2)} |F|^2.
$$
It follows by the co-area formula that there exists some $r_0\in (1,3/2)$ such that
$u_\e\in H^1(\partial B(0, r_0); \R^m)$ and
\begin{equation}\label{11-1-1}
\int_{\partial B(0, r_0)} |\nabla u_\varep|^2 \le C \int_{B(0,2)} |u_\e|^2 + C \int_{B(0,2)} |F|^2.
\end{equation}
Let $f =u_\varep |_{\partial B(0, r_0)}$.
We choose $g_\delta  \in H^{3/2}(\partial B(0, r_0); \R^m)$ such that
\begin{equation}\label{11-1-2}
\aligned
\| g_\delta  -f\|_{H^{1/2}(\partial B(0, r_0))}  & \le C\, \delta^{1/2} \| f\|_{H^1(\partial B(0, r_0))},\\
\| g_\delta \|_{H^{3/2}(\partial B(0, r_0))}  & \le C\, \delta^{-1/2} \| f\|_{H^1(\partial B(0, r_0))}.
\endaligned
\end{equation}
Let $v_\e$ be the weak solution of
$\mathcal{L}_\e (v_\e)=F$ in $B(0, r_0)$ with $v_\e= g_\delta$ on $\partial  B(0, r_0)$, and
$v$ the weak solution of
$\mathcal{L}_0 (v)=F$ in $B(0, r_0)$ with $v= g_\delta$ on $\partial  B(0, r_0)$.
Then
$$
\aligned
\| u_\e -v\|_{L^2(B(0,1))}
&\le \| u_\e - v_\e\|_{L^2(B(0,1))} +\| v_\e - v\|_{L^2(B(0,1))}\\
&  \le \| u_\e -v_\e\|_{H^1(B(0, r_0))} + \| v_\e -v\|_{L^2(B(0, r_0))}\\
& \le C \| f -g_\delta\|_{H^{1/2}(\partial B(0,r_0))}
+ C \big[\eta(\e)]^2 \| v\|_{H^2(B(0,r_0))}\\
& \le C \delta^{1/2} \|  f\|_{H^1(\partial B(0, r_0))}
+ C \big[\eta(\e)]^2 \delta^{-1/2} \| f\|_{H^1(\partial B(0,r_0))}\\
&\le C \Big\{ \delta^{1/2} + \big[\eta(\e)]^2 \delta^{-1/2}\Big\} 
\left\{ \| u_\e\|_{L^2(B(0, 2))}
+\| F\|_{L^2(B(0,2))} \right\},
\endaligned
$$
where we have used the condition (\ref{Lip-condition-11}) for the third inequality,
(\ref{11-1-2}) for the fourth and (\ref{11-1-1})
for the last.
Estimate (\ref{11-1-0}) now follows by letting $\delta= [\eta (\e)]^2$.
\end{proof}

\begin{lem}\label{lemma-11-2}
Assume $A$ satisfies the same conditions as in Theorem \ref{Lip-theorem}.
Let $\e \le r<1$.
Let $u_\varep\in H^1(B(0,2r); \R^m)$ be a weak solution of
$\mathcal{L}_\e (u_\e)=F$ in $B(0, 2r)$
for some $F\in L^2(B(0,2r); \R^m)$.
Then there exists $v\in H^1(B(0,r); \R^m)$ such that
$\mathcal{L}_0 (v)=F$ in $B(0,r)$ and
\begin{equation}\label{11-2-0}
\aligned
 &\left(\fint_{B(0, r)} |u_\e -v|^2\right)^{1/2}\\
& \le C\, \eta (\e/r )
\left\{ \left(\average_{B(0,2r)} |u_\e|^2\right)^{1/2}
+r^{2} \left(\fint_{B(0,2r)} |F|^2\right)^{1/2} \right\},
\endaligned
\end{equation}
where $C$ depends only on $d$, $m$ and $\mu$.
\end{lem}

\begin{proof}
Note that if $v(x) =u_\e (rx)$, then $\mathcal{L}_{\frac{\e}{r}} (v)(x)= r^2 F(rx)$.
As a result, the estimate (\ref{11-2-0}) follows readily from Lemma \ref{lemma-11-1} by rescaling.
\end{proof}

The next lemma gives a regularity property for solutions of elliptic systems with constant coefficients.

\begin{lem}\label{lemma-11-3}
Let $v$ be a weak solution of $\mathcal{L}_0 (v)=F$ in $B(0,r)$
for some $F\in L^2(B(0, r); \R^m)$.
Then, for any $0<t<r/4$, 
\begin{equation}\label{11-3-0}
\aligned
 &\inf_{\substack{M\in \R^{m\times d}\\ q\in \R^m}}
\frac{1}{t}
\left(\fint_{B(0,t)} | v-Mx-q|^2\right)^{1/2}\\
 &\qquad \le C \left(\frac{t}{r}\right)
\inf_{\substack{M\in \R^{m\times d}\\ q\in \R^m}}
\frac{1}{r}
\left(\fint_{B(0,r)} | v-Mx-q|^2\right)^{1/2}\\
&\qquad \qquad + C \log \left(\frac{r}{t}\right)
\sup_{x\in B(0, \frac{r}{2})}
 t \left(\fint_{B(x,t)} |F|^2\right)^{1/2},
 \endaligned
 \end{equation}
 where $C$ depends only on $d$, $m$ and $\mu$.
\end{lem}

\begin{proof}
By rescaling we may assume $r=1$.
We may also assume that $0<t<1/100$, as the case $1/100\le t\le 1$ is trivial.
Let $\Gamma_0 (x)$ denote the matrix of fundamental solutions for the operator
$\mathcal{L}_0$ with constant coefficients.
Let $\varphi\in C_0^1(B(0,1/2))$ such that $\varphi =1$ in $B(0, 3/8)$.
Using the representation by fundamental solutions, we may write
$v(x)= w(x) + I(x)$ for $x\in B(0,1/4)$, where
$$
w(x) =\int_{B(0,1/2)} \Gamma_0 (x-y) F(y) \varphi (y)\, dy  
$$ 
and the function $I(x)$ satisfies 
$$
|\nabla^2 I (x)|\le C \Big\{ \| v\|_{L^2(B(0,1))} +\| F\|_{L^2(B(0,1/2))} \Big\}.
$$
Note that for $x\in B(0, t)$, where $0<t<1/100$,
\begin{equation}\label{11-3-1}
|\nabla^2 w(x)|
\le \left| \nabla_x^2 \int_{B(0,2t)} \Gamma_0 (x-y) F(y) \, dy \right|
+ C\int_{2t\le |y|\le (1/2)}
\frac{|F(y)|}{|y|^d}\, dy,
\end{equation}
where we have used the estimate $|\nabla^2 \Gamma_0 (x)|\le C |x|^{-d}$.

Next, we observe that the second term in the r.h.s. of (\ref{11-3-1}) is bounded by
$$
  C \log \left(\frac{1}{t} \right)
\sup_{x\in B(0, \frac{1}{2})}
  \left(\fint_{B(x,t)} |F|^2\right)^{1/2}.
$$
To handle the first term in the r.h.s. of (\ref{11-3-1}), we use the singular integral estimates.
As a result we obtain 
$$
\left(\average_{B(0,t)} | \nabla^2 w| \right)^{1/2}
\le C \log \left(\frac{1}{t} \right)
\sup_{x\in B(0, \frac{1}{2})}
  \left(\fint_{B(x,t)} |F|^2\right)^{1/2}.
  $$
  
  Finally, we note that
  $$
  \aligned
  &\inf_{\substack{M\in \R^{m\times d}\\ q\in \R^m}}
\frac{1}{t}
\left(\fint_{B(0,t)} | v-Mx-q|^2\right)^{1/2}
\le C\, t \left(\fint_{B(0,t)} |\nabla^2 v |^2\right)^{1/2}\\
& \le C\,  t \left(\fint_{B(0,t)} |\nabla^2 w |^2\right)^{1/2}
+ C \, t \left(\fint_{B(0,t)} |\nabla^2  I |^2\right)^{1/2}\\
&\le C\, t \left(\fint_{B(0,1)} |v|^2\right)^{1/2}
  +
  C \log \left(\frac{1}{t} \right)
\sup_{x\in B(0, \frac{1}{2})}
  t \left(\fint_{B(x,t)} |F|^2\right)^{1/2}.
\endaligned
$$
Since $\mathcal{L}_0 (Mx +q)=0$,
we may replace $v$ in the inequalities above by $v-Mx -q$
for any $M\in \R^{m\times d}$ and $q\in \R^m$.
This gives the estimate (\ref{11-3-0}).
\end{proof}

\begin{lem}\label{lemma-11-4}
Fix $\sigma \in (0,1)$.
Let $u_\e\in H^1(B(0,1); \R^m)$ be a weak solution of $\mathcal{L}_\e (u_\e)=F$ in $B(0,1)$,
where $0<\e<1$ and $F\in L^2(B(0,1); \R^m)$.
Define
\begin{equation}\label{H}
H(r)=\frac{1}{r} \inf_{\substack{M\in \R^{m\times d}\\ q\in \R^m}}
\left(\fint_{B(0,r)} | u_\e-Mx-q|^2\right)^{1/2}
+\sup_{\substack{x\in B(0,1/2)\\ \e \le t\le r/2}} t \left(\frac{r}{t}\right)^\sigma
\left(\fint_{B(x, t)} |F|^2\right)^{1/2},
\end{equation}
and
\begin{equation}\label{Psi}
\Psi (r)=\inf_{q\in \R^m}\frac{1}{r}
\left(\fint_{B(0,2r)} |u_\e -q|^2\right)^{1/2}
+r \left(\fint_{B(0, 2r))} |F|^2\right)^{1/2}.
\end{equation}
Then there exists $\theta\in (0,1/4)$, depending only on $d$, $m$, $\sigma$ and $\mu$, such that
\begin{equation}\label{11-4-0}
H(\theta r) \le  (1/2) H(r) + C\eta (\e/r) \Psi (r)
\end{equation}
for any $r\in [\theta^{-1} \e, 1/2]$.
\end{lem}

\begin{proof}
Let $r\in [\theta^{-1}\e, 1/2]$, where $\theta\in (0,1/4)$ is to be determined.
Let $v$ be the solution of $\mathcal{L}_0 (v)=F$ in $B(0,r)$, given by Lemma \ref{lemma-11-2}.
Observe that by Lemmas \ref{lemma-11-3} and \ref{lemma-11-2},
$$
\aligned
H(\theta r)
 &\le \inf_{\substack{M\in \R^{m\times d}\\ q\in \R^m}}
\frac{1}{\theta r} \left(\fint_{B(0,\theta r)} | v-Mx-q|^2\right)^{1/2}
+\sup_{\substack{x\in B(0,1/2)\\ \e\le t\le \theta r} } 
t \left(\frac{\theta r}{t} \right)^\sigma  \left(\fint_{B(x, t)} |F|^2\right)^{1/2}\\
 &\qquad \qquad
 + \frac{1}{\theta r} \left(\fint_{B(0, \theta r)} |u_\e -v|^2\right)^{1/2}\\
& \le C \theta  \inf_{\substack{M\in \R^{m\times d}\\ q\in \R^m}}
\left(\fint_{B(0,r)} | v -Mx-q|^2\right)^{1/2}
 +\theta^\sigma \sup_{\substack{x\in B(0,1/2)\\ \e\le t\le r/2}} t \left(\frac{r}{t}\right)^\sigma
  \left(\fint_{B(x, t)} |F|^2\right)^{1/2}\\
& \qquad
+C\theta  \log (\theta^{-1})
\sup_{x\in B(0,r/2)} r\left(\fint_{B(x,\theta r)} |F|^2\right)^{1/2}
+\frac{C_\theta}{r} \left(\fint_{B(0,r)} |u_\e -v|^2\right)^{1/2}\\
&\le C \left\{ \theta +\theta^\sigma +\theta^\sigma  \log (\theta^{-1})\right\}
H(r)
+\frac{C_\theta}{r} \left(\fint_{B(0,r)} |u_\e -v|^2\right)^{1/2}\\
&\le C \left\{ \theta + \theta^\sigma +\theta^\sigma  \log (\theta^{-1})\right\}
H(r)\\
&\qquad
+\frac{C_\theta}{r} \eta (\e/r)
\left\{  \left(\fint_{B(0,2r)} |u_\e |^2\right)^{1/2}
+ r^2 \left(\fint_{B(0,2r)} |F|^2\right)^{1/2} \right\}.
\endaligned
$$
We now choose $\theta\in (0,1/4)$ so small that $ C\left\{ \theta +\theta^\sigma +\theta^\sigma \log (\theta^{-1})\right\}
\le (1/2)$.
Since the inequalities above also hold for $u_\e-q$ with any $q\in \R^m$,
we obtain  the estimate (\ref{11-4-0}).
\end{proof}

The next lemma was proved in \cite{Shen-boundary-2015}.

\begin{lem}\label{G-lemma-1}
Let $H(r)$ and $h(r)$ be two nonnegative  continuous functions on the interval $(0, 1]$. Let
$0<\varep<(1/4)$.
Suppose that there exists a constant $C_0$  such that
\begin{equation}\label{G-1}
\left\{
\aligned
& \max_{r\le t \le 2r} H(t) \le C_0\,  H(2r),\\
& \max_{r\le t,s\le 2r} |h(t) -h(s)|   \le C_0\,  H(2r),
\endaligned
\right.
\end{equation}
for any $r\in [ \varep, 1/2]$. We further assume that
\begin{equation}\label{G-3}
H(\theta r) \le (1/2) H(r) + C_0\,  \omega (\varep/r) \Big\{ H(2r) + h(2r) \Big\},
\end{equation}
for any $r\in [ \theta^{-1} \varep, 1/2]$,
where $\theta\in (0,1/4)$ and $\omega$ is a nonnegative increasing function $[0,1]$ such that
$\omega(0)=0$ and
\begin{equation}\label{G-4}
\int_0^1 \frac{\omega(t)}{t}\, dt  <\infty.
\end{equation}
Then
\begin{equation}\label{G-5}
\max_{\varep\le r\le 1}
\Big\{ H(r) +h(r) \Big\}
\le C \Big\{ H(1) +h (1) \Big\},
\end{equation}
where $C$ depends only on $C_0$, $\theta$,  and $\omega$.
\end{lem}

We are now in a position to give the proof of Theorem \ref{Lip-theorem}.

\begin{proof}[\bf Proof of Theorem \ref{Lip-theorem}]
By translation and dilation we may assume that $x_0=0$ and $R=1$.
Thus $u_\e$ is a weak solution of $\mathcal{L}_\e (u_\e)=F$ in $B(0, 1)$
and $0<\e<(1/2)$.
Let $H(r)$ be defined by (\ref{H}).
Let $h(r)=|M_r|$, where $M_r\in \R^{m\times d}$ is a matrix such that 
$$
\inf_{q\in \R^m}\frac{1}{r}
\left(\fint_{B(0, r)} |u_\e -M_r x -q|^2\right)^{1/2}
=\inf_{\substack{M\in \R^{m\times d}\\ q\in \R^m}}
\frac{1}{r}
\left(\fint_{B(0, r)} |u_\e -M x -q|^2\right)^{1/2}.
$$
As in \cite{Shen-boundary-2015}, it follows that if $t,s\in [r,2r]$,
$$
|h (t)- h(s)|\le |M_t -M_s|\le C \Big\{ H(t) +H(s)\Big\} \le C H(2r).
$$
Also, if $\Psi (r)$ is defined by (\ref{Psi}),
then
$$
\Psi(r)\le H(2r) +h(2r).
$$
By Lemma \ref{lemma-11-4} we see that
$$
H(\theta r)\le (1/2) H(r) +C \eta (\e/r) \Big\{ H(2r) +h (2r)\Big\},
$$
for any $r\in [\theta^{-1} \e, 1/2]$, where $\eta(t)$ satisfies the Dini condition (\ref{Dini}).
This allows us to apply Lemma \ref{G-lemma-1}  and obtain 
$$
\aligned
 &\inf_{q\in \R^m} \frac{1}{r} \left(\fint_{B(0,r)} |u_\e -q|^2\right)^{1/2}
\le C \Big\{ H(r) + h(r) \Big\}
\le C \Big\{ H(1) + h(1) \Big\}\\
&\qquad \le C \left\{ \left(\fint_{B(0,1)}| \nabla u_\e|^2\right)^{1/2}
+\sup_{\substack{x\in B(0, 1/2)\\ \e\le t\le 1/2} }  t^{1-\sigma}
\left(\fint_{B(x, t)} |F|^2\right)^{1/2} \right\}.
\endaligned
$$
Hence, by Caccioppoli's inequality, 
$$
\left(\fint_{B(0,r)} |\nabla u_\e|^2\right)^{1/2}
\le C \left\{ \left(\fint_{B(0,1)}| \nabla  u_\e|^2\right)^{1/2}
+\sup_{\substack{x\in B(0, 1/2)\\ \e\le t\le 1/2}} t^{1-\sigma}
\left(\fint_{B(x, t)} |F|^2\right)^{1/2} \right\},
$$
for any $r\in [\e, 1/2]$.
The proof is complete.
\end{proof}

Finally, we give the proof of Theorem \ref{main-theorem-Lip}

\begin{proof}[\bf Proof of Theorem \ref{main-theorem-Lip}]

Note that by (\ref{decay-Lip}),
$$
\inf_{1\le L\le t}
\left\{ \rho_k (L, L) + \exp \left(-\frac{c\, t^2}{L^2}\right) \right\}
\le C \left\{ \big\{ \log (t+1) \big\}^{-\alpha}
+ \exp (-ct) \right\}.
$$
Using
$$
\int_2^T \frac{dt}{(\log t)^\alpha t^\sigma }
\le C_{\sigma, \alpha} \, T^{1-\sigma} (\log T )^{-\alpha}
$$
for  any $T\ge 2$, where $\sigma \in (0,1)$ and $\alpha>0$,
we see that
$$
\frac{1}{T}
\int_1^T \inf_{1\le L\le t}
\left\{ \rho_k (L, L) + \exp \left(-\frac{c\, t^2}{L^2}\right) \right\} \left(\frac{T}{t}\right)^\sigma\, dt
\le C \big\{ \log (T+1) \big\}^{-\alpha}.
$$
In view of Theorem \ref{main-theorem-1}, this yields
$$
T^{-1} \|\chi_T\|_{S^2_1} \le C \big\{ \log (T+1) \big\}^{-\alpha},
$$
and
$$
\| \nabla \chi_T -\psi\|_{B^2} \le   C\sum_{j=1}^\infty
(2^j T)^{-1} \|\chi_{2^{j}T} \|_{S^2_1}
\le  C\big\{ \log (T+1) \big\}^{1-\alpha}.
$$
The same argument also gives the estimate for the adjoint operator,
$$
\| \nabla \chi^*_T -\psi^*\|_{B^2} \le  C \big\{ \log (T+1) \big\}^{1-\alpha}.
$$
It then follows from Theorem \ref{main-theorem-3} that the condition (\ref{Lip-condition-11})
holds for
$$
\eta (t)=  C \big\{ \log (2/t)\big\}^{\frac{1-\alpha}{2}}.
$$
Since $\alpha >3$, the function $\eta (t)$ satisfies the Dini condition (\ref{Dini}).
As a result, Theorem \ref{Lip-theorem} holds for the operator $\mathcal{L}_\e$.

Finally, let
$$
u(x)=\chi_{T, j}^\beta (x) + P_j^\beta (x).
$$
Then $\mathcal{L}_1 (u)=-T^{-2} \chi_{T, j}^\beta$ in $\R^d$.
It follows from Theorem \ref{Lip-theorem} with $r=\e=1$ and $R=T$ that
$$
\left(\fint_{B(x_0,1)} |\nabla u|^2\right)^{1/2}
\le C \left(\fint_{B(x_0, T)} |\nabla u|^2\right)^{1/2}
+ C T^{-1} \|\chi_T\|_{S^2_1}
\le C,
$$
for any $x_0\in \R^d$. 
Since $|\nabla \chi_{T,j}^\beta |\le |\nabla u| +C$,
this gives the estimate (\ref{Lip}).
\end{proof}

%%%%%%%%%%%%%%%%%%%%%%%%%%%%%%%%%%%%%%%%%%%%%%

%%%%%%%%%%%%%%%%%%%%%%%%%%%%%%%%%%%%%%%%%%%%%%%

\bibliography{Shen-Zhuge.bbl}

\providecommand{\bysame}{\leavevmode\hbox to3em{\hrulefill}\thinspace}
\providecommand{\MR}{\relax\ifhmode\unskip\space\fi MR }
\providecommand{\MRhref}[2]{%
  \href{http://www.ams.org/mathscinet-getitem?mr=#1}{#2}
}
\providecommand{\href}[2]{#2}
\begin{thebibliography}{10}

\bibitem{AGK-2015}
S.N. Armstrong, A.~Gloria, and T.~Kuusi, \emph{Bounded correctors in almost
  periodic homogenization}, arXiv1509.08390 (2015).

\bibitem{Armstrong-Mourrat}
S.N. Armstrong and J.-C. Mourrat, \emph{Lipschitz regularity for elliptic
  equations with random coefficients}, Arch. Ration. Mech. Anal. (to appear).

\bibitem{Armstrong-Shen-2016}
S.N. Armstrong and Z.~Shen, \emph{Lipschitz estimates in almost-periodic
  homogenization}, Comm. Pure Appl. Math. (to appear).

\bibitem{Armstrong-Smart-2014}
S.N. Armstrong and C.K. Smart, \emph{Quantitative stochastic homogenization of
  convex integral functionals}, Ann. Sci. \'Ec. Norm. Sup\'er (to appear).

\bibitem{AL-1987}
M.~Avellaneda and F.~Lin, \emph{Compactness methods in the theory of
  homogenization}, Comm. Pure Appl. Math. \textbf{40} (1987), 803--847.

\bibitem{B}
A.S. Besicovitch, \emph{{Almost Periodic Functions}}, Dover Publications, Inc.,
  New York, 1955.

\bibitem{Bondarenko-2005}
A.~Bondarenko, G.~Bouchitte, L.~Mascarenhas, and R.~Mahadevan, \emph{Rate of
  convergence for correctors in almost periodic homogenization}, Discrete
  Continu. Dynamical Systems \textbf{13} (2005), 503--514.

\bibitem{Caffarelli-2010}
L.~Caffarelli and P.E. Souganidis, \emph{Rates of convergence for the
  homogenization of fully nonlinear uniformly elliptic pde in random media},
  Invent. Math. \textbf{180} (2010), 301--360.

\bibitem{Dungey-2001}
N.~Dungey, A.F.M. ter Elst, and D.W. Robinson, \emph{On second-order
  almost-periodic elliptic operators}, J. London Math. Soc. \textbf{63} (2001),
  735--753.

\bibitem{Giaquinta}
M~Giaquinta, \emph{Multiple {I}ntegrals in the {C}alculus of {V}ariations and
  {N}onlinear {E}lliptic {S}ystems}, Ann. of Math. Studies, vol. 105, Princeton
  Univ. Press, 1983.

\bibitem{Ishii-2000}
H.~Ishii, \emph{{Almost-periodic homogenization of Hamilton-Jacobi equations}},
  International Conference on Differential Equations, vols. 1, 2, Berlin, 1999,
  World Sci. Publ., River Edge, 2000, pp.~600--605.

\bibitem{Jikov-1994}
V.V. Jikov, S.M. Kozlov, and O.A. Oleinik, \emph{Homogenization of
  {D}ifferential {O}perators and {I}ntegral {F}unctionals}, Springer-Verlag,
  Berlin, 1994.

\bibitem{Kozlov-1979}
S.M. Kozlov, \emph{Averaging differential operators with almost periodic,
  rapidly oscillating coefficients}, Math. USSR Sbornik \textbf{35} (1979),
  481--498.

\bibitem{Lions-2005}
P.-L. Lions and P.E. Souganidis, \emph{Homogenization of degenerate
  second-order pde in periodic and almost periodic environments and
  applications}, Ann. Inst. H. Poincar\'e, Anal. Non Lineaire \textbf{22}
  (2005), 667--677.

\bibitem{Papanicolaou-1979}
G.~Papanicolaou and S.R.S. Varadhan, \emph{Boundary value problems with rapidly
  oscillating random coefficients}, Proceed. Colloq. on Random Fields, Rigorous
  Results in Statistical Mechanics and Quantum Field Theory, Colloquia
  Mathematica Soc. Janos Bolyai, vol.~10, 1979, pp.~835--8873.

\bibitem{Shen-boundary-2015}
Z.~Shen, \emph{Boundary estimates in elliptic homogenization}, arXiv:1505.02525
  (2015).

\bibitem{Shen-2015}
\bysame, \emph{Convergence rates and {H}\"older estimates in almost-periodic
  homogenization of elliptic systems}, Analysis \& PDE \textbf{8} (2015),
  no.~7, 1565--1601.

\bibitem{SZ-2015}
Z.~Shen and J.~Zhuge, \emph{Convergence rates in periodic homogenization of
  systems of elasticity with mixed boundary conditions}, arXiv:1512.00823
  (2015).

\bibitem{Suslina-2012}
T.A. Suslina, \emph{{Homogenization of the elliptic Dirichlet problem: operator
  error estimates in $L_2$}}, Mathematika \textbf{59} (2013), no.~2, 463--476.

\end{thebibliography}

\medskip

\begin{flushleft}
Zhongwei Shen

Department of Mathematics
 
University of Kentucky

Lexington, Kentucky 40506, USA. 

% Fax: 1-859-257-4078.

E-mail: zshen2@uky.edu

\medskip

Jinping Zhuge

Department of Mathematics
 
University of Kentucky

Lexington, Kentucky 40506, USA. 

E-mail: jinping.zhuge@uky.edu

\end{flushleft}

\medskip

\noindent \today

\end{document}